\documentclass{amsart}
\usepackage{graphicx,verbatim, amsmath, amssymb, amsthm, amsfonts, epsfig, psfrag}	

\copyrightinfo{2010}%    % copyright year
  {Nathan Reading and David E Speyer}% copyright holder

\newtheorem{proposition}{Proposition}[section]
\newtheorem{theorem}[proposition]{Theorem}

\newtheorem{lemma}[proposition]{Lemma}
\newtheorem{prop}[proposition]{Proposition}
\newtheorem{cor}[proposition]{Corollary}

\theoremstyle{definition}
\newtheorem{example}[proposition]{Example}

\theoremstyle{remark}
\newtheorem{remark}[proposition]{Remark}
\newtheorem{problem}[proposition]{Problem}

\newcommand{\newword}[1]{\textbf{\emph{#1}}}

\newcommand{\nc}{\operatorname{nc}}

\newcommand{\Fix}{\operatorname{Fix}}
\newcommand{\Stab}{\operatorname{Stab}}

\newcommand{\cov}{\mathrm{cov}}
\newcommand{\fs}{\mathrm{fs}}
\newcommand{\ufs}{\mathrm{ufs}}
\newcommand{\covers}{{\,\,\,\cdot\!\!\!\! >\,\,}}
\newcommand{\covered}{{\,\,<\!\!\!\!\cdot\,\,\,}}
\newcommand{\set}[1]{{\left\lbrace #1 \right\rbrace}}
\newcommand{\pidown}{\pi_\downarrow}

\newcommand{\br}[1]{\langle #1 \rangle}
\newcommand{\A}{{\mathcal A}}

\newcommand{\F}{{\mathcal F}}
\newcommand{\X}{{\mathcal X}}
\newcommand{\join}{\vee}
\newcommand{\meet}{\wedge}
\renewcommand{\Join}{\bigvee}
\newcommand{\Meet}{\bigwedge}

\newcommand{\closeleftq}[2]{\!\!\phantom{.}^{#1}\! {#2}}

\newcommand{\ck}{^\vee}

\newcommand{\st}{^\mathrm{st}}
\renewcommand{\th}{^\mathrm{th}}

\newcommand{\0}{{\hat{0}}}

\newcommand{\Vol}{\mathrm{Vol}}
\newcommand{\lleq}{\le\!\!\!\le}
\newcommand{\notlleq}{\le\!\!\!\!\not\,\le}

\newcommand{\RR}{\mathbb{R}}

\newcommand{\Tits}{\mathrm{Tits}}
\newcommand{\Cone}{\mathrm{Cone}}
\newcommand{\Star}{\mathrm{Star}}
\newcommand{\Lin}{\mathrm{Lin}}

\newcommand{\Cox}{\mathrm{Cox}}

\DeclareMathOperator{\supp}{supp}
\DeclareMathOperator{\inv}{inv}

\newcommand{\odd}{\mathrm{odd}}

\title{Sortable elements in infinite Coxeter groups}

\author{Nathan Reading}
\address{Department of Mathematics, North Carolina State University, Raleigh, NC 27695}
\curraddr{}
\email{nathan\_reading@ncsu.edu}
\thanks{}

\author{David E Speyer}
\address{Department of Mathematics, Massachusetts Institute of Technology, Cambridge, MA 02139}
\curraddr{}
\email{speyer@math.mit.edu}
\thanks{David E Speyer was supported by a Research Fellowship from the Clay Mathematics Institute.
}

\begin{document}

%    \subjclass is required.
\subjclass[2000]{Primary 20F55}

\begin{abstract}
In a series of previous papers, we studied sortable elements in finite Coxeter groups, and the related Cambrian fans. We applied sortable elements and Cambrian fans to the study of cluster algebras of finite type and the noncrossing partitions associated to Artin groups of finite type. In this paper, as the first step towards expanding these applications beyond finite type, we study sortable elements in a general Coxeter group $W$.
We supply uniform arguments which transform all previous finite-type proofs into uniform proofs (rather than type by type proofs), generalize many of the finite-type results and prove new and more refined results.  
The key tools in our proofs include a skew-symmetric form related to (a generalization of) the Euler form of quiver theory and the projection $\pidown^c$ mapping each element of $W$ to the unique maximal $c$-sortable element below it in the weak order. 
The fibers of $\pidown^c$ essentially define the $c$-Cambrian fan.  
The most fundamental results are, first, a precise statement of how sortable elements transform under (BGP) reflection functors and second, a precise description of the fibers of $\pidown^c$.
These fundamental results and others lead to further results on the lattice theory and geometry of Cambrian (semi)lattices and Cambrian fans.
\end{abstract}
\maketitle

\setcounter{tocdepth}{1}
\tableofcontents

\section{Introduction}

This paper is the continuation of a series of papers \cite{cambrian,sortable,sort_camb,camb_fan} devoted to the study of sortable elements and Cambrian lattices in a Coxeter group $W$.
The primary motivations behind this work are the connections to cluster algebras (see~\cite{ca4} and~\cite{camb_fan}) and to the study of Artin groups via noncrossing partitions (see e.g.\ \cite{Bessis,BWKpi}).
In the first three of these papers, specific connections were conjectured between sortable elements/Cambrian lattices and generalized associahedra (see~\cite{ga}) and thus between sortable elements and cluster algebras of finite type (see~\cite{ca2}).
Two bijections were established:
one from sortable elements to vertices of the (simple) generalized associahedron and another from sortable elements to noncrossing partitions.
In~\cite{camb_fan} the former bijection was shown to induce a combinatorial isomorphism between the Cambrian fan and the normal fan of the generalized associahedron, leading to further results connecting the combinatorics of cluster algebras to the geometry of Cambrian fans. 
This last paper proved all of the remaining earlier conjectures concerning relations between sortable elements and finite type cluster algebras. 

The results of the earlier papers were confined to the very special case where the Coxeter group $W$ is finite.
In this paper, we extend the theory to the case of arbitrary $W$.
Although the infinite case is a vastly larger arena where significantly fewer tools are available, we not only generalize the basic results on sortable elements, but also greatly improve them.
Among the improvements is an explicit pattern avoidance characterization of $c$-sortable elements.
In~\cite{sortable}, the first author gave a pattern avoidance condition based on a family of ``$c$-orientations,'' which were constructed explicitly type by type.
We now construct the orientations uniformly, via the restriction of a skew symmetric bilinear form $\omega_c$ to parabolic subgroups.
Another improvement is in our understanding of the Cambrian congruence, a certain equivalence relation on the Coxeter group $W$.
In~\cite{sort_camb}, the Cambrian congruence was constructed in two ways:  as the lattice congruence generated by a small explicit set of equivalences, and as the fibers of a recursively defined map $\pidown^c$.
In neither case, given a representative element of a Cambrian congruence class, was it clear how to describe all elements in that equivalence class.
In Theorem~\ref{pidown fibers}, we give a description of any Cambrian congruence class by a finite list of immediately checkable conditions.

The description of the equivalence classes is intimately related to a third significant improvement to the theory:
Using the combinatorics of a special ``sorting word'' for each sortable element, we construct a simplicial cone for each sortable element.
In finite type, the collection of all such cones is the Cambrian fan, which was shown in~\cite{camb_fan} to be isomorphic to the cluster complex of the corresponding cluster algebra of finite type.
We do not yet prove that this collection of cones is a fan in infinite case.
However, we do prove that intersecting each cone with the Tits cone produces a collection of cones whose intersections are well-behaved.
The combinatorics of sorting words gives very precise combinatorial control over the Cambrian cones.
This precise control will be used in a future paper to generalize and extend the results of~\cite{camb_fan} connecting the Cambrian fan to cluster algebras.
(See Section~\ref{future}.)

Although this paper improves the general theory of sortable elements, it by no means supersedes the previous papers.
The main point of each paper, and the bulk of the arguments, are not replaced by the arguments of the present paper.
(The main points of \cite{cambrian,sortable,sort_camb,camb_fan} are respectively:  a lattice-theoretic and fiber-polytopal approach to generalized associahedra; using sortable elements to make a combinatorial connection between generalized associahedra and noncrossing partitions; using sortable elements to explicitly construct Cambrian lattices; and making the connection between sortable elements/Cambrian fans and the combinatorics of cluster algebras of finite type.)
The current paper does greatly improve the previous papers by providing uniform proofs for a small number of lemmas in~\cite{sortable} which were previously proved using type by type checks. 
These lemmas occupy very little space in~\cite{sortable}, but all of the main results of \cite{sortable,sort_camb,camb_fan} ultimately depend on the lemmas.
With the addition of this paper, these lemmas, and thus all of the results of the previous papers, have uniform proofs.
(See Section~\ref{Past} for details.)

%While we do not claim to have found the optimal proofs, we hope that we have discovered the vocabulary in which the 
%optimal proofs will be written. 

% Furthermore, we introduce the tools which we believe will underlie the future study of sortable elements and Cambrian lattices. 
% In particular, we introduce a combinatorial idea which we think has broader promise: using a skew-symmetric form on the 
% root space to provide an orientation on the rank-two parabolic subgroups, and considering pattern avoidance conditions with respect
% to this orientation. The second author has already used this skew-symmetric form~\cite{reduced} 
% to give a short proof that powers of Coxeter elements in irreducible infinite Coxeter groups are reduced.
% Using the skew-symmetric form and a generalization of the Euler form of quiver theory, we obtain a precise 
% characterization of the walls of cones in the Cambrian fan in terms of the combinatorics of sortable elements.  

\subsection{Summary of results} \label{Summary}
We now summarize the results of the paper.
Most of the notation and definitions surrounding Coxeter groups are put off until Section~\ref{CoxeterConventions}.

An element $c \in W$ is called a \newword{Coxeter element} if it is of the form $s_1 s_2 \cdots s_n$ for some ordering of the set $S:=\{ s_1, s_2, \ldots, s_n \}$ of simple reflections. 
For any word~$z$ in the alphabet $S$, the \newword{support} $\supp(z)\subseteq S$ of $z$ is the set of letters occurring in~$z$.
An element $w \in W$ is called $c$-sortable if~$w$ has a reduced word which is the concatenation $z_1z_2\cdots z_k$ of words $z_i$ subject to the following conditions:
\begin{enumerate}
\item Each $z_i$ is a subword of $s_1\cdots s_n$.
\item The supports of the $z_i$'s satisfy $\supp(z_1)\supseteq \supp(z_2)\supseteq\cdots\supseteq\supp(z_k)$.
\end{enumerate}
This word is called a $c$-sorting word for $w$.
The definition of $c$-sortability depends on the choice of~$c$ but not on the choice of reduced word for $c$.
As described in Proposition~\ref{CSortRecursive}, $c$-sortability can also be characterized recursively by induction on the length of $w$ and on the rank of $W$.

Before summarizing our results in the order they appear, we wish to emphasize two fundamental results on sortable elements.
These results appear late in the paper and use most of the tools developed over the course of the paper.
Yet if these two results could be proven first, most of the other results on sortable elements would follow.

The first result that we wish to emphasize is the following:
For any $w\in W$ there exists a unique maximal $c$-sortable element below $w$. 
The map $\pidown^c$ mapping $w$ to this maximal element is of fundamental importance.
In particular, the fibers of $\pidown^c$ essentially define the $c$-Cambrian fan.
More precisely, the maximal cones in the $c$-Cambrian fan, intersected with the Tits cone, are the fibers of $\pidown^c$.
Proving the existence of a unique maximal $c$-sortable element below each $w$ requires all the tools developed in the first part of the paper.
We therefore begin our treatment of $\pidown^c$ with a recursive definition of $\pidown^c(w)$.
We only later establish, as Corollary~\ref{max}, that $\pidown^c(w)$ can be described as it is in this paragraph.

The second result describes what happens when we change Coxeter elements.
We say $s\in S$ is initial in a Coxeter element $c$ if $c$ has a reduced word starting with~$s$.
In this case $scs$ is another Coxeter element, obtained by removing $s$ from the beginning of the word and placing it at the end. 
This operation on $c$ corresponds to applying a (BGP) reflection functor in the sense of quiver theory. See \cite[Section~VIII.5]{ASS} or~\cite{BGP}.
If~$s$ is initial in~$c$ then the map
\[v\mapsto\left\lbrace\begin{array}{cl}
sv&\mbox{if }v\ge s\\
s\join v&\mbox{if }v\not\ge s
\end{array}\right.\]
is a bijection from the set $\set{\mbox{$c$-sortable elements $v$ such that $s\join v$ exists}}$ to the set of all $scs$-sortable elements. (Cf. \cite[Remark 3.8]{sort_camb} and the discussion preceding \cite[Lemma~7.5]{camb_fan}.)
This result, with an explicit inverse map, appears as Theorem~\ref{c to scs}.

In finite type this bijection implies that the number of $c$-sortable elements is independent of $c$, using the fact that the weak order is a lattice in finite type and the following (essentially graph-theoretic) result:
If $c$ and $c'$ are Coxeter elements in a finite Coxeter group~$W,$  then $c$ and $c'$ are related by a sequence of moves which remove an initial letter and append it at the end.

We now describe in order the main results of the paper. 
In Section~\ref{omega sec} we introduce a technical device known as the Euler form. 
This is a certain non-symmetric bilinear form, denoted $E_c$, which is inspired by constructions in quiver theory.  The symmetrization of $E_c$ is the ordinary inner product; the anti-symmetrization $\omega_c$ of $E_c$ is a skew-symmetric bilinear form. 
The main result of Section~\ref{omega sec} is Proposition~\ref{InversionOrdering}, which states that $w\in W$ is $c$-sortable if and only if $w$ has a reduced word whose associated sequence of reflections is compatible with $\omega_c$, in a sense we describe in that section.
In Section~\ref{align sec} we characterize $c$-sortability of $w$ directly in terms of the set of inversions of $w$, without requiring a particular reduced word. This description identifies the property of $c$-sortability as a condition of pattern avoidance in the sense of Billey and Braden~\cite{BillBrad} and extends the first author's results in~\cite[Sections~3--4]{sortable}.

We next define, in Section~\ref{skip sec}, a collection $C_c(v)$ of $n$ roots for each $c$-sortable element $v$. 
(Here $n$ is the rank of $W$.) 
This set of roots defines a dual cone $\Cone_c(v)$, the \newword{$c$-Cambrian cone} associated to $v$.  
Using Proposition~\ref{InversionOrdering} and the Euler form, we then characterize $C_c(v)$ in terms of the combinatorics of $c$-sorting words.
In Section~\ref{pidown sec} we give the recursive definition of $\pidown^c$ and use a complicated inductive argument in Section~\ref{pidown sec} to prove two major theorems (Theorems~\ref{order preserving} and~\ref{pidown fibers}).
The first of these states that $\pidown^c$ is order-preserving and the second states that each fiber $(\pidown^c)^{-1}(v)$ is the intersection of the Tits cone with $\Cone_c(v)$. 
Given the combinatorial characterization of $C_c(v)$, the latter assertion amounts to an extremely explicit description of the fibers of $\pidown^c$. 

At this point, the paper changes tone. 
Until now, the focus has been on technical results about sortable elements, proved via very detailed lemmas. 
The attention now shifts to the development of the lattice theory and polyhedral geometry of sortable elements. 
Each of our two primary objects, the $c$-Cambrian semilattice and the $c$-Cambrian fan, can be viewed as a way of organizing the set of fibers of $\pidown^c$.
The proofs are often short consequences of general principles and technical lemmas proved previously. 
We feel that the sortable elements are the girders of our theory, underlying all of our results, while the Cambrian semilattice and Cambrian fan are the doorways, through which connections to other mathematical fields will be made. 

The \newword{$c$-Cambrian semilattice} is defined as the weak order restricted to $c$-sortable elements.
In Section~\ref{lattice sec} we show that $c$-sortable elements are a sub-semilatice of the weak order (Theorem~\ref{meets and joins}).
Furthermore, Theorem~\ref{semilattice hom} states that the map $\pidown^c$ respects the meet operation and the (partially defined) join operation.  
In particular the $c$-Cambrian semilattice is isomorphic to the partial order induced by the weak order on the fibers of $\pidown^c$.
Theorem~\ref{meets and joins} leads to the proof of Theorem~\ref{c to scs} which establishes the bijection, described above, from (almost all of) the $c$-sortable elements to the $scs$-sortable elements.  
In Section~\ref{join rep sec} we discuss canonical join representations of sortable elements.
This discussion leads to a bijection between $c$-sortable join-irreducible elements and reflections in $W$ which can appear as inversions of $c$-sortable elements.
This bijection leads to a surprising lattice-theoretic interpretation of noncrossing partitions. As we explain in Section~\ref{join rep sec}, $c$-noncrossing partitions encode the canonical join representations of $c$-sortable elements. 

In Section~\ref{fan sec} we show that the Cambrian cones $\Cone_c(v)$ are the maximal cones of a fan inside the Tits cone, meaning that the intersections of Cambrian cones are well-behaved inside the Tits cone.
We call this fan the \newword{$c$-Cambrian fan}.
We conjecture a stronger statement that the Cambrian cones are the maximal cones of a fan, i.e.\ that they intersect nicely even outside of the Tits cone.
Section~\ref{rank3sec} consists of pictorial examples illustrating many of the results of the paper; we encourage the reader to consult Section~\ref{rank3sec} frequently while reading the prior sections. 

\subsection{Future directions}\label{future}
One of the primary motivations for studying sortable elements and Cambrian lattices in this and in earlier papers is the goal of using sortable elements to study the structure of cluster algebras.  We now briefly outline work in progress which will apply the results of this paper to cluster algebras.
We assume familiarity with cluster algebras for this discussion; to the reader looking for background we suggest the early sections of~\cite{ca4}. 

The initial data defining a cluster algebra (of geometric type) is an integer $m\times n$ matrix $\tilde{B}$ (with $m\ge n$) satisfying conditions given in \cite[Definition~2.12]{ca4}.
The most important part of $\tilde{B}$ is its top square submatrix $B$; the remaining entries of $\tilde{B}$ should be thought of as a choice of coefficients.
Specifying $B$ is equivalent to choosing a symmetrizable integer generalized Cartan matrix $A$ and choosing an orientation of the Coxeter graph~$\Gamma$ of the Coxeter group $W$ determined by $A$.
If the orientation of $\Gamma$ is acyclic then it defines a Coxeter element $c$ of $W$.
Different choices of $\tilde{B}$ can give isomorphic cluster algebras; the main result of~\cite{ca2} is that a cluster algebra is of finite type if and only if it can be specified by a matrix $\tilde{B}$ such that the associated Coxeter group $W$ is finite.
In this case the orientation of $\Gamma$ is necessarily acyclic and thus is equivalent to a choice of Coxeter element for $c$.  

A cluster algebra (of geometric type) is generated by a collection of multivariate rational functions, known as cluster variables, which are grouped into overlapping $n$-element sets called clusters. 
Given $\tilde{B}$ and the choice of an initial cluster, the remaining clusters are determined by local moves called \newword{mutations}.
The mutation moves also assign to each cluster its own $m\times n$ matrix.
Each cluster variable is assigned a vector known as the $g$-vector~\cite{ca4}. 
The \newword{$g$-vector fan} is a fan whose rays are the $g$-vectors.
A collection of rays spans a maximal cone of the fan if and only if the corresponding collection of cluster variables is a cluster. 

In~\cite{camb_fan} we proved two results concerning cluster algebras of finite type.  
The first result is that, for a special choice of $c$, the $c$-Cambrian fan coincides with the $g$-vector fan. 
Assuming~\cite[Conjecture~7.12]{ca4}, the fans coincide for all choices of $c$. 
In a future paper, we will use the results of this paper to remove the dependence on~\cite[Conjecture~7.12]{ca4} and to give a generalization to infinite groups; in the finite crystallographic case, this result also appears in~\cite{YZ}.

The second result asserts that, for each maximal cone in the $c$-Cambrian fan, the matrix of inner products of the normals to the walls of the cone is a so-called ``quasi-Cartan" companion for the corresponding cluster. This means that the geometry of the cone in the $c$-Cambrian fan determines the absolute values of the entries in the $B$-matrix of the cluster. 
In the future paper, we will also use the results of this paper to determine the signs of the entries of $B$ and to determine, in the case of a cluster algebra with principal coefficients (see \cite[Definition~3.1]{ca4}), the entries of $\tilde{B}$ which do not lie in the square submatrix $B$.

The broader goal of this line of research is to study cluster algebras without restriction to finite type.
The first step towards that broader goal is to extend the results of~\cite{camb_fan} beyond finite type.
Specifically:
\begin{problem}\label{inf problem}
For an arbitrary cluster algebra of geometric type, construct the $g$-vector fan and the $B$-matrices, in the context of sortable elements/Cambrian fans.
\end{problem}
In moving beyond finite type, there are two independent obstacles to overcome.  
First, while every cone of the Cambrian fan meets the Tits cone, examples show that there are many cones in the $g$-vector complex which do not meet the Tits cone at all. 
A future paper will overcome this obstacle in part:
We will show how to solve Problem~\ref{inf problem} in an intermediate case between finite type and the completely general case: the case where $A$ is of affine type.
The second obstacle to moving beyond finite type is that our methods, thus far, can only address acyclic orientations of the Coxeter diagram.  
An extension of the definition of sortable elements and Cambrian fans to the non-acyclic case will appear in a future paper.

\subsection{Uniform proofs of previous results} \label{Past}
This section is included for the benefit of those citing specific results from previous papers in the series \cite{sortable,sort_camb,camb_fan}. 
All of the nontrivial general results of these three papers depend on one or more of the following five results in~\cite{sortable} which are proved with a type-by-type check:  \cite[Propositions~3.1 and~3.2]{sortable} and \cite[Lemmas~4.6, 6.6 and~6.7]{sortable}.  
The proofs in \cite{sort_camb,camb_fan} are non-uniform only to the extent to which they quote results from~\cite{sortable}.

All five of these results are given uniform proofs in this paper, so that in particular all results of \cite{sortable,sort_camb,camb_fan} now have uniform proofs. 
Specifically, in light of Lemmas~\ref{OmegaInvariance} and~\ref{OmegaNegativity}, the forms $\omega_c$ (as $c$ varies) provide the family of ``orientations'' of $W$ whose existence is asserted in \cite[Proposition~3.1]{sortable}, and thus Lemma~\ref{OmegaRestriction} implies \cite[Proposition~3.2]{sortable}.
The function of \cite[Lemma~4.6]{sortable} is to provide one case in the inductive proof of \cite[Theorem~4.1]{sortable}, which is generalized by Theorem~\ref{sortable is aligned} in the present paper.
The proof of Theorem~\ref{sortable is aligned} has the same inductive structure as the proof of \cite[Theorem~4.1]{sortable} and, indeed the generalization of \cite[Lemma~4.6]{sortable} is proved in the course of the proof of Theorem~\ref{sortable is aligned}.
Finally, \cite[Lemmas~6.6 and~6.7]{sortable} are generalized by Propositions~\ref{Cc final} and~\ref{Cc initial}. 

We caution the reader that, while we now have a uniform proof of \cite[Theorem~1.4]{sort_camb}, the formulation of this result requires $W$ to be a lattice, and hence finite.
We have not found a correct generalization of \cite[Theorem~1.4]{sort_camb} to infinite $W$. 

\begin{samepage}
\addtocontents{toc}{\medskip \textbf{Part I: Fundamental results about sortable elements} \medskip}
\bigskip
\bigskip
\centerline{\Large{\sc{Part I: Fundamental results about sortable elements}}}

\section{Coxeter groups and sortable elements} \label{CoxeterConventions}

In this section we establish notation and conventions for Coxeter groups and root systems and then review the definition of sortable elements.
We assume familiarity with the most basic definitions and results on Coxeter groups.
(See, for example, \cite{Bj-Br,Bourbaki, Humphreys}.) 
All of the notations in this section are fairly standard. 
\end{samepage}

\subsection{Coxeter groups}
Throughout the paper,~$W$ will be a Coxeter group with simple generators $S$.
The identity element will be denoted by $e$.
The labels on the Coxeter diagram for~$W$ will be written $m(s,s')$.
We permit $m(s,s')=\infty$.
The \newword{rank} of~$W$ is the cardinality of $S$, which will be assumed to be finite and will be called $n$.
The \newword{length} of $w\in W$, written $\ell(w)$, is the number of letters in a \newword{reduced} (i.e.\ shortest possible) word for~$w$ in the alphabet $S$.
The \newword{support} of~$w$ is the set of generators occurring in a reduced word for~$w$.
This is independent of the choice of reduced word.

A \newword{reflection} in~$W$ is an element of~$W$ which is conjugate to some element of~$S$.
The symbol $T$ will stand for the set of reflections in~$W$.
An \newword{inversion} of~$w$ is a reflection $t$ such that $\ell(tw)<\ell(w)$.
The set of inversions of~$w$, written $\inv(w)$, uniquely determines~$w$.
Given any reduced word $a_1\cdots a_k$ for~$w$, the inversion set of~$w$ consists of $k$ distinct reflections.
These are of the form $t_i=a_1\cdots a_i\cdots a_1$ for $1\le i\le k$.
The sequence $t_1,\ldots,t_k$ is the \newword{reflection sequence} for the reduced word $a_1\cdots a_k$.
Since $w=t_kt_{k-1}\cdots t_1$, we have the following lemma.

\begin{lemma}\label{commutes}
Let $u,w\in W$ and suppose $u$ commutes with every inversion of~$w$.
Then $u$ commutes with~$w$. 
\end{lemma}

A \newword{cover reflection} of $w\in W$ is an inversion $t$ of~$w$ such that $tw=ws$ for some $s\in S$.
The name ``cover reflection'' refers to the fact that~$w$ covers $tw$ in the weak order (defined below).
The set of cover reflections of~$w$ is written $\cov(w)$.
If $t$ is a cover reflection of~$w$ then $\inv(tw)=\inv(w)\setminus\set{t}$.

Suppose $s\in S$ and $w\in W$ such that $\ell(sw)<\ell(w)$.
Then the inversion sets of~$w$ and $sw$ are related as follows:
$\inv(w)=\set{s}\cup\set{sts:t\in\inv(sw)}$.
If $s$ is a cover reflection of~$w$ then $\inv(sw)=\inv(w)\setminus\set{s}$, so that $\inv(sw)=\set{sts:t\in\inv(sw)}$.

\subsection{Root systems} \label{Root subsec}
We now fix a reflection representation for~$W$ in the standard way by choosing a root system.  
A \newword{generalized Cartan matrix} for a Coxeter group~$W$ is a real matrix~$A$ with rows and columns indexed by $S$ such that:
\begin{enumerate}
\item[(i) ]$a_{ss}=2$ for every $s\in S$;
\item[(ii) ]$a_{ss'}\le 0$ with $\displaystyle a_{ss'}a_{s's}=4\cos^2\left(\frac{\pi}{m(s,s')}\right)$ when $s\neq s'$ and $m(s,s')<\infty$, and $\displaystyle a_{ss'}a_{s's} \ge 4$ if $m(s,s')=\infty$; and 
\item[(iii) ]$a_{ss'}=0$ if and only if $a_{s's}=0$.
\end{enumerate}
Most references, such as~\cite[Chapter~1]{Kac}, require that $A$ have integer entries. We permit $A$ to have real entries and say that $A$ is \newword{crystallographic} if it has integer entries.

Let $V$ be a real vector space of dimension $n$ with a basis $\alpha_s$ labeled by the simple reflections and let $V^*$ be the dual vector space; the set $\Pi:=\{ \alpha_s: s \in S \}$ is called the set of \newword{simple roots}. 
The canonical pairing between $x^*\in V^*$ and $y\in V$ will be written $\br{x^*,y}$.
The vector space $V$ carries a reflection representation of~$W$ such that $s\in S$ acts on the basis element $\alpha_{s'}$ by $s(\alpha_{s'})=\alpha_{s'}-a_{ss'} \alpha_s$. 
When necessary for clarity, we will denote $V$ and $V^*$ by $V(W)$ and $V^*(W)$. 
The action of $w\in W$ on a vector $x\in V$ will be denoted $wx$ and the same notation will be used for the dual action on $V^*$.
(This is the natural action on the dual, which is defined whether or not we have any way of identifying $V$ and $V^*$.)

We will restrict our attention to \newword{symmetrizable generalized Cartan matrices}. We say that $A$ is symmetrizable if there exists a positive real-valued function $\delta$ on $S$ such that $\delta(s) a_{s s'}=\delta(s') a_{s' s}$ and, if $s$ and $s'$ are conjugate, then $\delta(s)=\delta(s')$. (This second condition is non-standard; we will discuss it further in Section~\ref{RootConjugacy}.) We set $\alpha_s^{\vee}= \delta(s)^{-1} \alpha_s$; the $\alpha^{\vee}$ are called the simple co-roots.\footnote{The standard way of proceeding, described in~\cite[Chapters~1-3]{Kac}, is more sophisticated than this: one first constructs a pair of dual vector spaces $\mathfrak{h}$ and $\mathfrak{h^*}$, containing the roots and co-roots respectively, and later builds an isomorphism $\nu$ between $\mathfrak{h}$ and $\mathfrak{h}^*$. We start off by constructing the roots and co-roots in the same space to begin with, and only look at the subspace $V$ of $\mathfrak{h}$ which is spanned by the roots.}
Define a bilinear form $K$ on $V$ by $K(\alpha^{\vee}_{s}, \alpha_{s'})=a_{s s'}$. Note that $K(\alpha_s, \alpha_{s'})= \delta(s) K(\alpha^{\vee}_s, \alpha_{s'}) = \delta(s) a_{s s'} = \delta(s') a_{s' s} = K(\alpha_{s'}, \alpha_s)$ so $K$ is symmetric.  
One easily verifies that $W$ acts on the basis of simple co-roots by $s(\alpha\ck_{s'})=\alpha\ck_{s'}-a_{s's} \alpha\ck_s$.
%\[s(\alpha\ck_{s'})=\alpha\ck_{s'}-A(\alpha\ck_s,\alpha\ck_{s'}) \alpha_s=\alpha\ck_{s'}-A(\alpha\ck_{s'},\alpha\ck_s) \alpha_s=\alpha\ck_{s'}-A(\alpha\ck_{s'},\alpha_s) \alpha\ck_s=\alpha\ck_{s'}-a_{s's} \alpha\ck_s.\] 
The action of $W$ preserves $K$.

\begin{example} \label{B2 roots}
Consider the group $B_2$.
This is a Coxeter group of rank two, with generators $p$ and $q$, and $m(p,q)=4$.
We take our Cartan matrix to be $A = \left(\begin{smallmatrix} 2 & -2 \\ -1 & 2 \end{smallmatrix}\right)$. 
%We take our Cartan matrix to be $A ={\small\left[ \begin{array}{rr} 2 & -2 \\ -1 & 2 \end{array}\right]}$. 
So $A$ is crystallographic, and is symmetrizable with $\delta(p) = 1/2$ and $\delta(q)=2$.
On the left hand side of Figure~\ref{B2 root figure}, we depict the roots and co-roots of $B_2$
\end{example}

\begin{figure}[ht] 
\begin{tabular}{lccr}
\psfrag{aq}[cc][cc]{\Huge$\alpha_q=\alpha^{\vee}_q$}
\psfrag{bqpq}[cc][cc]{\Huge$\beta_{qpq}$}
\psfrag{bcqpq}[cc][cc]{\Huge$\beta^{\vee}_{qpq}$}
\psfrag{bpqp}[cc][cc]{\Huge$\beta_{pqp}=\beta^{\vee}_{pqp}$}
\psfrag{ap}[cc][cc]{\Huge$\alpha_p$}
\psfrag{acp}[cc][cc]{\Huge$\alpha^{\vee}_p$}
\scalebox{0.50}{\includegraphics{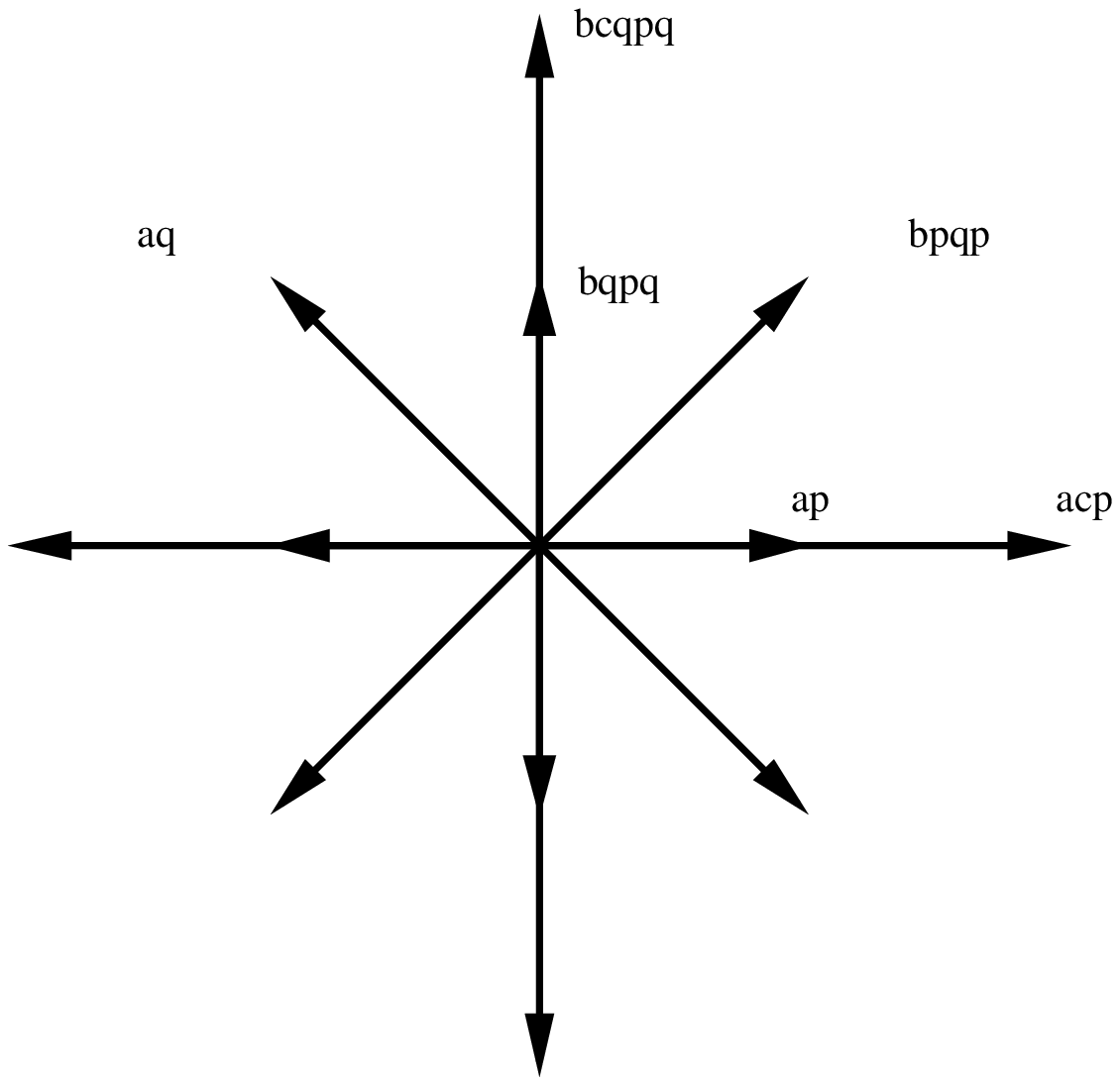}} &&&

\psfrag{e}[cc][cc]{\Huge$D$}
\psfrag{p}[cc][cc]{\Huge$p \cdot D$}
\psfrag{pq}[cc][cc]{\Huge$pq \cdot D$}
\psfrag{pqp}[cc][cc]{\Huge$pqp \cdot D$}
\psfrag{q}[cc][cc]{\Huge$q \cdot D$}
\psfrag{qp}[cc][cc]{\Huge$qp \cdot D$}
\psfrag{qpq}[cc][cc]{\Huge$qpq \cdot D$}
\psfrag{pqpq}[cc][cc]{\Huge$pqpq \cdot D$}
\scalebox{0.518}{\includegraphics{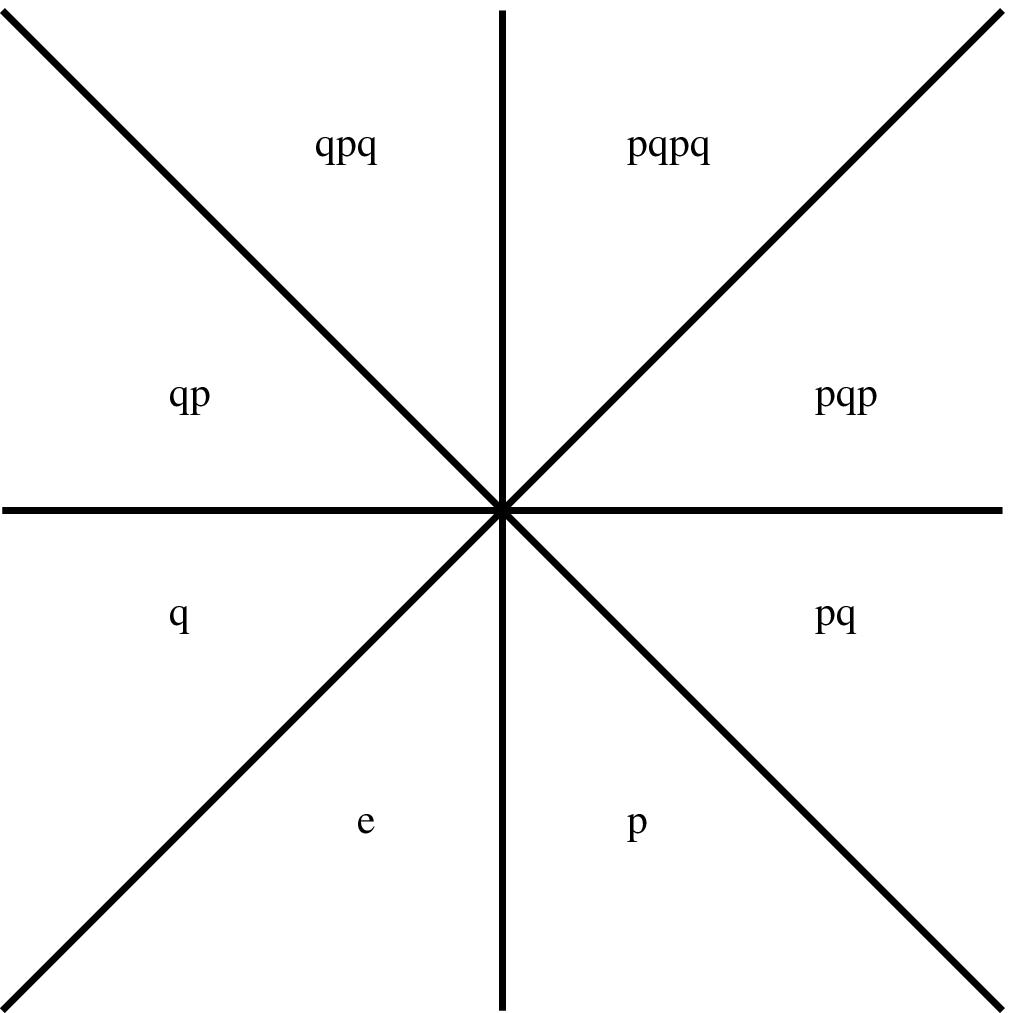}}
\end{tabular}
\caption{The $B_2$ root system and hyperplane arrangement.} \label{B2 root figure}
\end{figure}

We emphasize that we use symmetrizability only to prove combinatorial results which depend on the group $W$ but which are intrinsically independent of the choice of $A$.
In Section~\ref{fan sec} we describe the $c$-Cambrian fan, whose geometry (but not whose combinatorics) depends on the choice of Cartan matrix~$A$.  
The results on the Cambrian fan are valid for any choice of~$A$, since they rely only on the combinatorial results of the previous sections.  
For our present purposes, we could restrict further to the \newword{symmetric} case where $\delta$ is identically $1$. 
(Note, for any $W$, we can make a symmetric choice of $A$.)
However, we would then have to introduce a new set of conventions when relating the $c$-Cambrian fan to cluster algebras, so we work with the symmetrizable case from the beginning.

The set of vectors of the form $w\alpha_s$, for $s\in S$ and $w\in W$, is the \newword{root system}~$\Phi$ and its elements are called roots.\footnote{In some contexts, these are called \newword{real roots}.}
The positive roots $\Phi_+$ are the roots in $\Phi$ which are in the positive linear span of $\Pi$.
Since the simple roots are linearly independent, each positive root has a unique expression as a positive combination of simple roots.
The negative roots are $\Phi_-=-\Phi_+$ and $\Phi=\Phi_+\cup\Phi_-$.
There is a bijection $t\mapsto\beta_t$ between the reflections $T$ in~$W$ and the positive roots.\footnote{This statement depends on the assumption that $\delta(s)=\delta(s')$ if $s$ and $s'$ are conjugate.}
We define $\beta_t^{\vee} = (2/K(\beta_t, \beta_t)) \beta_t$. If $t = w s w^{-1}$, then $\beta^{\vee}_t = \delta(s)^{-1} \beta_t$.
The action of $t$ on $V$ is by the relation $t \cdot x=x - K(\beta_t^{\vee}, x) \beta_t=x - K(x, \beta_t) \beta^{\vee}_t$.
In the action of~$W$ on $V^*$, the element $t$ is represented by a reflection fixing the hyperplane $H_t=\set{x^*\in V^*:\br{x^*,\beta_t}=0}$.
Under this bijection, $\alpha_s=\beta_s$ and $\alpha^{\vee}_s=\beta^{\vee}_s$ for each $s\in S$.

For $w\in W$ and $t \in T$, the reflection $t$ is an inversion of $w$ if and only if $w^{-1} \beta_t$ is a negative root.
The action of a simple generator $s\in S$ on a positive root is $s\beta_t=\beta_{sts}$ unless $t=s$, in which case we have $s\alpha_s=-\alpha_s$.

We now review the polyhedral geometry of the representation of $W$ in $V^*$.
A \newword{polyhedral cone} in $V^*$ is the intersection of finitely many closed halfspaces whose boundary hyperplane contains the origin. 
A \newword{face} of a polyhedral cone $C$ is a set of the form $\{ x \in C: \lambda(x) = 0 \}$, where $\lambda$ is a real-valued linear function which is nonnegative on $C$.
%A \newword{face} of a polyhedral cone is the intersection of the cone with any hyperplane $H$ (through the origin) such that the cone is completely contained in one of the closed halfspaces defined by $H$.
A polyhedral cone is \newword{pointed} if its minimal face is the origin.
A \newword{simplicial cone} is a polyhedral cone of full dimension in $V^*$ such that the normal vectors to the defining halfspaces are a basis for $V$.
A set $\F$ of polyhedral cones in a real vector space is a \newword{fan} if 
\begin{enumerate}
\item[(i) ] for every $F\in\F$, the faces of $F$ are also in $\F$ and 
\item[(ii) ] for any subset $\X\subset\F$, the intersection $\bigcap_{G\in \X} G$ is a face of $F$ for each $F\in\X$.
\end{enumerate}

The \newword{dominant chamber in $V^*$} is the set 
\[D=\bigcap_{s \in S} \set{v^*\in V^*: \br{v^*,\alpha_s}\ge 0}\]
The dominant chamber is an $n$-dimensional simplicial cone. 
The set $\F(W)$ of all cones $wD$ and their faces is a fan in $V^*$ which we call the \newword{Coxeter fan}.
The union of the cones of $\F(W)$ is a convex subset of $V^*$ known as the Tits cone and denoted $\Tits(W)$. 
(Both $\F(W)$ and $\Tits(W)$ depend on the choice of reflection representation of $W$.)

Each cone $wD$ is a \newword{region} in the hyperplane arrangement $\A=\set{H_t:t\in T}$, meaning that $wD$ is a full-dimensional connected component of $V\setminus\left(\cup_{t\in T}H_t\right)$.
The map $w\mapsto wD$ is a bijection between elements of~$W$ and regions of $\A$ which lie in $\Tits(W)$. 
The inversion set of $w\in W$ equals $\set{t\in T:\br{x^*,\beta_t}\le 0\mbox{ for all }x^*\in wD}$, or in other words, the set of reflections $t$ such that $H_t$ separates $wD$ from $D$. 

\begin{example} \label{B2 cones}
On the right hand side of Figure~\ref{B2 root figure}, we illustrate the hyperplane arrangement $\A$ in $V^*$, in the case where $W$ is $B_2$.
Since $W$ is finite in this case, the Tits cone is all of $V^*$.
From this figure, the reader can check that, for example, the inversions of $pq$ are $p$ and $pqp$, and can similarly determine the inversions of every other element in $W$.

%The left half of Figure~\ref{B2 root figure} is drawn at an odd angle to emphasize that we do not identify $V$ and $V^*$.
\end{example}

\subsection{Conjugacy classes of reflections} \label{RootConjugacy}

In defining a symmetrizable generalized Cartan matrix, we required that $\delta(s)=\delta(s')$ whenever $s$ and $s'$ are conjugate. As stated, this condition appears difficult to check. In this section, we point out that this condition is very computable, and that it is automatic if $A$ is crystallographic.

Define $\Gamma^{\odd}$ to be the graph whose vertex set is $S$ and where there is an edge between $s$ and $s'$ if and only if $m(s, s')$ is odd. For this purpose, $\infty$ is even.  
The following proposition appears as Exercise~1.16.a of~\cite{Bj-Br}.
\begin{proposition} \label{even conj}
Two simple reflections $s$ and $s'$ are conjugate if and only if they lie in the same connected component of $\Gamma^{\odd}$.
\end{proposition}

\begin{cor}
Let $A$ be a crystallographic generalized Cartan matrix and let $\delta$ be a positive real-valued function on $S$ such that $\delta(s) a_{s s'}=\delta(s') a_{s' s}$ for all $s,s'\in S$.
Then $\delta(s)=\delta(s')$ whenever $s$ and $s'$ are conjugate. 
In particular, $A$ is symmetrizable.
\end{cor}

\begin{proof}
 We may immediately reduce to the case that $s$ and $s'$ are joined by an edge of $\Gamma^{\odd}$. The requirement that $A$ be crystallographic implies that $m(s,s')$ is $2$, $3$, $4$, $6$ or $\infty$.  
Thus $m(s,s')$ is odd only when $m(s,s')=3$. 
In this case, $a_{s s'} a_{s' s} = 1$ for nonpositive integers $a_{s s'}$ and $a_{s' s}$.
Therefore $a_{s s'}=a_{s' s}=-1$. 
Since $\delta(s) a_{s s'}=\delta(s') a_{s' s}$, we have $\delta(s)=\delta(s')$ as desired.
\end{proof}

\subsection{Parabolic subgroups} \label{Parabolic Subsec}
A \newword{standard parabolic subgroup} of~$W$ is a subgroup $W_J$ generated by some subset $J\subseteq S$ of the simple generators.
The most important case will be when $J=S\setminus\set{s}$ for some $s\in S$. 
We use the notation $\br{s}=S\setminus\set{s}$. 
The standard parabolic subgroup $W_J$ is in particular a Coxeter group with simple generators $J$.
The reflections of $W_J$ are exactly $T_J=T\cap W_J$.

The following lemma about standard parabolic subgroups will be useful later.
\begin{lemma}\label{span inv s}
If $w\in W_{\br{s}}$, then the positive root $\beta_{wsw^{-1}}$ is in the positive linear span of $\set{\alpha_s}\cup\set{\beta_t:t \in \inv(w)}$.
\end{lemma}
\begin{proof}
We argue by induction on $\ell(w)$.
The base case $w=e$ is trivial.
If $\ell(w)>0$ then write a reduced word $r_1\cdots r_{\ell(w)}$ for~$w$ and let $t_i=r_1\cdots r_i\cdots r_1$ for $i=1,\ldots,\ell(w)$.
Consider the element $w' :=r_2\cdots r_{\ell(w)}$.  
By induction, $\beta_{w' s(w')^{-1}}$ is in the positive linear span of $\set{\alpha_s}\cup\set{\beta_{t}:t \in \inv(w')}$.
Thus $\beta_{wsw^{-1}}$ is in the positive linear span of $\set{\beta_{r_1 s r_1}}\cup\set{\beta_{t}:t \in r_1 \inv(w') r_1}$.
But $\beta_{r_1 s r_1}=\alpha_s-a_{s r_1} \alpha_{r_1}$ is in the positive linear span of $\alpha_s$ and $\alpha_{r_1}=\beta_{t_1}$. 
(Since $s \neq r_1$, we know that $a_{s r_1} \leq 0$.) 
Thus we see that $\beta_{wsw^{-1}}$ is in the positive span of $\beta_t$ for $t \in r_1 \inv(w') r_1 \cup \set{r_1, s}=\inv(w) \cup \set{s}$.
\end{proof}

Given $w\in W$ and $J\subseteq S$, the coset $W_Jw$ contains a unique element $\closeleftq{J}w$ of minimal length.
Equivalently, there is a unique factorization $w=w_J\cdot\closeleftq{J}w$ maximizing $\ell(w_J)$ subject to the requirement that $\ell(w_J)+\ell(\closeleftq{J}w)=\ell(w)$.
(For a dual formulation, see \cite[Section~2.4]{Bj-Br}.)
The map $w\mapsto w_J$ from~$W$ to $W_J$ is defined by the property that $\inv(w_J)=\inv(w)\cap W_J$. 
To see why, first write a reduced word $a_1\cdots a_k$ such that $a_1\cdots a_i$ is a reduced word for $w_J$, for some $i\le k$.
Then since the inversions of~$w$ are of the form $t_j=a_1\cdots a_j\cdots a_1$ for $1\le j \le k$, we have $\inv(w_J)\subseteq\inv(w)\cap W_J$. 
If there exists a reflection $t\in\left[\inv(w)\cap W_J\right]\setminus\inv(w_J)$ then there exists~$j$ with $i<j\le k$ such that $t=a_1\cdots a_j\cdots a_1$.
Since $tw=a_1\cdots\widehat{a_j}\cdots a_k$ (deleting $a_j$ from $a_1\cdots a_k$), the element $u$ given by $a_{i+1}\cdots\widehat{a_j}\cdots a_k$ is $(w_J)^{-1}\cdot t\cdot w$.
In particular, $u$ is a representative of the coset $W_Jw$, because $t$ and $w_J$ are in $W_J$.
This contradicts the fact that $\closeleftq{J}w$ is the minimal-length coset representative, thus proving that $\inv(w_J)=\inv(w)\cap W_J$.

A \newword{parabolic subgroup} of~$W$ is any subgroup $W'$ conjugate in~$W$ to a standard parabolic subgroup.
In particular, a parabolic subgroup $W'=wW_Jw^{-1}$ is generated by the reflections $\set{wsw^{-1}:s\in J}$.
Thus the following theorem applies.

\begin{theorem}\label{deodyer}
Let $W$ be a Coxeter group and let $W'$ be a subgroup of $W$ generated by reflections.
Then there is a unique subset $R\subseteq W' \cap T$ satisfying the following two conditions:
\begin{enumerate}
\item[(i) ]For any $t\in W'\cap T$, the positive root $\beta_t$ is in the positive span of $\set{\beta_r:r \in R}$.
\item[(ii) ]For any $r\in R$, the positive root $\beta_r$ is not in the positive span of the set $\set{\beta_t:t \in W' \cap T \setminus \set{r }}$.
\end{enumerate}
Furthermore $(W',R)$ is a Coxeter system with reflections $W'\cap T$.
\end{theorem}
We call $R$ the \newword{canonical generators} of $W'$.
Theorem~\ref{deodyer} is the main theorem of~\cite{Deodhar}; however, some details in the statement of Theorem~\ref{deodyer} occur only in the proof of the main theorem of~\cite{Deodhar}.
Essentially the same theorem was proved in~\cite{DyerReflection} at about the same time.

A standard parabolic subgroup $W_J$ has canonical generators $J$.
More generally, let $W'=wW_Jw^{-1}$ be a parabolic subgroup.
Then the set $\Phi'=w\Phi_Jw^{-1}$ is a root system for $W'$.
Define $\Phi'_+=\Phi'\cap\Phi_+$ and let $\Pi'$ be the unique minimal subset of $\Phi'_+$ such that any root in $\Phi'_+$ is contained in the positive linear span of $\Pi'$.
Theorem~\ref{deodyer} says that $\Pi'$ is the set of positive roots associated to the canonical generators of $W'$.
We have an additional characterization of the canonical generators of parabolic subgroups.
The characterization refers to the unique minimal coset representative $w^J$ of the coset $wW_J$.
(The left-right dual notation $\closeleftq{J}w$ was used above.)

\begin{prop}\label{can gen}
Let $R$ be the set of canonical generators of a parabolic subgroup $W'=w W_J w^{-1}$.
If $w = w^J$ then $R= w J w^{-1}$.
\end{prop}
\begin{proof}
Let $w = w^J$.
We only need to show that $w J w^{-1}$ satisfies conditions (i) and (ii) of Theorem~\ref{deodyer}.
It is enough to show that, for each reflection $u\in T\cap W_J$, the root $w\beta_u$ is positive:
If so then conditions (i) and (ii) on $wJw^{-1}$ (for the subgroup $W'$) follow from conditions (i) and (ii) on $J$ (for the subgroup $W_J$).

Let $u\in T\cap W_J$.
Now $\ell(wu)>\ell(w)$ since $w$ is the unique minimal-length coset representative of $wW_J$.
Thus $wuw^{-1}\not\in\inv(w)$ (since $\ell(wuw^{-1}w)=\ell(wu)>\ell(w)$), so $w^{-1}\beta_{wuw^{-1}}$ is a positive root.
This positive root is necessarily $\beta_{w^{-1}wuw^{-1}w}=\beta_u$.
Since $w^{-1}\beta_{wuw^{-1}}=\beta_u$, the root $w\beta_u$ is the positive root $\beta_{wuw^{-1}}$.
\end{proof} 

Let $J\subseteq S$ and let $V_J$ be the subspace of $V$ spanned by $\set{\alpha_s:s \in J}$. 
Then $W_J$ preserves $V_J$ and acts on $V_J$ by the representation $V(W_J)$; henceforth we will identify $V(W_J)$ with $V_J$. 
In particular, if $t$ is a reflection of $W_J$, the notation $\beta_t$ denotes the same vector whether we think of it a root of~$W$ or of $W_J$. 
The root system associated to $W_J$ is $\Phi_J=\Phi\cap V_J$.
Dual to the inclusion of $V_J$ into $V$, there is a surjection from $V^*$ onto $(V_J)^*$, which we denote by $P_J$. The dominant chamber for $W_J$ in $V^*_J$ is the set 
\[D_J=\bigcap_{s \in J} \set{v^*\in V^*_J: \br{v^*,\alpha_s}\ge 0}=P_J(D).\]
The projection from~$W$ to $W_J$ satisfies $P_J(w D)\subseteq w_J(D_J)$ and furthermore 
\[\bigcup_{\substack{x\in W\\x_J=w_J}}xD=(P_J)^{-1}(w_JD_J)\cap\Tits(W).\]

We will need a generalization of the notion of a rank two parabolic subgroup. 
Let $t_1$ and $t_2$ be distinct reflections of $W$. 
Then the \newword{generalized rank two parabolic subgroup containing $t_1$ and $t_2$} is the subgroup $W'$ of $W$ generated by all reflections $t$ such that $\beta_t$ is in the vector space $U$ spanned by $\beta_{t_1}$ and $\beta_{t_2}$. 
%If $r_1$ and $r_2$ are distinct reflections in $W'\cap T$ then $r_1$ and $r_2$ are called the \newword{canonical generators} of $W'$ if, for every reflection $t \in W' \cap T$, the root $\beta_t$ is in the positive span of $\beta_{r_1}$ and $\beta_{r_2}$.
If distinct reflections $t_1$ and $t_2$ lie in a rank two  parabolic subgroup $W'$ then $W'$ is the generalized rank two parabolic subgroup containing $t_1$ and $t_2$.
There is a geometric criterion for when a generalized rank two parabolic subgroup is actually parabolic: 

\begin{prop} \label{para geom}
Let $p$ and $q$ be two distinct reflections of $W$. Then the generalized rank two parabolic subgroup containing $p$ and $q$ is a parabolic subgroup if and only if the convex set $H_p \cap H_q \cap \Tits(W)$ has dimension $n-2$.
\end{prop}
Since $H_p \cap H_q$ is a linear space of dimension $n-2$, the condition is that a positive volume portion of this linear space meets $\Tits(W)$.

\begin{proof}
Suppose that $H_p \cap H_q$ meets $\Tits(W)$ in the required manner. 
The intersection $H_p \cap H_q \cap \Tits(W)$ is a union of faces of the Coxeter fan, so there must be some face $\gamma$ of the Coxeter fan, of dimension $n-2$, with $\gamma \subseteq H_p \cap H_q$.
Let $w$ be such that $\gamma$ lies in the boundary of $wD$. 
Then $w^{-1}(\gamma)$ is a face of $D$ of dimension $n-2$, and must be of the form $D \cap \alpha_r^{\perp} \cap \alpha_s^{\perp}$ for some simple reflections $r$ and $s$.
The reflections $w^{-1}pw$ and $w^{-1}qw$ fix $D \cap \alpha_r^{\perp} \cap \alpha_s^{\perp}$, so $w^{-1}pw$ and $w^{-1}qw$ lie in $W_{\{ r,s \}}$ by \cite[Lemma 4.5.1]{Bj-Br}.
Thus $p$ and $q$ lie in $w W_{\{ r,s \}} w^{-1}$, so the parabolic subgroup $w W_{\{ r,s \}} w^{-1}$ is the generalized rank two parabolic subgroup containing $p$ and $q$.

The proof of the converse is similar, and easier.
\end{proof}

\begin{example}\label{not para}
Proposition~\ref{para geom} says in particular that to find a generalized parabolic subgroup which is not parabolic, one must find hyperplanes that meet outside the Tits cone. 
A simple example is to take $W$ to be $\tilde{A}_2$, the rank three Coxeter group with $m(r,s)=m(r,t)=m(s,t)=3$.
The Tits cone for $\tilde{A}_2$ is the union of an open half space with the singleton $\set{0}$.
The reflecting planes $H_{rsr}$ and $H_t$ meet in a line $L$ on the boundary of the Tits cone.
There are infinitely many reflections of $W$ whose reflecting planes contain $L$, so the generalized rank two parabolic subgroup containing $rsr$ and $t$ is infinite.
However, all rank two parabolic subgroups of $\tilde{A}_2$ are finite.
\end{example}

\begin{prop} \label{RankTwoIsOK}
Let $W'$ be a generalized rank two parabolic subgroup of $W$. 
Then $W'$ has exactly two canonical generators $\set{r_1,r_2}$.
Furthermore, $W' \cap T=\{ r_1, r_1 r_2 r_1, r_1 r_2 r_1 r_2 r_1, \ldots, r_2 r_1 r_2 r_1 r_2, r_2 r_1r_2, r_2 \}$.
\end{prop}
%In particular the set $\Phi'$ of roots contained in $U$ is a root system for $W'$ with positive roots $\Phi'\cap\Phi_+$ and simple roots $\set{\beta_{r_1},\beta_{r_2}}$.

\begin{proof}
Let $R$ be the set of canonical generators of $W'$.
Since $\set{\beta_t:t\in W'\cap T}$ contains at least two independent vectors, $R$ has at least two elements.
We claim that $R$ has exactly two elements. 
If $p$, $q$ and $r$ are distinct elements of $R$ then $\beta_p$, $\beta_q$ and $\beta_r$ all lie in the two dimensional space spanned by $\beta_{t_1}$ and $\beta_{t_2}$.
So there is a linear relation $a \beta_p+b \beta_q + c \beta_r=0$. 
Since $\beta_p$, $\beta_q$ and $\beta_r$ are all positive roots, $a$, $b$ and $c$ can't all have the same sign; assume without loss of generality that $a<0$ and \mbox{$b,c \geq 0$}. 
Then $\beta_p$ is in the positive span of $\beta_q$ and $\beta_r$, contradicting property (ii) of the set $R$. 

Now, Theorem~\ref{deodyer} asserts that $W'$ is a Coxeter group with simple generators $R$ and that $W' \cap T$ is the set of reflections of $W'$ in this presentation. 
This fact implies the last assertion.
\end{proof}

\begin{remark} One corollary of Proposition~\ref{RankTwoIsOK} is that a subgroup $W'$ of $W$ is a generalized parabolic of rank two if and only if $W'$ is generated by two reflections and is not contained in any larger subgroup which is generated by a pair of reflections. This characterization shows that the property of being a generalized parabolic subgroup of rank two depends only on $W$, not on the choice of root system for $W$.
\end{remark}

\begin{remark} \label{gen high rank}
We have deliberately avoided defining generalized parabolic subgroups of rank larger than two.
One could take any linear space $L \subset V$ and take the subgroup $W'$ of $W$ generated by the reflections for the roots lying in $L$.
Then Theorem~\ref{deodyer} states that $(W',R)$ is a Coxeter system, where $R$ is the set of canonical generators of $W'$.
Surprisingly, $(W,R)$ may have rank greater than $\dim L$.
(We thank Matthew Dyer for pointing this out to us.)
For example, let $W=\tilde{A}_3$, the rank four Coxeter group with $m(r,s)=m(s,t)=m(t,u)=m(u,r)=3$ and $m(r,t)=m(s,u)=2$.
Take $L = \{ a \alpha_r + b \alpha_s + c \alpha_t + d \alpha_u : a+c=b+d \}$.
One can show that the group $W'$ generated by reflections in the roots contained in $L$ has canonical generators $\set{rsr, sts, tut, uru}$.
This is the Coxeter group of type $\tilde{A}_1 \times \tilde{A}_1$ and rank four.
The lack of control over rank forms an obstacle whenever we want to induct on rank, and we have chosen to avoid the matter altogether in this paper.
\end{remark}

We now mention several easy lemmas about parabolic subgroups.
(Cf. \cite[Lemma~1.2]{sortable}, \cite[Lemma~1.4]{sortable}, \cite[Lemma~1.5]{sortable} and \cite[Lemma~6.6]{congruence}.
An inconsequential misstatement in \cite[Lemma~1.5]{sortable} is corrected in Lemma~\ref{int} below.)
In the proofs that follow, $\Phi'$ will denote the root system associated to $W'$. 

\begin{lemma}\label{s in}
Let~$W'$ be a parabolic subgroup (of any rank) or generalized rank-two parabolic subgroup with canonical generators $R$.
If $s\in S\cap W'$ then $s\in R$. 
\end{lemma}
\begin{proof}
The root $\alpha_s$ spans an extreme ray in the positive span of $\Phi_+$ and therefore must span an extreme ray in the positive span of $\Phi'_+$.
\end{proof}

\begin{lemma}
\label{canon conj}
Let~$W'$ be a parabolic subgroup (of any rank) or generalized rank-two parabolic subgroup with canonical generators $R$ and let $s\in S$.
If $s\not\in R$ then the parabolic subgroup $sW's$ has canonical generators $sRs=\set{ss's:s'\in R}$.
\end{lemma}
\begin{proof}
The action of $s$ permutes $\Phi_+\setminus\set{\alpha_s}$.
Therefore the positive roots associated to $sW's$ are $s\left(\Phi'_+\right)$.
\end{proof}

\begin{lemma}\label{int}
For any standard parabolic subgroup~$W_J$ and any generalized rank two parabolic subgroup~$W'$, the intersection $W'\cap W_J$ is either $\set{e}$, a pair $\set{e,p}$ where $p$ is a canonical generator of $W'$, or all of~$W'$.
\end{lemma}
\begin{proof}
If $\Phi'\cap\Phi_J=\emptyset$ then $W'\cap W_J=\set{e}$.
If $\Phi'\subseteq\Phi_J$ then $W'\cap W_J=W'$.
Otherwise, since both $\Phi'$ and $\Phi_J$ can be described as the intersection of $\Phi$ with a subspace, $\Phi'\cap\Phi_J$ must span a line.
In this case, $W'\cap W_J=\set{e,t}$ for some reflection $t\in W'\cap T$.
It remains to show that $t$ is a canonical generator of $W'$.

Supposing to the contrary that $t$ is not a canonical generator, the positive root $\beta_t$ is in the positive linear span of $\set{\beta_{r_1},\beta_{r_2}}$, where $r_1$ and $r_2$ are the canonical generators of $W'$.
However, $\beta_t$ is also in the positive linear span of $\set{\alpha_s:s\in J}$.
Since $\beta_{r_1}$ and $\beta_{r_2}$ are not in $\Phi_J$, they are not in the linear span of $\set{\alpha_s:s\in J}$.
This contradicts the fact that $\beta_t$ has a unique expression as a positive combination of simple roots.
\end{proof}

Let $W'$ be a generalized rank two parabolic subgroup of~$W$ with canonical generators $r_1$ and $r_2$. 
We index the reflections of $W'$ as 
\[u_1=r_1,\,\, u_2=r_1r_2r_1,\,\, u_3=r_1r_2r_1r_2r_1,\,\,\ldots\,\,, u_{m-1}=r_2r_1r_2,\,\, u_m=r_2,\]
where $m=m(r_1,r_2)$ is the order of $r_1r_2$.
If $m$ is finite, then the index set is the integers from 1 to $m$.
If $m=\infty$ then the index set is the union of the positive integers with the set of formal expressions $\infty-k$, where $k$ is a nonnegative integer.
The index set is totally ordered in the obvious way, with the convention that $i<\infty-k$ for any positive integer $i$ and nonnegative integer $k$.
The following lemma is the second Proposition of~\cite{Pilk}, restated in our language: 
\begin{lemma} \label{PilkLemma}
 Let $I$ be a finite subset of $T$. Then the following are equivalent:
\begin{enumerate}
 \item[(i)] There is an element $w$ of $W$ such that $I=\inv(w)$.
\item[(ii)] For every generalized parabolic subgroup $W'$ of rank two, with canonical generators $r_1$ and $r_2$, the intersection $I \cap W'$ is either an initial or a final segment of $u_1, u_2, \ldots, u_m$.
\end{enumerate}
\end{lemma}

\subsection{Weak order} \label{weak lattice sec}
The \newword{(right) weak order} on~$W$ is the transitive closure of the relations $w < ws$ for all $w\in W$ and $s\in S$ such that $\ell(w)<\ell(ws)$. 
The weak order is ranked by the length function $\ell$ so that in particular these defining relations are cover relations. 
Order relations can also be described directly:  $v\le w$ if and only if some (or equivalently every) reduced word for $v$ occurs as a prefix of some reduced word for~$w$.
There is also a left weak order, with cover relations relating~$w$ and~$sw$.
However, in this paper, all references to the weak order refer to the right weak order.
Furthermore, the symbol ``$\le$'' as applied to elements of~$W$ always refers to the weak order.
The weak order has $w\le w'$ if and only if $\inv(w)\subseteq\inv(w')$.
Thus in particular $x\covered y$ implies that $x_{\br{s}}$ is covered by or equal to $y_{\br{s}}$. 

Recall that when $\ell(sw)<\ell(w)$, we have $\inv(w)=\set{s}\cup\set{sts:t\in\inv(sw)}$.
As immediate consequence, we obtain a simple but useful fact about the weak order.
For $s\in S$, let $W_{\ge s}$ be the set of elements of~$W$ which are above $s$ in the weak order to and let $W_{\not\ge s}$ be $W\setminus W_{\ge s}$.

\begin{prop} \label{AboveBelow}
Let $w \in W$ and $s \in S$. 
Then $\ell(sw)<\ell(w)$ if and only if $w\ge s$ if and only if $s\in\inv(w)$. 
Left multiplication by~$s$ is an isomorphism of posets from $W_{\not \geq s}$ to $W_{\geq s}$. 
If $w' \gtrdot w$, $w' \geq s$ and $w \not \geq s$ then $w'=sw$.
\end{prop}

The following anti-isomorphism result for intervals will also be useful.
\begin{prop}\label{anti}
For any $w\in W$, the interval $[e,w]$ in the weak order is anti-isomorphic to the interval $[e,w^{-1}]$, via the map $x\mapsto w^{-1}x$.
\end{prop}
\begin{proof}
Suppose $x\le y\le w$.
Then there exists a reduced word $a_1\cdots a_k$ for~$w$  and $i$ and~$j$ with $i\le j\le k$ such that $a_1\cdots a_i$ is a reduced word for $x$ and $a_1\cdots a_j$ is a reduced word for $y$.
Then $a_ka_{k-1}\cdots a_{j+1}$ is a reduced word for $w^{-1}y$, occurring as a prefix of the reduced word $a_ka_{k-1}\cdots a_{i+1}$ for $w^{-1}x$, which in turn is a prefix of the reduced word $a_ka_{k-1}\cdots a_1$ for $w^{-1}$.
In particular the map $x\mapsto w^{-1}x$ is a well-defined order-reversing map from $[e,w]$ to $[e,w^{-1}]$.
By swapping~$w$ with $w^{-1}$, we see that the inverse map $x\mapsto wx$ is also order-reversing. \end{proof}

The weak order is finitary, meaning that for each $w\in W$, the set of elements below~$w$ is finite.
Weak order on~$W$ is also a complete meet-semilattice with a unique minimal element~$e$, the identity element of~$W$.
(Recall that a meet-semilattice is complete if the meet is defined for any nonempty set of elements, even an infinite set.)
In particular, any set of elements of~$W$ has at most one minimal upper bound (the meet of all upper bounds), and thus any interval in the weak order is a finite lattice.
The symbol ``$\meet$'' will refer to the meet in the weak order and the symbol ``$\join$'' will refer to the join, when the join exists.
An element~$j$ in a meet-semilattice is called \newword{join-irreducible} if any finite set $X$ with $j=\Join X$ has $j\in X$.
This is equivalent (in a finitary meet-semilattice $L$) to the requirement that~$j$ covers exactly one element of $L$.
(Note that the word ``exactly'' cannot be replaced by ``at most one'' here:
The only element of $L$ that covers no other element of $L$ is the minimal element $\0$, which is $\Join\emptyset$.)
Thus in the case of the weak order on~$W,$ join-irreducible elements are those which have a unique cover reflection.

When $W$ is finite, the weak order on $W$ is a lattice.
In particular, $W$ has a unique maximal element $w_0$.
This maximal element is an involution.
Conjugation by $w_0$ permutes $S$.

We now describe the interaction of the weak order with the map $w\mapsto w_J$.
It is immediate that this map is order preserving. 
In fact, much more is true.
For any $A\subseteq W$ and $J\subseteq S$, let $A_J=\set{w_J:w\in A}$. 
The following result is proven by Jedli\v{c}ka in~\cite{Jed}. 
(Jedli\v{c}ka denotes the map $w \mapsto w_J$ by $\alpha_J$. 
The fact that $\alpha_J$ commutes with meet is \cite[Lemma 4.2.iii]{Jed} and the fact that it commutes with join is \cite[Lemma 4.5]{Jed}.)

\begin{prop}\label{para hom}
For any $J\subseteq S$ and any subset $A$ of $W$, if $A$ is nonempty then $\Meet (A_J)=\left( \Meet A \right)_J$ and, if $A$ has an upper bound, then $\Join (A_J)=\left( \Join A \right)_J$.
\end{prop}

Proposition~\ref{para hom} has the following simple, but useful, corollary.
\begin{cor}\label{join below}
Let $x\not\ge s$ and $y\not\ge s$ and suppose that $x\join y$ exists.
Then $(x\join y)\not\ge s$. 
\end{cor}
\begin{proof}
The assertion that $x\not\ge s$ is equivalent to the assertion that $x_{\set{s}}$ is the identity element~$e$.
(Notice that the subscript is $\set{s}$, not $\br{s}$.)
If $x_{\set{s}}=e$ and $y_{\set{s}}=e$ then $(x\join y)_{\set{s}}=x_{\set{s}}\join y_{\set{s}}=e$ by Proposition~\ref{para hom}.
\end{proof}

We now give two lemmas about cover reflections and the join operation.
These are \cite[Lemma~2.7]{sort_camb} and \cite[Lemma~2.8]{sort_camb} respectively.

\begin{lemma}
\label{s join w br s}
For $w\in W,$ if~$s$ is a cover reflection of~$w$ and every other cover reflection of~$w$ is in~$W_{\br{s}}$ then $w=s\join w_{\br{s}}$.
\end{lemma}

\begin{lemma}
\label{cov w br s}
If $w\in W_{\br{s}}$ and $s\join w$ exists then $\cov(s\join w)=\cov(w)\cup\set{s}$. 
\end{lemma}

We now use the weak order to prove several more lemmas about cover reflections, reflection sequences and canonical generators of parabolic subgroups.

\begin{lemma}\label{cov para}
Let $w\in W$ and let $J\subseteq S$.
If $\cov(w)\subseteq W_J$ then $w\in W_J$.
\end{lemma}
\begin{proof}
Suppose $w\not\in W_J$ so that in particular $w_J<w$.
Choose any $w'$ such that $w_J\le w'\covered w$.
Since the map $x\mapsto x_J$ is order preserving, $(w')_J=w_J$.
In particular $\inv(w')\cap W_J=\inv(w)\cap W_J$, so that the unique reflection in $\inv(w)\setminus\inv(w')$ is not in $W_J$.
Thus $\cov(w)\not\subseteq W_J$.
\end{proof}

The following lemma is very close in spirit to \cite[Lemma~2.11]{Dyer}. 
The notation $u_1, u_2, \ldots, u_m$ is explained in the paragraph before Lemma~\ref{PilkLemma}.

\begin{lemma}  \label{dyer lem}
Let $t_1$, $t_2$, \ldots, $t_k$ be a sequence of distinct reflections in $T$. The following are equivalent:
\begin{enumerate}
 \item[(i)] There is a reduced word $s_1 s_2 \cdots s_k$ in $W$ such that $t_1$, $t_2$, \dots, $t_k$ is the reflection sequence of $s_1 s_2 \cdots s_k$.
\item[(ii)] For every generalized parabolic subgroup $W'$ of rank two, with canonical generators $r_1$ and $r_2$, the subsequence of  $t_1$, $t_2$, \ldots, $t_k$ consisting of elements of $W'$ is either an initial or final subsequence of $u_1, u_2, \ldots, u_m$.
\end{enumerate}
\end{lemma}

\begin{proof}
Define $I_j=\{ t_1, t_2, \ldots, t_j \}$. First, assume (i) and fix a generalized rank two parabolic subgroup $W'$ whose reflections are ordered as in (ii) above.
For every~$j$ from $1$ to $k$, the inversions of $s_1 s_2 \cdots s_j$ are $I_j$ and hence, by Proposition~\ref{PilkLemma}, $I_j \cap W'$ is either an initial or final segment of the reflections of $W'$. Since this is true for every $I_j$, the subsequence of $I$ consisting of reflections in $W'$ is either an initial or final subsequence of the reflections of $W'$. Now, assume (ii) and fix some~$j$ between~$1$ and $k$. For every rank two generalized parabolic subgroup $W'$, the intersection $W' \cap I_j$ consists of the elements of an initial or final subsequence of the reflections of $W'$. So, by Proposition~\ref{PilkLemma} again, $I_j$ is the inversion set of some element $w_j$ of~$W$. Then $e$, $w_1$, $w_2$, \dots, $w_k$ is an ascending saturated chain in the weak order (since the inversion sets are increasing) and thus there is a reduced word $s_1 s_2 \cdots s_k$ such that $I_j=\{ s_1, s_1 s_2 s_1, \ldots, s_1 s_2 \cdots s_j \cdots s_2 s_1\}$. Then we must have $\{ s_1 s_2 \cdots s_j \cdots s_2 s_1 \} = I_{j} \setminus I_{j-1} = \{ t_j \}$, and $t_1$, $t_2$, \ldots, $t_k$ is the reflection sequence of  $s_1 s_2 \cdots s_k$ as required.
\end{proof}

Lemma~\ref{dyer lem} immediately implies the following:
\begin{lemma}\label{dyer cov}
Let $w\in W$ and let $W'$ be a rank two parabolic subgroup all of whose reflections are in $\inv(w)$.
If $t\in W'$ is a cover reflection of~$w$ then~$t$ is a canonical generator of $W'$.
\end{lemma}

Lemma~\ref{dyer cov} and a well-known result show the following. 

\begin{lemma}\label{cov canon}
Let $w\in W$ and let $U=\bigcap_{t\in\cov(w)}H_t$.
Then the stabilizer of $U$ in $W$ is a finite parabolic subgroup with canonical generators $\cov(w)$.
\end{lemma}

\begin{proof}
Let the cover reflections of $w$ be $t_1$, $t_2$, \ldots, $t_r$. 
Then $t_i w=w s_i$ for some $s_i \in S$; let $J = \{ s_1, \ldots, s_r \} \subseteq S$. 
So $U=w V_J^{\perp} w^{-1}$. 
Here $V_J^{\perp}$ is the subspace of $V^*$ which is orthogonal to the subspace $V_J$ of $V$.  
We claim that the stabilizer of $U$ is $w W_J w^{-1}$. 

Each generator $t_i$ of $w W_J w^{-1}$ stabilizes $U$, so the stabilizer of $U$ is contained in $w W_J w^{-1}$. Conversely, let $u$ stabilize $U$ and write $u=w v w^{-1}$, so $v$ stabilizes $V_J^{\perp}$. In particular, $v$ stabilizes $V_J^{\perp} \cap D$ so  $v$ is in $W_J$ by \cite[Theorem~5.13(a)]{Humphreys}. 
We deduce that $u \in w W_J w^{-1}$. This concludes the proof that the stabilizer of $U$ is $w W_J w^{-1}$. 

Now, let $x$ be the minimal length element of $wW_J$. 
The map $v \mapsto xv$ is a poset isomorphism from $W_J$ to $w W_J$. 
We assumed that $w$ covers $|J|$ other members of $w W_J$ (namely $t_1 w$, \dots, $t_r w$), so we conclude that $W_J$ has an element which covers $|J|$ other elements. 
In particular, $W_J$ is finite and thus, $w W_J w^{-1}$ is finite.
Furthermore $w=x w_0(J)$, where $w_0(J)$ is the maximal element of $W_J$. 
Now, $\cov(w)=w J w^{-1}=x w_0(J) J w_0(J)^{-1} x^{-1}=x J x^{-1}$. 
(We have used the fact that conjugation by $w_0(J)$ permutes $J$.) 
By Proposition~\ref{can gen}, the canonical generators of $w W_J w^{-1}$ are $x J x^{-1}$, since $x$ is a minimal length coset representative of $wW_J$.
Thus the canonical generators of $w W_J w^{-1}$ are the cover reflections of $w$.
\end{proof}

\subsection{Coxeter elements}
A \newword{Coxeter element}~$c$ of~$W$ is an element represented by a (necessarily reduced) word $s_1s_2\cdots s_n$ where $S=\set{s_1,\ldots,s_n}$ and $n=|S|$.
Any reduced word for a given Coxeter element $c$ is a permutation of $S$, and typically $c$ has several reduced words.
As a special case of Tits' solution to the word problem for Coxeter groups  (see e.g. \cite[Section~3.3]{Bj-Br}), any two reduced words for the same Coxeter element are related by a sequence of transpositions of commuting generators.
A simple generator $s\in S$ is called \newword{initial in $c$} if there is some reduced word for $c$ having $s$ as its first letter.
Two different initial letters of the same Coxeter element must commute with each other.
A simple generator $s\in S$ is \newword{final} if it is the last letter of some reduced word for $c$.

The \newword{restriction} of a Coxeter element $c$ to $W_J$ is the Coxeter element of $W_J$ obtained by deleting the letters in $S\setminus J$ from some reduced word for~$c$.
This is independent of the reduced word chosen for $c$.

Several constructions in this paper depend on a choice of $c$, and this dependence is made explicit by superscripts or subscripts in the notation.
We often take advantage of the explicit mention of $c$ in this notation to indicate a standard parabolic subgroup.
For example, if $c'$ is the restriction of $c$ to $W_J$, the notation $\pidown^{c'}$ (see Section~\ref{pidown sec}) stands for a map defined on $W_J$, not on~$W$.

\subsection{Sortable elements} \label{Sortable Subsec}

Fix a Coxeter element~$c$ and a particular word $s_1s_2\cdots s_n$ for~$c$ and write a half-infinite word 
\[(s_1\cdots s_n)^\infty=s_1s_2\cdots s_n|s_1s_2\cdots s_n|s_1s_2\cdots s_n|\ldots\]
The symbols ``$|$'' are inert ``dividers'' which facilitate the definition of sortable elements.
When subwords of $(s_1\cdots s_n)^\infty$ are interpreted as expressions for elements of~$W,$ the dividers are ignored.
The \newword{$(s_1\cdots s_n)$-sorting word} for $w\in W$ is the lexicographically first (as a sequence of positions in $(s_1\cdots s_n)^\infty$) subword of $(s_1\cdots s_n)^\infty$ which is a reduced word for~$w$.
The $(s_1\cdots s_n)$-sorting word can be interpreted as a sequence of subsets of~$S$:
Each subset in the sequence is the set of letters of the $(s_1\cdots s_n)$-sorting word which occur between two adjacent dividers. 

A \newword{$c$-sorting word} for~$w$ is a word that arises as the $(s_1\cdots s_n)$-sorting word for~$w$ for some reduced word $s_1\cdots s_n$ for $c$.
Since any two reduced words for~$c$ are related by commutation of letters, the $c$-sorting words for~$v$ arising from different reduced words for~$c$ are related by commutations of letters, with no commutations across dividers.
%sinceTAMS:  changed subwords to subsets in following
Furthermore, the sequence of subsets defined by a $c$-sorting word is independent of the choice of reduced word $s_1\cdots s_n$ for $c$.
An element $v\in W$ is \newword{\mbox{$c$-sortable}} if a $c$-sorting word for $v$ defines a sequence of subsets which is weakly decreasing under inclusion.

\begin{example} \label{easy sort}
Let $W$ be $A_3$, the Coxeter group with $S=\set{p,q, r}$, $m(p,q) = m(q,r) =3$ and $m(p,r) =2$. 
Let $c=pqr$. Then the $c$-sortable elements are $e$, $p$, $pq$, $pqr$, $pqrp$, $pqrpq$, $pqrpqp$, $pqrq$, $pqp$, $pr$, $q$, $qr$, $qrq$ and $r$.
There are $14$ of these $pqr$-sortable elements; in general, for any Coxeter element $c$ in $A_n$, the number of $c$-sortable elements is $\frac{1}{n+2} \binom{2n+2}{n+1}$, the $(n+1)^{\mathrm{st}}$ Catalan number.

We emphasize that sorting words must, by definition, be reduced.
The word $pq|pq$ may appear to be the $pqr$-sorting word of a $pqr$-sortable element.
But, in fact, it is a non-reduced word for the element $qp$, which is not $pqr$-sortable.
\end{example}

% \begin{example}\label{easy sort}
% Let $W$ be $A_2$, the Coxeter group with $S=\set{p,q}$ and $m(p,q)=3$, and let $c=pq$.
% Then the $c$-sortable elements of $W$ are the identity, $p$, $pq$, $pqp$, and~$q$.
% The only non-$c$-sortable element is $qp$.
% \end{example}

Sortable elements satisfy the following recursion, which is easily verified from the definition. 
\begin{prop} \label{CSortRecursive}
Let~$W$ be a Coxeter group,~$c$ a Coxeter element and let~$s$ be initial in~$c$.
Then an element $w \in W$ is $c$-sortable if and only if
\begin{enumerate}
\item[(i) ] $w \geq s$ and $sw$ is $scs$-sortable or
\item[(ii) ] $w \not \geq s$ and~$w$ is an $sc$-sortable element of $W_{\br{s}}$.
\end{enumerate}
\end{prop}

Proposition~\ref{CSortRecursive} amounts to a recursive definition of sortable elements.
Recall from Proposition~\ref{AboveBelow} that $w\geq s$ is equivalent to the statement that~$w$ has a reduced word starting with~$s$, which is equivalent to $\ell(sw)<\ell(w)$.
Thus in the case $w\ge s$, the definition works by induction on length.
In the other case, the induction is on the rank of the Coxeter group.
This induction on length and rank is an essential feature of the proofs in this paper.

The following observation is immediate from the definition of sortable elements.
\begin{prop}\label{sort para easy}
Let $J \subset S$ and let $c'$ be the restriction of $c$ to $W_J$. 
Then an element $v \in W_J$ is $c$-sortable if and only if it is $c'$-sortable. 
\end{prop}

\section{The Euler form $E_c$ and the form $\omega_c$} \label{omega sec}
In this section, we define two bilinear forms $E_c$ and $\omega_c$ and establish their basic properties.
We then use the form $\omega_c$ to characterize the reflection sequences of $c$-sortable elements.
We prove some consequences of this characterization.
Throughout the section, we fix a symmetrizable choice $A$ of Cartan matrix for~$W$, and a particular choice of symmetrizing function $\delta$.

%sinceTAMS:  I changed A(\alpha^{\vee}_{s_i},\alpha_{s_j}) to K(\alpha^{\vee}_{s_i},\alpha_{s_j}) in the following
We continue to let $s_1\cdots s_n$ be a reduced word for a Coxeter element~$c$.
The \newword{Euler form} $E_c$ is defined on the bases of simple roots and co-roots as follows:
\[E_c(\alpha^{\vee}_{s_i},\alpha_{s_j})=\left\lbrace\begin{array}{ll}
K(\alpha^{\vee}_{s_i},\alpha_{s_j})=a_{s_i s_j}&\mbox{if } i>j,\\
1&\mbox{if }i=j,\mbox{ or}\\
0&\mbox{if } i<j.
\end{array}\right.\]
The symmetrization $E_c(\alpha, \beta)+E_{c}(\beta, \alpha)$ equals $K(\alpha, \beta)$.

Since all reduced words for $c$ differ only by exchanging commuting generators, the Euler form is independent of the choice of reduced word for~$c$. 
We now present three simple lemmas on the Euler form.
The first two follow immediately from the definition.

\begin{lemma}\label{Ec initial}
Let~$s$ be initial in~$c$ and let $\beta_t$ be a positive root.
Then $E_c(\alpha^{\vee}_s, \beta_t)$ equals the coefficient of $\alpha_s$ when $\beta_t$ is expanded in the basis of simple roots.
In particular, $E_c(\alpha^{\vee}_s, \beta_t) \ge 0$, with equality if and only if $t \in W_{\br{s}}$.
\end{lemma}

\begin{lemma}\label{Ec final}
Let~$s$ be final in~$c$ and let $\beta^{\vee}_t$ be a positive co-root.
Then $E_c(\beta_t^{\vee}, \alpha_s)$ equals the coefficient of $\alpha^{\vee}_s$ when $\beta_t^{\vee}$ is expanded in the basis of simple co-roots.
In particular, $E_c(\beta_t^{\vee}, \alpha_s) \ge 0$, with equality if and only if $t \in W_{\br{s}}$.
\end{lemma}

\begin{lemma}\label{Ec invariant}
If~$s$ is initial or final in $c$, then $E_c(\beta^{\vee},\beta')=E_{scs}(s \beta^{\vee}, s \beta')$ for all $\beta$ and~$\beta'$ in $V$.
\end{lemma}
\begin{proof}
The cases where~$s$ is initial and final are equivalent by exchanging the roles of~$c$ and $scs$.  
We take~$s$ to be initial and choose a reduced word for~$c$ with $s_1=s$.
By linearity, it is enough to take $p,q\in S$ and verify the relation for $\beta\ck=\alpha^{\vee}_p$ and $\beta'=\alpha_q$.
Then $E_{scs}(s\alpha^{\vee}_p,s\alpha_q)=E_{scs}(\alpha^{\vee}_p-a_{ps}\alpha^{\vee}_s,\alpha_q-a_{sq}\alpha_s)$, which equals \renewcommand{\theequation}{$\ast$}
\begin{equation}\label{rhs}
E_{scs}(\alpha^{\vee}_p,\alpha_q)  -  E_{scs}(\alpha^{\vee}_p,\alpha_s)a_{sq} -  E_{scs}(\alpha^{\vee}_s,\alpha_q)a_{ps}  +  a_{ps}a_{sq}.
\end{equation}
If $p\neq s$ and $q\neq s$ then (\ref{rhs}) is 
\begin{equation*}
E_{scs}(\alpha_p^{\vee}, \alpha_q)  -  0 \cdot a_{sq} -  a_{sq}a_{ps} +  a_{ps}a_{sq}= E_{scs}(\alpha_p^{\vee}, \alpha_q) =E_c(\alpha^{\vee}_p,\alpha_q).
\end{equation*}
If $p=s$ and $q\neq s$ then (\ref{rhs}) is $a_{sq} -  1 \cdot a_{sq}  - a_{sq} \cdot 2 + 2 \cdot a_{sq} = 0 = E_c(\alpha^{\vee}_p, \alpha_q)$.
If $p\neq s$ and $q=s$ then (\ref{rhs}) is $0  -  0 \cdot 2  -  1 \cdot a_{ps}  +  a_{ps}\cdot 2 = a_{ps} = E_c(\alpha^{\vee}_p,\alpha_q)$.
Finally, if $p=q=s$ then (\ref{rhs}) is $1-1\cdot 2-1\cdot 2+2\cdot 2=1=E_c(\alpha^{\vee}_p,\alpha_q)$.
\end{proof}

\begin{remark}
The Euler form arises in quiver theory. 
Lemmas~\ref{Ec initial},  \ref{Ec final} and~\ref{Ec invariant}, in the case where $A$ is symmetric and crystallographic, have simple interpretations and proofs in the context of quiver representations.
For a brief and accessible introduction to quiver representation theory, see~\cite{Derk-Wey}. For a thorough introduction which is well suited to our approach, see~\cite{ASS}. In particular, see  Section~III.3 for the homological motivation underlying the Euler form. One thing that is not in~\cite{ASS} is an explanation of the connection between roots and indecomposable representations for a general quiver; for that we recommend the highly readable paper of Kac~\cite{KacQuiver}. Our work is motivated by connections to cluster algebras, as is much current research on quivers; the reader who wishes to know the best current results relating quiver representation theory to cluster algebras should consult~\cite{Keller}. 
\end{remark}

Define a skew-symmetric form $\omega_c(\beta,\beta')=E_c(\beta,\beta')-E_c(\beta',\beta)$.
Retaining the fixed reduced word $s_1\cdots s_n$ for $c$, the form $\omega_c$ has the following values on simple roots $\alpha_{s_i}$ and $\alpha_{s_j}$:
\[\omega_c(\alpha_{s_i},\alpha_{s_j})=\left\lbrace\begin{array}{ll}
K(\alpha_{s_i},\alpha_{s_j})&\mbox{if }i>j,\\
0&\mbox{if }i=j,\mbox{ or}\\
-K(\alpha_{s_i},\alpha_{s_j})&\mbox{if }i<j.
\end{array}\right.\]

The form $\omega_c$, like the Euler form $E_c$, is independent of the choice of reduced word for~$c$.
In what follows, the sign of $\omega_c$, but not its exact value, will be of crucial importance.
When $W$ is of rank three, there is a vector $\zeta_c$ such that the sign of $\omega_c(\alpha, \beta)$ is positive or negative 
according to whether the angle from $\alpha$ to $\beta$ circles $\RR \zeta_c$ in a counterclockwise or clockwise direction. 
See Figures~\ref{G2tildezeta}, \ref{B3Zeta} and~\ref{542Zeta} and the surrounding text. Unlike the Euler form $E_c$, which is common in quiver representation theory, and the symmetric form $K$, which occurs in Coxeter theory, Lie theory and elsewhere, the form $\omega_c$ seems to be new. The only prior appearances we are aware of are in~\cite{reduced}, by the second author, and~\cite{Palu}.

\begin{remark}
Since $\beta^{\vee}_t$ is a positive scalar multiple of $\beta_t$, the quantities $\omega_c(\beta_{p}, \beta_{q})$ and $\omega_c(\beta^{\vee}_{p}, \beta_{q})$ always have the same sign. In any statement about the sign of $\omega_c$, it is a matter of choice which to write. In this paper, we will prefer the former, because it makes manipulations involving the antisymmetry of $\omega_c$ more transparent. However, in cases where the magnitude of $\omega_c$ is important, the latter always seems to be the relevant quantity. 
\end{remark}

\begin{example} \label{AOrientExample}
%We now describe $\omega_c$ in type $A$. 
The Coxeter group $A_{n-1}$ is naturally isomorphic to the permutation group $S_n$, with simple generators the transpositions $s_i=(i \ i+1)$. 
The reflections are the transpositions $(i j)$ for $1 \leq i < j \leq n$, with $\beta_{(i j)} = \alpha_i + \alpha_{i+1} + \cdots + \alpha_{j-1}$. 
So, for $1 \leq i < j < k \leq n$, we have $\omega_c(\beta_{ij}, \beta_{jk}) = \omega_c(\alpha_{j-1}, \alpha_j)$. 
So $\omega_c(\beta_{ij}, \beta_{jk}) > 0$ if $s_{j-1}$ precedes $s_j$ in $c$, and $\omega_c(\beta_{ij}, \beta_{jk}) < 0$ if $s_j$ precedes $s_{j-1}$. 
The value of $\omega_c$ on other pairs of roots can be worked out similarly.
%We invite the reader to work out the similar formulas for the $\omega$-pairing between the other possible pairs of roots.
\end{example}

The following are the key properties of $\omega_c$. 
The first is immediate from the definition and the second is immediate from Lemma~\ref{Ec invariant}.
Recall that $V_J$ is the real linear span of the parabolic sub root system $\Phi_J$.

\begin{lemma} \label{OmegaRestriction}
Let $J\subseteq S$ and let $c'$ be the restriction of $c$ to $W_J$.
Then $\omega_c$ restricted to $V_J$ is $\omega_{c'}$. 
\end{lemma}

\begin{lemma} \label{OmegaInvariance}
If~$s$ is initial or final in $c$, then $\omega_c(\beta,\beta')=\omega_{scs}(s\beta,s\beta')$ for all roots $\beta$ and $\beta'$.
\end{lemma}

\begin{lemma} \label{OmegaNegativity}
Let~$s$ be initial in~$c$ and let~$t$ be a reflection in~$W$.
Then \mbox{$\omega_c(\alpha_s, \beta_t) \geq 0$}, with equality only if~$s$ and~$t$ commute.
\end{lemma}

\begin{proof}
Expanding in the basis of simple roots, write $\beta_t=\sum_{r \in S} a_r \alpha_{r}$.
Since $\beta_t$ is a positive root, all of the $a_r$ are nonnegative. 
So $\omega_c(\alpha_s, \beta_t)= - \sum_{r \in S \setminus \set{s}} a_r K(\alpha_s, \alpha_r)$. 
Each $K(\alpha_s, \alpha_r)$ is nonpositive, yielding the required inequality.

Furthermore $\omega_c(\alpha_s, \beta_t)=0$ only if $a_r=0$ for all $r$ such that $\omega_c(\alpha_s, \alpha_r)\neq 0$. 
Thus $\omega_c(\alpha_s, \beta_t)=0$ only if~$t$ is a reflection in the standard parabolic subgroup $W_{J}$, where $J=\{ s' : s s' = s' s \}$.   
Since $s$ commutes with every generator of $W_J$, we know that $s$ commutes with $t$.
\end{proof}
A similar proof establishes the following:
\begin{lemma} \label{OmegaPositivity}  
Let~$s$ be final in~$c$ and let~$t$ be a reflection in~$W$.
Then $\omega_c(\alpha_s, \beta_t) \leq 0$, with strict equality only if~$s$ and~$t$ commute.
\end{lemma}

The form $\omega_c$ has a special relation with the reflection sequences of $c$-sortable elements.
Recall that when $a_1\cdots a_k$ is a reduced word for some $w\in W$, the reflection sequence associated to $a_1\cdots a_k$ is $t_1,\ldots,t_k$, where $t_i=a_1a_2\cdots a_i\cdots a_2a_1$.

\begin{prop} \label{InversionOrdering}
Let $a_1\cdots a_k$ be a reduced word for some $w\in W$ with reflection sequence $t_1,\ldots,t_k$.
Then the following are equivalent:
\begin{enumerate}
\item[(i) ]$\omega_{c}(\beta_{t_i}, \beta_{t_j}) \geq 0$ for all $i\le j$ with strict inequality holding unless $t_i$ and $t_j$ commute.
\item[(ii) ]$w$ is $c$-sortable and $a_1\cdots a_k$ can be converted to a $c$-sorting word for~$w$ by a sequence of transpositions of adjacent commuting letters.
\end{enumerate}
\end{prop}

Figures~\ref{542Zeta} and~\ref{word fig}  illustrate the sequence of partial products of a $c$-sorting word, and the corresponding reflection sequence, in such a way that the truth of condition~(i) is visually apparent.

\begin{proof}
We prove each implication separately by induction on the rank of~$W$ and the length of~$w$.
The case where either of these is $0$ is trivial and serves as our base case. 
We take advantage of the fact that both (i) and (ii) are unaffected by transpositions of adjacent commuting letters of $a_1\cdots a_k$.

Suppose (i) holds.
Let~$s$ be initial in~$c$.
If~$s$ does not occur in the sequence $a_1\cdots a_k$, then $w\in W_{\br{s}}$.
By Lemma~\ref{OmegaRestriction}, (i) holds with respect to $sc$ for~$w$ and $a_1\cdots a_k$ in $W_{\br{s}}$, so by induction on rank,~$w$ is $sc$-sortable and $a_1\cdots a_k$ can be converted to an $sc$-sorting word $a'_1\cdots a'_k$ for~$w$ by transpositions of adjacent commuting letters.
By Proposition~\ref{CSortRecursive},~$w$ is $c$-sortable.
Furthermore, $a'_1\cdots a'_k$ is a $c$-sorting word for~$w$.

If~$s$ occurs in $a_1\cdots a_k$ then, if necessary, apply a sequence of transpositions of adjacent commuting letters so as to place the  first appearance of $s$ as early as possible.
Let $a_1\cdots a_k$ now stand for the new word thus obtained and let $a_j$ be the first occurrence of $s$.  

Suppose (for the sake of contradiction) that $j>1$. Lemma~\ref{span inv s} states that $\beta_{t_j}$ is in the positive span of $\alpha_s$ and $\beta_{t_i}$ for $i<j$, say
\[\beta_{t_j}=a \alpha_s + \sum_{i=1}^{j-1} b_i \beta_{t_i}.\]
By hypothesis, $\omega_c(\beta_{t_{j-1}}, \beta_{t_j}) \geq 0$. 
On the other hand, 
\[\omega_c(\beta_{t_{j-1}}, \beta_{t_j}) = a \omega_c(\beta_{t_{j-1}}, \alpha_s) + \sum_{i=1}^{j-2} b_i \omega_{c}(\beta_{t_{j-1}}, \beta_{t_i}).\]
By our hypothesis, by Lemma~\ref{OmegaNegativity} and by the antisymmetry of $\omega_c$, all of the terms on the right side are nonpositive.
Thus $\omega_c(\beta_{t_{j-1}}, \beta_{t_j})=0$, 
and by the hypothesis again, $t_{j-1}$ and $t_j$ commute, so $a_{j-1}$ and $a_j$ commute. 
This contradiction to our choice of $a_1\cdots a_k$ proves that $j=1$ so $a_1=s$.

Now consider the reduced word $a_2\cdots a_k$ for $sw$.
By Lemma~\ref{OmegaInvariance}, condition (i) holds for $sw$ and $a_2\cdots a_k$ with respect to $\omega_{scs}$.
By induction on length, $sw$ is $scs$-sortable and $a_2\cdots a_k$ can be converted to an $scs$-sorting word for $sw$ by transpositions of commuting letters.
More specifically, since all $scs$-sorting words for $sw$ are related by transpositions of commuting letters, $a_2\cdots a_k$ can be converted to the $(s_2\cdots s_ns_1)$-sorting word for $sw$, where $s_1\cdots s_n$ is a reduced word for $c$ with $s_1=s$.
Prepending $a_1$ then gives the $(s_1\cdots s_n)$-sorting word for the $c$-sortable element~$w$.
We have established that (i) implies (ii).

Now suppose (ii) holds.
We may as well assume that $a_1\cdots a_k$ actually \textbf{is} the $(s_1\cdots s_n)$-sorting word for~$w$, where $s_1\cdots s_n$ is a reduced word for $c$.
Let~$s=s_1$.
If $w\not\ge s$ then all the $a_i$ and $t_i$ are in $W_{\br{s}}$ and all the $\beta_{t_i}$ are in $\Phi_{\br{s}}$.
By induction on rank, (i) holds with respect to $\omega_{sc}$.
By Lemma~\ref{OmegaRestriction}, (i) then holds with respect to $\omega_c$.

If $w \geq s$ then $a_1=s$.
Lemma~\ref{OmegaNegativity} says that $\omega_{c}(\alpha_s, \beta_t) \geq 0$ for all reflections~$t$, with strict inequality unless~$s$ and~$t$ commute.
Thus (i) holds in the case $i=1$.
We now consider the case $i>1$. 
The word $a_2 \cdots a_k$ is the $(s_2\cdots s_ns_1)$-sorting word for $sw$. 
For each~$j$ with $2\le j\le k$, let $t'_j=a_2a_3 \cdots a_j \cdots a_3a_2$. 
By induction on length, $\omega_{scs}(\beta_{t'_i}, \beta_{t'_j}) \geq 0$ when $2\le i<j\le k$, with strict inequality unless $t'_i$ and $t'_j$ commute. 
Now, since $a_1 \cdots a_k$ is reduced, $s \beta_{t'_j}=\beta_{s t'_j s}=\beta_{t_j}$ (rather than $- \beta_{t_j}$).
Thus $\omega_{c}(\beta_{t_i}, \beta_{t_j})=\omega_{scs}(\beta_{t'_i}, \beta_{t'_j}) \geq 0$ when $i<j$, with strict inequality unless $t_i$ and $t_j$ commute. 
\end{proof}

Proposition~\ref{InversionOrdering} leads to a uniform proof (valid for arbitrary~$W$) of the following lemma \cite[Lemma~6.7]{sortable} which was proved in~\cite{sortable} for finite~$W$ by a case-by-case check of the classification of finite Coxeter groups. 
\begin{lemma}\label{nc lemma 2}
Let~$s$ be final in~$c$ and let $v\in W$ be $c$-sortable.
Then $v \geq s$ if and only if~$s$ is a cover reflection of~$v$.
\end{lemma}

\begin{proof}
The ``if'' statement is by the definition of cover reflection, so we suppose $v \geq s$ and prove that~$s$ is a cover reflection.
Let $t_1,\ldots,t_k$ be the reflection sequence of a $c$-sorting word $a_1\cdots a_k$ for $v$ and let $i$ be such that $t_i=s$.
By Proposition~\ref{InversionOrdering}, $\omega_{c}(\alpha_s, \beta_{t_j}) \geq 0$ for all $j \geq i$ with strict inequality holding unless $t_i$ and $t_j$ commute.
Lemma~\ref{OmegaPositivity} states that $\omega_{c}(\alpha_s, \beta_t) \leq 0$ for any~$t$.
Thus $\omega_{c}(\alpha_s, \beta_{t_j})=0$ for all $j \geq i$, so that $t_i$ commutes with $t_j$ for all $j \geq i$.
Conjugating by $a_{i}\cdots a_1$, we see that $a_i$ commutes with $a_{i+1}\cdots a_j\cdots a_{i+1}$ for all $i\le j$.
By Lemma~\ref{commutes}, $a_i$ commutes with $a_{i+1}\cdots a_k$.
In particular, $va_i<v$ so that $a_1\cdots a_{i-1}a_{i+1}\cdots a_ka_ia_k\cdots a_{i+1}a_{i-1}\cdots a_1=t_i=s$ is a cover reflection of~$w$.
\end{proof}

Proposition~\ref{InversionOrdering} also implies the following useful fact.

\begin{prop}\label{sort para}
If $v$ is $c$-sortable and $J\subseteq S$, then $v_J$ is $c'$-sortable, where $c'$ is the restriction of $c$ to $W_J$.
\end{prop}
\begin{proof}
Suppose $v$ is $c$-sortable and let $t_1\cdots t_k$ be the reflection sequence of a $c$-sorting word for $v$.
Recall that the inversion set of $v_J$ is $\inv(v)\cap W_J$.
Order $\inv(v_J)$ by restricting the ordering $t_1\cdots t_k$.
Lemmas~\ref{int} and~\ref{dyer lem} imply that the restricted order is the reflection sequence for a reduced word $a_1\cdots a_m$ for $v_J$.
Lemma~\ref{OmegaRestriction} and Proposition~\ref{InversionOrdering} imply that $a_1\cdots a_m$ can be converted, by a sequence of transpositions of adjacent commuting letters, to a $c'$-sorting word for $v_J$.
\end{proof}

\section{Orientation and alignment}\label{align sec}  

In this section, we relate Proposition~\ref{InversionOrdering} to the notions of $c$-orientation and $c$-alignment defined in~\cite{sortable}. 
To understand the essential idea behind orientation and alignment, consider the example of the rank two Coxeter group~$W$ of type $B_2$, with $S=\set{s_1,s_2}$ and $m(s_1,s_2)=4$.
By Lemma~\ref{PilkLemma}, the inversion set of any element of~$W$ is either an initial segment or a final segment of the sequence $s_1,s_1s_2s_1,s_2s_1s_2,s_2$.
The inversion sets of the $s_1s_2$-sortable elements are 
\[\emptyset, \set{s_1},\, \set{s_1,s_1s_2s_1},\, \set{s_1,s_1s_2s_1,s_2s_1s_2},\, \set{s_1,s_1s_2s_1,s_2s_1s_2,s_2},\mbox{ and }\set{s_2}.\]
Notice that if the inversion set of an $s_1s_2$-sortable element contains more than one reflection, it must be an initial segment of $s_1,s_1s_2s_1,s_2s_1s_2,s_2$.
By symmetry, the inversion sets of $s_2s_1$-sortable elements containing more than one reflection must be final segments of $s_1,s_1s_2s_1,s_2s_1s_2,s_2$.
Thus the choice of Coxeter element specifies an orientation of the reflections in the rank two Coxeter group.

The content of \cite[Section~3]{sortable} is that for any finite Coxeter group~$W,$ the choice of a Coxeter element specifies an orientation, with certain desirable properties, of each rank two parabolic subgroup of~$W$.
The content of \cite[Theorem~4.1]{sortable} is that an element $w\in W$ is $c$-sortable if and only if it is $c$-aligned, meaning that, for any rank two parabolic subgroup $W'$, if the intersection $\inv(w)\cap W'$ has more than one element, it respects the $c$-orientation of $W'$ in the sense of the $B_2$ example above.

Theorem~\ref{sortable is aligned}, below, generalizes \cite[Theorem~4.1]{sortable} to the case where~$W$ is not necessarily finite.
The proof of Theorem~\ref{sortable is aligned} specializes to an alternate proof of \cite[Theorem~4.1]{sortable} without resorting to type-by-type arguments.
Theorem~\ref{sortable is aligned} is a strengthening of Proposition~\ref{InversionOrdering} which provides a more pleasant characterization of sortability:
It characterizes a sortable element by its inversion set, rather than by the requirement that its inversion sets can be totally ordered in a special way.

The statement of Theorem~\ref{sortable is aligned} differs in notation from the statement of \cite[Theorem~4.1]{sortable}.
In~\cite{sortable}, $c$-orientations are described as a directed cycle structure on the reflections of each rank two parabolic subgroup~$W,$
and \cite[Theorem~4.1]{sortable} is phrased in terms of the oriented cycle structure.
The oriented cycles of~\cite{sortable} are natural in light of the combinatorial models for $c$-orientation introduced in \cite[Section~3]{sortable}.
In the present context, it is more reasonable to totally order the reflections of $W'$ based on the sign of $\omega_c$ applied to the corresponding positive roots, and phrase Theorem~\ref{sortable is aligned} in terms of the total orders.
The interested reader will easily see the relationship between the directed cycles of~\cite{sortable} and the total orders used here.

Let $W'$ be a generalized rank two parabolic subgroup of~$W$ with canonical generators $r_1$ and $r_2$, and continue to index the reflections of $W'$ as $u_1$, $u_2$, \ldots, $u_m$ as explained in the paragraph before Lemma~\ref{PilkLemma}.

\begin{prop} \label{OmegaCyclic} 
Let $W'$ be a generalized rank two parabolic subgroup with reflections indexed as above.
\begin{enumerate}
\item[(i) ] If $\omega_c(\beta_{u_1},\beta_{u_m})=0$ then $\omega_c(\beta_{u_i}, \beta_{u_j})=0$ for all indices $i$ and~$j$. 
\item[(ii) ] If $\omega_c(\beta_{u_1},\beta_{u_m})<0$ then $\omega_c(\beta_{u_i},\beta_{u_j})<0$ for all indices $i < j$.
\item[(iii) ] If $\omega_c(\beta_{u_1},\beta_{u_m})>0$ then $\omega_c(\beta_{u_i},\beta_{u_j})>0$ for all indices $i < j$.
\end{enumerate}
\end{prop}
When (i) applies, we say that \newword{the restriction of $\omega_c$ to $W'$ is zero}.
\begin{proof}
Assertion (i) is immediate from the fact that $\omega_c$ is linear and skew-symmetric and the fact that all positive roots associated to reflections in $W'$ are in the positive span of $\beta_{u_1}$ and $\beta_{u_m}$.

Since $u_1$ and $u_m$ are the canonical generators of $W'$, we can write $\beta_{u_i}=a\beta_{u_1}+b\beta_{u_m}$ and $\beta_{u_j}=c\beta_{u_1}+d\beta_{u_m}$ for nonnegative scalars $a$, $b$, $c$ and $d$ with $a+b>0$ and $c+d>0$.  
Both assertions (ii) and (iii) follow from the assertion that $ad-bc>0$.
The latter is immediate in the case where $i=1$ or $j=m$.
If neither $i=1$ nor $j=m$ then write $u_1(\beta_{u_i})=\beta_{u_{i'}}$ for some $i'$ and $u_1(\beta_{u_j})=\beta_{u_{j'}}$ for some $j'$.
One easily verifies that $i'>j'$ and, similarly, that the action of $u_m$ on $\beta_{u_i}$ and $\beta_{u_j}$ reverses the total order on $u_2,\ldots, u_{m-1}$.
One also easily checks that the action of $u_1$ or $u_m$ reverses the sign of $ad-bc$.
Acting a finite number of times by $u_1$ and $u_m$, we eventually reach the case where $i=1$ or $j=m$, thus completing the proof.
\end{proof}

We are now prepared to state the main result of this section.
Let $W'$ be a noncommutative generalized rank two parabolic subgroup of~$W$ and let~$w$ be an element of~$W$. 
We say that~$w$ is \newword{$c$-aligned with respect to $W'$} if one of the following cases holds:
\begin{enumerate}
\item Indexing the reflections as above, $\omega_c(\beta_{u_i},\beta_{u_j}) > 0$ for all indices $i < j$, and $\inv(w) \cap W'$ is either the empty set, the singleton $\set{u_m}$ or a (necessarily finite) initial segment of $u_1, u_2, \ldots, u_m$. 
\item The restriction of $\omega_c$ to $W'$ is zero and $\inv(w) \cap W'$ is either the empty set or a singleton.
\end{enumerate} 
We say that~$w$ is \newword{$c$-aligned} if~$w$ is $c$-aligned with respect to every noncommutative rank two parabolic subgroup of~$W$. As we will see below, this implies that $w$ is also $c$-aligned with respect to every generalized noncommutative rank two parabolic subgroup.

\begin{example} \label{AAligned}
We describe $c$-aligned elements in type~$A_{n-1}$, relying on the computations from Example~\ref{AOrientExample}. Let $W'$ be a noncommutative rank two parabolic subgroup. $W'$ must be of type~$A_2$, with canonical generators $(i j)$ and $(j k)$ for some $1 \leq i < j < k \leq n$. Suppose that $s_{j-1}$ precedes $s_j$ in $c$. Then, as we saw in Example~\ref{AOrientExample}, we have $\omega_c(\beta_{(ij)}, \beta_{(jk)}) \geq 0$. In the notation above, the reflections in this rank two parabolic are $(u_1, u_2, u_3) = ((i j), (i k), (j k))$. So an element $w \in A_{n-1}$ is $c$-aligned with respect to $W'$ if and only if $\inv(w) \cap W' \neq \{ (i k), (j k) \}$. 

In terms of the one-line representation of the permutation $w$, this says that $w$ avoids $kij$. 
Similarly, if $s_{j}$ precedes $s_{j-1}$ in $c$, then $w$ must avoid $jki$. 
This recovers the pattern avoidance description of alignment from~\cite{sortable}. 
It is similarly easy in this way to recover the pattern avoidance descriptions of alignment in the other classical finite types, and produce new pattern avoidance descriptions in the classical affine types. 
In the exceptional types, there is no simple description of $W$ in terms of permutations.
However, the uniform definition of alignment in terms of $\omega_c$ is a dramatic improvement over the computer-aided definition in~\cite[Sections~3--4]{sortable}.
\end{example}

For a geometric depiction of the inversion set of a $c$-aligned element, see the discussion surrounding Figure~\ref{542Zeta}. 

\begin{theorem} \label{sortable is aligned}
An element~$w$ of~$W$ is $c$-sortable if and only if it is $c$-aligned. In this case, $w$ is also $c$-aligned with respect to every generalized noncommutative rank two parabolic subgroup of~$W$.
\end{theorem}

%\begin{remark}
%Theorem~\ref{sortable is aligned} would follow from Proposition~\ref{InversionOrdering} if we knew that there did not exist any sequence $(t_1, t_2, \ldots, t_r)$ of reflections such that $\omega_c(\beta_{t_1},\beta_{t_2})$, $\omega_c(\beta_{t_2},\beta_{t_3})$, \dots, $\omega_c(\beta_{t_{r-1}}, \beta_{t_r})$, $\omega_c(\beta_{t_r},\beta_{t_1})$ were all positive. However such cycles can exist, although not when $W$ is finite or of rank less than or equal to $3$.
%\end{remark}

\begin{proof}
The fact that $c$-sortability implies $c$-alignment follows immediately from Proposition~\ref{InversionOrdering} and Lemma~\ref{PilkLemma}. The same argument also shows that $c$-sortability implies alignment with respect to all generalized noncommutative rank two parabolic subgroups.

Suppose~$w$ is $c$-aligned.
Let $s$ be initial in $c$ and first consider the case where $s\le w$.
Let~$W'$ be a noncommutative parabolic subgroup of rank two with canonical generators $r_1$ and $r_2$.
By hypothesis,~$w$ is $c$-aligned with respect to $W'$.
We claim that $sw$ is $scs$-aligned with respect to $sW's$. 
If $r_1,r_2\neq s$ then the claim is immediate by Lemmas~\ref{canon conj} and~\ref{OmegaInvariance}.
Otherwise, take $r_1=s$ without loss of generality. In this case, $s W' s=W'$.
Since $w \geq s$, we know that $sw \not \geq s$ and $s$ is not an inversion of $sw$. 
So, by Lemma~\ref{PilkLemma}, $\inv(sw) \cap W'$  is an initial segment of $r_2$, $r_2 s r_2$, $r_2 s r_2 s r_2$ \dots. 
Since $W'$ is noncommutative, $\omega_{scs}(\beta_{r_2}, \alpha_s) >  0$ by Lemma~\ref{OmegaNegativity} and Proposition~\ref{OmegaCyclic}, and we see that $sw$ is $scs$-aligned with respect to $W'$.
Letting~$W'$ vary over all noncommutative rank two parabolic subgroups, we see that $sw$ is $scs$-aligned with respect to every noncommutative rank two parabolic subgroup, or in other words that $sw$ is $scs$-aligned.
By induction on $\ell(w)$, the element $sw$ is $scs$-sortable, so~$w$ is $c$-sortable by Proposition~\ref{CSortRecursive}.

Now consider the case where $s\not\le w$.
We first show that $w\in W_{\br{s}}$.
Suppose to the contrary that $w\not\in W_{\br{s}}$, so that in particular, $w_{\br{s}}<w$.
By Lemmas~\ref{int} and~\ref{OmegaRestriction} we see that $w_{\br{s}}$ is $sc$-aligned, so it is $sc$-sortable by induction on the rank of~$W$ and therefore $c$-sortable by Proposition~\ref{sort para easy}.
Let $a_1\cdots a_j$ be a $c$-sorting word for $w_{\br{s}}$.
Since $w_{\br{s}}<w$, we can extend this word to a reduced word $a_1\cdots a_k$ for~$w$ for some $k>j$.
Note that $a_{j+1}=s$, because otherwise $a_1\cdots a_{j+1}$ is a reduced word for an element of $W_{\br{s}}$ which is greater than $w$, contradicting the definition of $w_{\br{s}}$.
For each $i\in[k]$, let $t_i$ be the reflection given by $a_1\cdots a_i\cdots a_1$ and let $r_i=a_1\cdots a_i s a_i\cdots a_1$.
We claim that $r_i$ is an inversion of~$w$ for each $i$ with $0\le i\le j$.

The claim is proven by induction on $j-i$, with the base case $j-i=0$ given by the fact that $r_j=t_{j+1}$.
Suppose $0\le i<j$.  
The reflection $t_i$ is an inversion of~$w$ and by induction $r_{i+1}$ is also an inversion of~$w$.
The rank two parabolic subgroup  $W'=a_1\cdots a_iW_{\set{a_{i+1},s}}a_i\cdots a_1$ contains $r_i$, $r_{i+1}$ and $t_{i+1}$.
The word $a_1\cdots a_{i+1}$ is reduced because it is a prefix of a reduced word. 
The word $a_1\cdots a_i$ is reduced for the same reason and thus $a_1\cdots a_is$ is reduced because the letter $s$ does not occur in $a_1\cdots a_i$.
Thus $a_1\cdots a_i$ is a minimal length coset representative for $a_1\cdots a_iW_{\set{a_{i+1},s}}$, so the canonical generators of $W'$ are $t_{i+1}$ and $r_i$ by Proposition~\ref{can gen}.
If $t_{i+1}$ and $r_i$ commute then $a_{i+1}$ and $s$ commute, so that the claim holds because $r_i=r_{i+1}$.
We may thus assume that $t_{i+1}$ and $r_i$ do not commute.
In particular $r_{i+1}=t_{i+1}r_it_{i+1}$ is distinct from $t_{i+1}$ and $r_i$.

Since $a_1\cdots a_{i+1}$ is a reduced word for an element of $W_{\br{s}}$, by Lemma~\ref{span inv s} we conclude that $\beta_{r_{i+1}}$ is in the positive span of $\alpha_s$ and $\set{\beta_{t_l}:l\le i+1}$, say $\beta_{r_{i+1}}=a \alpha_s + \sum_{l=1}^{i+1} b_l \beta_{t_l}$.  
Thus 
\[\omega_c(\beta_{t_{i+1}}, \beta_{r_{i+1}}) = a \omega_c(\beta_{t_{i+1}}, \alpha_s) + \sum_{l=1}^{i+1} b_l \omega_{c}(\beta_{t_{i+1}}, \beta_{t_l}).\]
By Lemma~\ref{OmegaNegativity} and the antisymmetry of $\omega_c$, the first term on the right side is nonpositive.
Since $a_1\cdots a_j$ is a $c$-sorting word for a $c$-sortable element, Proposition~\ref{InversionOrdering} and the antisymmetry of $\omega_c$ imply that the remaining terms on the right side are also nonpositive.
Furthermore, if every term of the right side is zero then $t_{i+1}$ commutes with each $t_l$ with $l\le i$, and $s$ commutes with $t_{i+1}$.
If $t_{i+1}$ commutes with each $t_l$ with $l\le i$ then one easily concludes that $a_{i+1}$ commutes with each $a_l$ with $l\le i$, and that $t_{i+1}=a_{i+1}$.
If in addition $s$ commutes with $t_{i+1}=a_{i+1}$ then $t_{i+1}$ and $r_i$ commute, contradicting our assumption. 
Thus $\omega_c(\beta_{t_{i+1}}, \beta_{r_{i+1}}) < 0$.
Applying Proposition~\ref{OmegaCyclic}, we conclude that $\omega_c(\beta_{t_{i+1}}, \beta_{r_i}) < 0$.
Since~$w$ is $c$-aligned and its inversion set contains $t_{i+1}$ and $r_{i+1}$, the inversion set of~$w$ must also contain $r_i$.
This proves the claim.

The claim for $i=0$ is the statement that $s\le w$, contradicting the assumption that $s\not\le w$.
This contradiction proves that $w\in W_{\br{s}}$.
It is immediate by Lemmas~\ref{int} and~\ref{OmegaRestriction} that~$w$ is $sc$-aligned as an element of $W_{\br{s}}$.
By induction on rank,~$w$ is $sc$-sortable, and thus is $c$-sortable by Proposition~\ref{CSortRecursive}.
\end{proof}
% 
% If $W$ is finite then the longest element $w_0$ of $W$ is obviously $c$-aligned for every $c$.
% So we have the following corollary to Theorem~\ref{sortable is aligned}.
% (Cf. \cite[Corollary~4.4]{sortable}.)
% \begin{cor}\label{w0 sort}
% If~$W$ is finite then $w_0$ is $c$-sortable for any Coxeter element~$c$.
% \end{cor}

\begin{remark}
A consequence of Theorem~\ref{sortable is aligned} is that, if $w$ is any $c$-aligned element, then we can order $\inv(w)$ as $t_1$, $t_2$, \dots $t_k$ such that $\omega(t_i, t_j) \geq 0$ whenever \mbox{$i < j$}. 
%Conversely, if we knew that $\inv(w)$ could be so ordered, we could establish Theorem~\ref{sortable is aligned} as follows: 
If we knew that $\inv(w)$ for could be so ordered, even when $w$ is not $c$-sortable, then we could establish Theorem~\ref{sortable is aligned} as follows: 
By Lemma~\ref{dyer lem}, $t_1$, $t_2$, \dots $t_k$ is the reflection sequence of some reduced word for $w$.
(This step uses the hypothesis that $w$ is aligned.)
By Proposition~\ref{InversionOrdering}, that reduced word is the sorting word for a sortable element. 
Unfortunately, $\inv(w)$ cannot always be so ordered, although counter-examples are quite large and occur only in infinite groups.
\end{remark}

\begin{remark}\label{rule out 2}
When $W$ is finite, case (2) in the definition of $c$-alignment cannot happen.
(Indeed, if case (2) occurred, then the present definition of alignment would not agree with the definition in \cite[Theorem~4.1]{sortable}.)
One way to see that case (2) cannot occur when $W$ is finite is to appeal to Corollary~\ref{w0 sort}, which says that $w_0$ is $c$-sortable, and thus $c$-aligned by Theorem~\ref{sortable is aligned}.
Since $\inv(w_0)$ is all of $T$, in particular $\inv(w_0)\cap W'$ cannot be the empty set or a singleton.
\end{remark}

\section{Forced and unforced skips}\label{skip sec}

In this section, we define a set $C_c(v)$ of $n$ roots associated to each sortable element $v$, and prove several key properties of $C_c(v)$.
In Section~\ref{pidown sec} we define and study the $c$-Cambrian cone $\Cone_c(v)$ whose inward normals are $C_c(v)$.  
The main results of the present section are characterizations of $C_c(v)$ in terms of ``skips'' in a $c$-sorting word for $v$ and in terms of cover reflections of $v$.

Let $v$ be $c$-sortable and let $r\in S$ be a simple reflection. 
We define a root $C^r_c(v)$ inductively as follows.
For~$s$ initial in $c$,
\[C_c^r(v)=\left\lbrace\begin{array}{ll}
\alpha_r&\mbox{if } v\not\ge s\mbox{ and }r=s\\
C_{sc}^r(v)&\mbox{if } v\not\ge s\mbox{ and }r\neq s\\
s C^r_{scs}(sv)&\mbox{if } v\ge s
\end{array}\right.\]
In the second case listed above, the root $C_{sc}^r(v)$ is defined in $W_{\br{s}}$ by induction on rank and in the third case, $C_{scs}(sv)$ is defined by induction on length.
Implicitly, this recursive description depends on a choice of initial element for $c$, but it is easy to check that $C^r_c(v)$ is, in fact, well defined. 
See Lemma~\ref{PidownWellDef} for a detailed discussion of a similar issue.

Let $C_c(v)=\set{C_c^r(v):r\in S}$.  
In other words, 
\[C_c(v)=\left\lbrace\begin{array}{ll}
C_{sc}(v)\cup\set{\alpha_s}&\mbox{if } v\not\ge s\\
s C_{scs}(sv)&\mbox{if } v\ge s
\end{array}\right.\]
An easy inductive argument shows that $C_c(v)$ is a basis of $V$ and, in the crystallographic case, of the root lattice. 

We define $A_c(v)=C_c(v) \cap \Phi^-$ and $B_c(v)=C_c(v) \cap \Phi^+$.
We now give a combinatorial description of $A_c(v)$ and $B_c(v)$ in terms of a $c$-sorting word for $v$.

Let~$w$ be an element of~$W,$ let $s_1\cdots s_n$ be a reduced word for $c$, let $a_1\cdots a_k$ be the $(s_1\cdots s_n)$-sorting word for~$w$ and let $r\in S$.
Recall the definition of the $(s_1\cdots s_n)$-sorting word as the leftmost subword of $(s_1\cdots s_n)^\infty$ which is a reduced word for~$w$.
Since~$w$ is of finite length, there is a leftmost instance of $r$ in $(s_1\cdots s_n)^{\infty}$ which is not in that subword.
Let $i$ be such that this leftmost instance of $r$ in $(s_1\cdots s_n)^\infty$ occurs (unless $i=0$) after the location of $a_i$ and (unless $i=k$) before the location of $a_{i+1}$.
Say~$a_1\cdots a_k$ \newword{skips} $r$ in the $(i+1)\st$ position.
If $a_1\cdots a_ir$ is reduced then this is an \newword{unforced skip}; 
otherwise it is a \newword{forced skip}.

For any $w\in W$, define 
\begin{eqnarray*}
\fs_c(w)\!\!\!\!&=&\!\!\!\!\set{a_1\cdots a_i r a_i\cdots a_1:a_1\cdots a_k\mbox{ has a forced skip of $r$ in position $i+1$}}\\
\ufs_c(w)\!\!\!\!&=&\!\!\!\!\set{a_1\cdots a_i r a_i\cdots a_1:a_1\cdots a_k\mbox{ has an unforced skip of $r$ in position $i+1$}}.
\end{eqnarray*}
The sets $\fs_c(w)$ and $\ufs_c(w)$ are easily seen to be independent of the choice of reduced word $s_1\cdots s_n$ for~$c$, but do, of course, depend on the choice of $c$. 

The main results of this section are the following four propositions.
For examples illustrating these results, see Section~\ref{aff sec}.
Recall that $\cov(w)$ stands for the set of cover reflections of an element $w\in W$.  

\begin{prop}\label{walls}
Let $v$ be $c$-sortable with $c$-sorting word $a_1\cdots a_k$ and let $r\in S$.
Let $a_1\cdots a_k$ skip $r$ in the $(i+1)\st$ position and let $t=a_1\cdots a_ira_i\cdots a_1$.
Then $C_c^r(v)=\pm\beta_t$.
Furthermore
\begin{eqnarray*}
A_c(v)	&=&\set{C_c^r(v):\mbox{the skip of $r$ in $a_1\cdots a_k$ is forced}}\\
		&=&\set{-\beta_t:t\in\fs_c(v)},\mbox{ and}\\
B_c(v)	&=&\set{C_c^r(v):\mbox{the skip of $r$ in $a_1\cdots a_k$ is unforced}}\\
		&=&\set{\beta_t:t\in\ufs_c(v)}
\end{eqnarray*}
\end{prop}

\begin{prop}\label{lower walls}
Let $v$ be $c$-sortable.
Then $A_c(v)=\set{-\beta_t:t\in\cov(v)}$.
\end{prop}

\begin{prop}\label{Cc final}
Let~$s$ be final in~$c$ and let $v$ be $c$-sortable with $v \geq s$.
Then
\begin{enumerate}
\item[(i) ]$v=s\join v_{\br{s}}$,
\item[(ii) ]$\cov(v)=\set{s}\cup\cov(v_{\br{s}})$, and
\item[(iii) ]$\ufs_c(v)=\ufs_{cs}(v_{\br{s}})$.
\end{enumerate}
\end{prop}

\begin{prop}\label{Cc initial}
Let~$s$ be initial in~$c$ and let $v$ be a $c$-sortable element of~$W$ such that~$s$ is a cover reflection of $v$.
Then 
\begin{enumerate}
\item[(i) ]$v=s\join v_{\br{s}}$,
\item[(ii) ]$\cov(v)=\set{s}\cup\cov(v_{\br{s}})$, and
\item[(iii) ]$\ufs_c(v)=\set{sts:t\in\ufs_{sc}(v_{\br{s}})}$.
\end{enumerate}
\end{prop}

\begin{example} \label{skip example}
Let $W$ be $\tilde{A}_2$, the rank three Coxeter group with $S=\set{p,q,r}$ and $m(p,q)=m(p,r)=m(q,r)=3$. 
Take $c=pqr$ and consider $v=pqrpr$. 
This word is reduced, so $v$ is $c$-sortable. 
The letter $q$ is skipped in the $5$th position, the letters $p$ and $r$ in the $6$th. 
So Proposition~\ref{walls} states that $C_c^q(v) = pqrp \alpha_q = \beta_{pqrpqprqp}$, $C_c^p(v) = pqrpr \alpha_p = - \beta_{pqrqp}$ and $C_c^r(v) = pqrpr \alpha_r = - \beta_q$. 
The reader who checks this from the recursive definition of $C_v$ will quickly see why Proposition~\ref{walls} is true. 
The reflection $pqrpqprqp$ is an unforced skip, while $pqrqp$ and $q$ are forced. 
Proposition~\ref{lower walls} states that $pqrqp$ and $q$ are the cover reflections of $v$.
\end{example}

The proofs of the latter two results require the Euler form, and thus a choice of symmetrizable Cartan matrix.
However, the results themselves are strictly combinatorial, not geometric. 
These combinatorial results restrict the geometry of those cones of the Cambrian fan which have a wall contained in $H_s$,
as explained in Proposition~\ref{recursive fan final} and Section~\ref{rank3sec}. 
See Figures~\ref{bij1} and~\ref{bij2} and the surrounding discussion.
We first prepare to prove Propositions~\ref{walls} and~\ref{lower walls} by proving some preliminary lemmas.
\begin{lemma}\label{AB inv}
If $-\beta_t \in A_c(v)$ then~$t$ is an inversion of $v$.
If $\beta_t \in B_c(v)$ then~$t$ is not an inversion of $v$.
\end{lemma}

\begin{proof}
We prove the equivalent statement
\renewcommand{\theequation}{$\ast \ast$}
\begin{equation}\label{C cap inv}
C_c(v)\cap\set{\pm\beta_t:t\in\inv(v)}=A_c(v). 
\end{equation}

Let~$s$ be initial in~$c$. 
If $v\not\ge s$ then $C_c(v)=C_{sc}(v)\cup\set{\alpha_s}$.
By induction on rank, $C_{sc}(v)\cap\set{\pm\beta_t:t\in\inv(v)}=A_{sc}(v)$.
This completes the argument in this case since $s\not\in\inv(v)$ and $A_{sc}(v)=A_c(v)$.

If $v\ge s$ then $C_c(v)=sC_{scs}(sv)$ and $\inv(sv)=(s\inv(v)s)\setminus\set{s}$.
By induction on length, $C_{scs}(sv)\cap\set{\pm\beta_t:t\in\inv(sv)}=A_{scs}(sv)$. 
Thus\footnote{Note that the use of the notation $E \setminus F$ does not imply $F \subset E$. Here and throughout the paper, the definition of $E \setminus F$ is $\{ x \in E : x \not \in F \}$.} 
\[C_c(v)\cap\set{\pm\beta_t:t\in\inv(v), t\neq s}=sA_{scs}(sv)=A_c(v) \setminus \{ - \alpha_s \}. \]

So we just need to consider whether $\pm \alpha_s$ could be contained in either side of equation~(\ref{C cap inv}).  If $- \alpha_s\in C_c(v)$ then $A_c(v)=sA_{scs}(sv)\cup\set{- \alpha_s}$; otherwise $A_c(v)=sA_{scs}(sv)$.
So $- \alpha_s$ is contained in the right hand side of equation~(\ref{C cap inv}) if and only if it is contained in the left hand side. 
We now check that $\alpha_s$ cannot lie in the left hand side; by definition, $\alpha_s$ cannot belong to $A_c(v)$.  
If $\alpha_s$ was in $C_{c}(v)$, then $-\alpha_s$ would be in $C_{scs}(sv)$ and hence $-\alpha_s \in A_{scs}(sv)$. 
So, by induction on length, $s$ would be an inversion of $sv$, which contradicts our assumption that $v \geq s$.
\end{proof}

\begin{lemma}\label{fs ufs}
Let $w\in W$ have $c$-sorting word $a_1\cdots a_k$.
If~$a_1\cdots a_k$ skips $r$ in the $(i+1)\st$ position then the skip is forced if and only if $a_1\cdots a_ira_i\cdots a_1\in\inv(w)$.
\end{lemma}
\begin{proof}
Suppose~$a_1\cdots a_k$ skips $r$ in the $(i+1)\st$ position.
Set $t=a_1\cdots a_ira_i\cdots a_1$ and $t_j=a_1\cdots a_j\cdots a_1$ for each~$j$ with $1 \leq j \leq k$.
Since $a_1\cdots a_i$ is reduced, the word $a_1\cdots a_ir$ is reduced if and only if $t\neq t_j$ for every~$j$ with $1\le j\le i$.
However, if $t=t_j$ for some $i<j\le k$ then skipping $r$ in the $(i+1)\st$ position violates the definition of a $c$-sorting word.
Thus $a_1\cdots a_ir$ is reduced (and thus the skip is unforced) if and only if $a_1\cdots a_ira_i\cdots a_1\not\in\inv(w)$.
\end{proof}

\begin{lemma}\label{unforced basics}
Let $v$ be a $c$-sortable element of~$W$.
Suppose a $c$-sorting word $a_1\cdots a_k$ for $v$ has an unforced skip of the generator $r$ at the $(i+1)\st$ position.
Let $t=a_1\cdots a_ira_i\cdots a_1$ and for each $j\in[k]$, define $t_j=a_1\cdots a_j\cdots a_1$. Then
\begin{enumerate}
\item$a_1\cdots a_ir$ is a $c$-sorting word for a $c$-sortable element.
\item$r\neq a_j$ for every~$j$ with $i<j\le k$.
\item$\omega_c(\beta_t,\beta_{t_j}) \geq 0$ for every~$j$ with $i<j\le k$, with strict equality holding unless~$t$ and $t_j$ commute.
\end{enumerate}
\end{lemma}
\begin{proof}
Let $c=s_1 s_2 \cdots s_n$ and let $a_1 a_2 \cdots a_k$ be the $(s_1 \cdots s_n)$-sorting word for $v$. 

The first two assertions are immediate from the relevant definitions.
For the last assertion, we argue by induction on length and rank.
Let~$s=s_1$.
First suppose $v\not\ge s$.  
If $r=s$ then Lemma~\ref{OmegaNegativity} proves the assertion.
If not then $v\in W_{\br{s}}$ is $sc$-sortable, $a_1\cdots a_k$ is an $sc$-sorting word and the assertion follows by induction on rank and by Lemma~\ref{OmegaRestriction}.

On the other hand, if $v\ge s$ then in particular $a_1=s$ so an $scs$-sorting word $a_2\cdots a_k$ for $sv$ has an unforced skip of $r$ at position $i$.
By induction of length for $t'=a_2\cdots a_ira_i\cdots a_2$ and $t'_j=a_2\cdots a_j\cdots a_2$ we have $\omega_{scs}(\beta_{t'},\beta_{t'_j}) \geq 0$ for all~$j$ with $i<j\le k$, with strict equality holding unless~$t$ and $t_j$ commute.
Now apply Lemma~\ref{OmegaInvariance}.  
\end{proof}

We now prove Propositions~\ref{walls} and~\ref{lower walls}.
\begin{proof}[Proof of Proposition~\ref{walls}]
We will verify the statement $C_c^r(v)=\pm\beta_t$ by induction on length and rank.
The statements about $A_c$ and $B_c$ follow by Lemmas~\ref{AB inv} and~\ref{fs ufs}.
Let $a_1\cdots a_k$ be the $(s_1\cdots s_n)$-sorting word for $v$ and let $s=s_1$.

If $v\not\ge s$ then consider the cases $r=s$ and $r\neq s$.
If $r=s$ then $t=s$ and $C_c^s(v)=\alpha_s$.
If $r\neq s$ then $C_c^r(v)=C_{sc}^r(v)$ which equals $\pm\beta_t$ by induction on rank.

If $v\ge s$ then $C_c^r(v)=sC_{scs}^r(sv)$.
An $scs$-sorting word for $sv$ can be obtained from $a_1\cdots a_k$ by deleting the initial $s=a_1$.
Thus $sv$ skips $r$ in the $i\th$ position and $a_2\cdots a_ira_i\cdots a_2=sts$.
By induction on the length of $v$, $C_{scs}^r(sv)=\pm\beta_{sts}$, so $C_c(v)=s\left(\pm\beta_{sts}\right)=\pm\beta_t$.
\end{proof}

\begin{proof}[Proof of Proposition~\ref{lower walls}]
Let $a_1\cdots a_k$ be the $(s_1\cdots s_n)$-sorting word for $v$ and let $s=s_1$.
We split into several cases.

\textbf{Case 1:} $v\not\ge s$. Then $v\in W_{\br{s}}$ and $C_c(v)=C_{sc}(v)\cup\set{\alpha_s}$, so $A_c(v)=A_{sc}(v)$. By induction on rank, $A_{sc}(v)=\set{- \beta_t:t\in\cov(v)}$.

\textbf{Subcase 2a:} $v\ge s$ and~$s$ is a cover reflection of $v$. In this case $\cov(sv)$ is equal to $\set{sts:t\in(\cov(v)\setminus\set{s})}$. We have $A_{scs}(sv)=\set{-\beta_{sts}:t\in(\cov(v)\setminus\set{s})}$ by induction on length.  
Because $C_c(v)=sC_{scs}(sv)$ and since only $\pm\alpha_s$ have their positive/negative status changed by the action of the reflection~$s$, either $A_c(v)=sA_{scs}(sv)$ or $A_c(v)=\set{-\alpha_s}\cup sA_{scs}(sv)$.
To finish the proof in this case, it remains to show that the latter case holds, or in other words that $\alpha_s\in C_{scs}(sv)$.
By Proposition~\ref{walls}, this is equivalent to showing that there is an unforced skip of some~$r$ in the $i\th$ position of $a_2\cdots a_k$ with $s=a_2\cdots a_ira_i\cdots a_2$.

Consider the $(s_2\cdots s_ns_1)$-sorting word $a'_1\cdots a'_k$ for $v$. 
Take $i$ to be the index such that $s=a'_1\cdots a'_i\cdots a'_1$.
Since~$s$ is a cover reflection of $v$, $\inv(v)=\set{s}\cup\inv(sv)$.
Thus $a'_1,\ldots,a'_{i-1}$ are the first $i-1$ letters of the $(s_2\cdots s_ns_1)$-sorting word for $sv$.
By the uniqueness of sorting words, the words $a'_1,\ldots,a'_{i-1}$ and $a_2,\ldots,a_i$ are equal.
Since~$s$ is not an inversion of $sv$, $a'_i$ is skipped at position $i$ in $a_2\cdots a_k$.
This is an unforced skip, since $a'_1\cdots a'_i=a_2\cdots a_ia'_i$ is reduced.
This completes the proof in the case where $v\ge s$ and~$s$ is a cover reflection of $v$.

\textbf{Subcase 2b:} $v\ge s$ and~$s$ is not a cover reflection of $v$. Then $\cov(sv)=\set{sts:t\in\cov(v)}$.
By induction on length, $A_{scs}(sv)=\set{-\beta_{sts}:t\in\cov(v)}$.
In this case we want to show that $A_c=sA_{scs}(sv)$, or equivalently that $\alpha_s\not\in C_{scs}(sv)$.
Appealing again to Proposition~\ref{walls}, we complete the proof by showing that there is no unforced skip in $a_2\cdots a_k$ whose corresponding reflection is~$s$.
Suppose to the contrary that that there is an unforced skip of some $r$ in the $i\th$ position of $a_2\cdots a_k$ with $s=a_2\cdots a_ira_i\cdots a_2$.
Define $t'_j=a_2\cdots a_j\cdots a_2$ for every~$j$ with $2\le j\le k$.
By Lemma~\ref{unforced basics}, $\omega_{scs}(\alpha_s,\beta_{t'_j}) \geq 0$ for every~$j$ between $i+1$ and $k$ (inclusive), with strict inequality unless~$s$ and $t'_j$ commute.
But, by Lemma~\ref{OmegaPositivity}, $\omega_{scs}(\alpha_s,\beta_{t}) \leq 0$, for every reflection~$t$, because~$s$ is final in $scs$.  Thus~$s$ and $t'_j$ commute for every~$j$ with $i<j\le k$.
For each such~$j$, let $t''_j=a_{i+1}\cdots a_j\cdots a_{i+1}$.
Since $s=a_2\cdots a_i r a_i \cdots a_2$ commutes with $t'_j=a_2\cdots a_it''_ja_i\cdots a_2$, the reflection $r$ commutes with $t''_j$.
Since the $t''_j$'s are the inversions of $a_{i+1}\cdots a_k$, by Proposition~\ref{commutes}, $r$ commutes with $a_{i+1} \cdots a_k$.
So we have $v=a_1a_2\cdots a_i a_{i+1}\cdots a_k=a_2\cdots a_i r a_{i+1} \cdots a_k=a_2 \cdots  a_i  a_{i+1} \cdots a_k r$. Since $\ell(v)=k$, the word $a_2 \cdots  a_i a_{i+1} \cdots a_k r$ is reduced. So $a_2 \cdots a_k r =v$ covers $a_2 \cdots a_k=sv$, contradicting the assumption that $s$ is not a cover reflection of $v$.
\end{proof}

We now proceed to prove Propositions~\ref{Cc final} and~\ref{Cc initial}, again starting with preliminary lemmas.
The second of the two lemmas is essentially a converse to Lemma~\ref{unforced basics}.

\begin{lemma}\label{Cc order} 
Let $v$ be $c$-sortable, with $c=s_1 s_2 \cdots s_n$.  
Let $\Omega_c(v)$ be the subword of $(s_1 s_2 \cdots s_n)^{\infty}$ which is the complement of the $(s_1 s_2 \cdots s_n)$-sorting word for $v$.
Let $(r_1, r_2, \ldots, r_n)$ be the simple reflections, ordered according to when they first appear in $\Omega_c(v)$. Set $\beta_i=C_c^{r_i}(v)$.

With the above definitions, $E_c(\beta_i, \beta_j)=0$ for $i < j$. 
\end{lemma}
\begin{proof}
We use the usual argument by induction on length and rank. The base case, where $v=e$, is immediate from the definition of $E_c$.

First, suppose that $v \geq s$. 
In particular, the orderings on $S$ obtained from $\Omega_c(v)$ and from $\Omega_{scs}(sv)$ are the same.
Now $E_c(\beta_i, \beta_j)=E_{c}(sC^{r_i}_{scs}(sv), sC^{r_j}_{scs}(sv) )$.
By Lemma~\ref{Ec invariant}, the latter equals $E_{scs}(C^{r_i}_{scs}(sv), C^{r_j}_{scs}(sv) )$ which by induction equals zero for $i<j$.

Now, suppose that $v \not \geq s$, so $v \in W_{\br{s}}$. 
Then $s=r_1$ and $\beta_1=\alpha_s$.
For $i >1$, the root $\beta_i$ lies in $\Phi_{\br{s}}$ and, thus, $E_c(\beta_1, \beta_i)=0$. 
The ordering on $S\setminus\set{s}$ obtained from $\Omega_{sc}(v)$ is the restriction of the ordering obtained from $\Omega_c(v)$.
Thus if $1 < i < j$, then $E_{c}(\beta_i, \beta_j)=0$ by induction on rank.
\end{proof}

\begin{lemma} \label{unforced basics converse}
Let $t_1$, $t_2$, \dots, $t_j$, $t$, $t_{j+1}$, \dots $t_N$ be a sequence of $N+1$ distinct reflections of $W$, and let $v$ be a $c$-sortable element of $W$. Suppose that the following conditions hold.
\begin{enumerate}
\item[(i)] $t_1$, $t_2$, \dots, $t_j$, $t_{j+1}$, \dots $t_N$ is the reflection sequence of a $c$-sorting word for~$v$.
\item[(ii)] $t_1$, $t_2$, \dots, $t_j$, $t$ is the reflection sequence of a reduced word.
\item[(iii) ]$\omega_c(\beta_{t_i},\beta_{t}) \geq 0$ for all $i\le j$, with strict inequality unless $t_i$ and $t$ commute, and
\item[(iv) ]$\omega_c(\beta_{t},\beta_{t_k})\geq 0$ for all $k>j$, with strict inequality unless $t$ and $t_k$ commute.
\end{enumerate}
Then $t$ is an unforced skip of $v$.
\end{lemma}

\begin{proof}
Let $s$ be initial in $c$.
If $v \not \geq s$ then $v\in W_{\br{s}}$ and each $t_i\in W_{\br{s}}$.
Let $v'$ be the element with a reduced word having reflection sequence $t_1,\ldots,t_{j},t$.
By Proposition~\ref{InversionOrdering}, applied twice, we see that $v'$ is $c$-sortable.
We claim that either $t=s$ or $t\in W_{\br{s}}$.
If $t\not\in W_{\br{s}}$ then we have $v'\not\in W_{\br{s}}$, so, since $v'$ is $c$-sortable, $s\le v'$.
But $t$ is the only inversion of $v'$ that is not an inversion of $v$, and $s\not\le v$.
Thus $t=s$, and this proves the claim. 
Since $v \not \geq s$ and $s$ is initial in $c$, we have $s\in\ufs_c(v)$.
Thus, by the claim, it remains only to consider the case $t\in W_{\br{s}}$.
By Lemma~\ref{OmegaRestriction} we see that (iii) and (iv) also hold with respect to $\omega_{sc}$.
Thus by induction on rank, $t\in\ufs_{sc}(v)\subset\ufs_c(v)$.

If $v \geq s$ then $s\in\set{t_1,\ldots,t_N}$, so $t\neq s$.
Since $s$ is initial in $c$, we can transpose initial commuting entries in the sequence $t_1,\ldots,t_N$ to make $s$ appear first in the sequence, without upsetting properties (i)--(iv).
By Lemma~\ref{OmegaInvariance} and Proposition~\ref{CSortRecursive}, the sequence $s t_2 s$, $s t_3 s$, \ldots, $s t_j s$, $s t s$, $s t_{j+1} s$, \ldots, $s t_N s$ satisfies (i)--(iv) with respect to $scs$. 
By induction on $N$ we have $sts\in\ufs_{scs}(sv)$, so $t\in\ufs_c(v)$.
\end{proof}

We will soon turn to the proofs of Propositions~\ref{Cc final} and~\ref{Cc initial}. We will deduce both these propositions from the following result. 

\begin{prop} \label{Cc together}
Let $s$ be either initial or final in $c$. Let $v$ be $c$-sortable with $s\in\cov(v)$ and let $t$ be a reflection with $\pm \beta_t$ in $C_c(v) \setminus \{ - \alpha_s \}$.
\begin{enumerate}
\item[(i) ] If $s$ is final in $c$ or $t\in\cov(v)$ then $t\in W_{\br{s}}$.
\item[(ii) ] If $s$ is initial in $c$ and $t\in\ufs_c(v)$ then $sts\in W_{\br{s}}$.
\end{enumerate}
\end{prop}

By Proposition~\ref{walls}, either $t\in\cov(v)$ or $t\in \ufs_c(v)$, so one of these cases applies.

\begin{proof}
 Proposition~\ref{lower walls} and the fact that $s$ is a cover reflection of $v$ imply that $-\alpha_s\in C_c(v)$.
 By Lemma~\ref{Cc order}, either $E_c(-\alpha_s, \pm \beta_t)=0$ or $E_c(\pm \beta_t, -\alpha_s)=0$. 
 If $E_c(-\alpha_s, \pm \beta_t)=0$ and $s$ is initial in $c$, then $t$ is in $W_{\br{s}}$ by Lemma~\ref{Ec initial}.  
 Similarly, if $E_c(\pm \beta_t, -\alpha_s)=0$ and $s$ is final in $c$, then $t$ is in $W_{\br{s}}$ by Lemma~\ref{Ec final}. 
 If $E_c(\pm \beta_t, -\alpha_s)=0$ and $s$ is initial in $c$ then by Lemma~\ref{Ec invariant},
\[E_{scs}(\pm \beta_{sts}, -\alpha_s)=E_{scs}(\pm s\beta_t,s \alpha_s)=E_c(\pm \beta_t, \alpha_s)=0,\]
so $sts \in W_{\br{s}}$ by Lemma~\ref{Ec final}. 
Similarly, if $E_c(-\alpha_s, \pm \beta_t)=0$ and $s$ is final in $c$ then $sts \in W_{\br{s}}$.
We have shown that either $t$ or $sts$ is in $W_{\br{s}}$.
It remains to establish the more specific claims of (i) and (ii).
If $s$ and $t$ commute, then $t=sts$ and both claims follow immediately, so we may assume that they do not commute. 

Consider the generalized rank two parabolic subgroup $W'$ of $W$ which contains~$s$ and $t$.
Since $s$ and $t$ do not commute, $W'$ is not commutative. 
By Lemmas~\ref{s in} and~\ref{int}, $s$ is a canonical generator of $W'$ and the other canonical generator is in $W_{\br{s}}$. 
Let $p$ be this other generator. We break into cases: 

\textbf{Case 1:} $t$ is a cover reflection of $v$. 
Then $t \in \inv(v)$. 
Since we assumed that $s$ is also a cover reflection of $v$, we have $\inv(v)=\inv(sv) \setminus \{ s \}$ and thus $sts$ is also an inversion of $v$. 
Since either $t$ or $sts$ is $p$, we know that $s$ and $p$ are both inversions of $v$ and thus every reflection of $W'$ is an inversion of $v$ by Lemma~\ref{PilkLemma}. 
Then, by Lemma~\ref{dyer cov}, $t$ must be $p$. 
In particular, $t$ is in $W_{\br{s}}$ as desired.

\textbf{Case 2:} $t$ is an unforced skip of $v$. 
In this case, let $a_1 a_2 \ldots a_k$ be a $c$-sorting word for $v$ and let $t=a_1 a_2\cdots a_i r a_i\cdots a_2 a_1$ for an unforced skip of $r$ at position $(i+1)$ in~$v$. 
Let $v'$ be the $c$-sortable element $a_1 a_2\cdots a_i r$. 
Now, since $s$ is a cover reflection of $v$, we have $\inv(v) = s \inv(v) s$. 
Using the fact that $t$ is not an inversion of $v$, we deduce from Lemma~\ref{PilkLemma} that $\inv(v) \cap W'=\{ s \}$. 
In particular, $p$ and $sps$ are not inversions of $v$. 
Now, $\inv(v') \setminus \{ t \} \subseteq \inv(v)$, so $\left( \inv(v') \cap W' \right) \setminus \{ t \}$ is either empty or $\{ s \}$. 
Applying Lemma~\ref{PilkLemma}, either $\left( \inv(v') \cap W' \right)=\{ p \}=\{ t \}$ or $\left( \inv(v') \cap W' \right)=\{ s, sps \}=\{ s, t \}$. 
If $s$ is initial in $c$ then $a_1=s$ so $s \in \inv(v')$, so $\left( \inv(v') \cap W' \right)=\{ s, sps \}=\{ s, t \}$ and thus $sps=t$, so that $sts=p\in W_{\br{s}}$.
If $s$ is final in $c$ then, since $v'$ is $c$-aligned with respect to $W'$, $\left( \inv(v') \cap W' \right)$ cannot be $\{ s, sps \}$. (Here we use the fact that $W'$ is not commutative.)
Thus $\left( \inv(v') \cap W' \right)=\{ p \}=\{ t \}$.
Now $t=p \in W_{\br{s}}$. 
 \end{proof}

We now prove Propositions~\ref{Cc final} and~\ref{Cc initial}. 

\begin{proof}[Proof of Proposition~\ref{Cc final}]
The assertion that $s\in\cov(v)$ is Lemma~\ref{nc lemma 2}, so by Proposition~\ref{Cc together} we deduce that $t$ is in $W_{\br{s}}$ whenever $\pm \beta_t \in C_c(v) \setminus \{ -\alpha_s \}$.
Thus $\cov(v)\setminus\set{s} \subseteq W_{\br{s}}$ by Lemma~\ref{lower walls}.
Assertion (i) now holds by Lemma~\ref{s join w br s} and assertion (ii) then follows by Lemma~\ref{cov w br s}.

Since there are exactly $n$ reflections in $\cov(v)\cup\ufs_c(v)$ and exactly $n-1$ reflections in $\cov(v_{\br{s}})\cup\ufs_{cs}(v_{\br{s}})$, assertion (ii) implies that the sets $\ufs_c(v)$ and $\ufs_{cs}(v_{\br{s}})$ have the same size.
Thus to prove assertion (iii), it is enough to show that $\ufs_c(v)\subseteq\ufs_{cs}(v_{\br{s}})$.
Suppose $t\in\ufs_c(v)$ and let $a_1\cdots a_{\ell(v)}$ be a $c$-sorting word for $v$ with associated reflection sequence $t_1,\ldots,t_{\ell(v)}$. 
Let $t=a_1\cdots a_jra_j\cdots a_1$ for a skip of~$r$ in the $(j+1)^{\textrm{st}}$ position.
By Lemmas~\ref{AB inv} and~\ref{unforced basics} and Proposition~\ref{InversionOrdering}, the sequence $t_1$, $t_2$, \dots, $t_j$, $t$, $t_{j+1}$, \dots, $t_{\ell(v)}$ satisfies the hypotheses of Lemma~\ref{unforced basics converse}. 
Let $t'_1$, $t'_2$, \ldots, $t'_{j'}$, $t$, $t'_{j'+1}$, \ldots, $t'_{\ell(v_{\br{s}})}$ be the subsequence of $t_1$, $t_2$, \dots, $t_j$, $t$, $t_{j+1}$, \dots, $t_{\ell(v)}$ consisting of those reflections in the larger sequence which lie in $W_{\br{s}}$.
By Proposition~\ref{Cc together}, $t$ is in this subsequence.

Arguing as in the proof of Proposition~\ref{sort para}, we see that $t'_1$, $t'_2$, \dots, $t'_{\ell(v_{\br{s}})}$ is the reflection sequence of a reduced word $a'_1 \ldots a'_{\ell(v_{\br{s}})}$ for $v_{\br{s}}$. 
Furthermore, $a'_1\cdots a'_{\ell(v_{\br{s}})}$ can be converted, by a sequence of transpositions of adjacent commuting letters, to a $cs$-sorting word for $v_{\br{s}}$. 
The sequence $t_1$, $t_2$, \dots, $t_j$, $t$, $t_{j+1}$, \dots, $t_{\ell(v)}$  satisfies the hypotheses of Lemma~\ref{unforced basics converse} so Lemmas~\ref{int}, \ref{PilkLemma} and~\ref{OmegaRestriction} imply that $t'_1$, \dots, $t'_{j'}$, $t$, $t'_{j'+1}$, \dots, $t'_{\ell(v_{\br{s}})}$ satisfies the hypotheses of Lemma~\ref{unforced basics converse}.
Thus $t\in\ufs_{cs}(v_{\br{s}})$ by Lemma~\ref{unforced basics converse}.
\end{proof}

\begin{proof}[Proof of Proposition~\ref{Cc initial}]
 By hypothesis, $s \in \cov(v)$, so Proposition~\ref{Cc together} applies. As in the previous proof, every cover reflection of $v$, besides $s$, is in $W_{\br{s}}$.
Now (i) follows by Lemma~\ref{s join w br s} and (ii) follows from (i) by Lemma~\ref{cov w br s}.

We now prove assertion (iii).
As in the proof of Proposition~\ref{Cc final}, it is enough to show that $\ufs_c(v)\subseteq\set{sts:t\in\ufs_{sc}(v_{\br{s}})}$.
Let $t\in\ufs_c(v)$.
Then $sts\in\ufs_{scs}(sv)$.
Since $sts\in W_{\br{s}}$, we argue, exactly as in the last paragraph of the proof of Proposition~\ref{Cc final}, that $sts\in\ufs_{sc}\left((sv)_{\br{s}}\right)$.

Now, $\inv((sv)_{\br{s}})=\inv(sv) \cap W_{\br{s}}$. 
Since $s$ is a cover reflection of $v$, we have $\inv(sv)=\inv(v) \setminus \{ s \}$ and thus $\inv(sv) \cap W_{\br{s}}=\inv(v) \cap W_{\br{s}}=\inv(v_{\br{s}})$. 
Thus $(sv)_{\br{s}}$ and $v_{\br{s}}$ are equal because $\inv((sv)_{\br{s}})=\inv(v_{\br{s}})$.
So  $sts\in\ufs_{sc}(v_{\br{s}})$, or in other words $t\in\set{sts:t\in\ufs_{sc}(v_{\br{s}})}$. 
\end{proof}

\section{The projection $\pidown^c$}\label{pidown sec}
In this section we define a downward projection $\pidown^c$ from $W$ to the $c$-sortable elements of $W$.
Let~$c$ be a Coxeter element of~$W$ and let~$s$ be initial in~$c$. 
Then, for each $w \in W$, let
\[\pidown^c(w)=\left\lbrace\begin{array}{ll}
s \pidown^{scs}(sw)&\mbox{ if }w \geq s\\
\pidown^{sc}(w_{\br{s}})&\mbox{ if }w \not \geq s.
\end{array}\right.\]
This definition is recursive, by the usual induction on length and rank. 
As a base for the recursion, set $\pidown^c(e)=e$ for any~$c$ in any~$W$. 
Recall that by convention, the map $\pidown^{sc}$ is a map defined on the parabolic subgroup $W_{\br{s}}$. 

After establishing that $\pidown^c$ is well-defined and presenting its most basic properties, we prove two major theorems.
First:

\begin{theorem} \label{order preserving}
$\pidown^c$ is order preserving.
\end{theorem}

As an immediate consequence of Theorem~\ref{order preserving} and the basic properties of $\pidown^c$, we have the following corollary, which makes the connection between the description of $\pidown^c$ in Section~\ref{Summary} and the recursive definition above. 

\begin{cor}\label{max}
Let $w\in W$.
Then $\pidown^c(w)$ is the unique maximal $c$-sortable element below~$w$ in the weak order.
\end{cor}

The other main result of this section is Theorem~\ref{pidown fibers}, which gives a geometric description of the fibers of $\pidown^c$. 
Define a cone  
\[\Cone_c(v)=\bigcap_{\beta \in C_c(v)}\set{x^*\in V^*:\br{x^*,\beta} \geq 0}\]
and recall that $D$ is the cone $\bigcap_{s \in S} \set{x^*\in V^*: \br{x^*,\alpha_s}\ge 0}$.
The definition of $\Cone_c(v)$, like the definition of $D$, depends on the choice of symmetrizable Cartan matrix~$A$ for~$W$.
Since each $C_c(v)$ is a basis for $V$, each cone $\Cone_c(v)$ is a full-dimensional simplicial cone.

\begin{theorem} \label{pidown fibers}
Let $v$ be $c$-sortable. Then $\pidown^c(w)=v$ if and only if $w D \subseteq \Cone_c(v)$.
\end{theorem}

The condition $wD \subseteq \Cone_c(v)$ is equivalent to the following:
$t \in \inv(w)$ for every $-\beta_t \in A_c(v)$ and $t \not\in \inv(w)$ for every $\beta_t \in B_c(v)$. 
Less formally (and as an explanation of the use of the symbols~$A$ and $B$), $\pidown^c(w)$ is $v$ if and only if $wD$ is \textbf{above} the hyperplane orthogonal to $\beta$ for every $\beta \in A_c(v)$ and \textbf{below} the hyperplane orthogonal to $\beta$ for every $\beta \in B_c(v)$. 
Propositions~\ref{walls} and~\ref{lower walls} allow Theorem~\ref{pidown fibers} to be restated as follows:
\begin{theorem}
\label{comb classes}
Let $w\in W$ and let $v$ be $c$-sortable.
Then $\pidown^c(w)=v$ if and only if both of the following two conditions hold:
\begin{enumerate}
\item[(i) ]$\cov(v)\subseteq\inv(w)$ and 
\item[(ii) ]$\ufs_c(v)\cap\inv(w)=\emptyset$.
\end{enumerate}
\end{theorem}

\begin{example} \label{walls example}
The computations of Example~\ref{skip example} show that, in the notation of that example, $\pidown^c(w)=v$ if and only if $q$ and $pqrqp$ are inversions of $w$ and $pqrpqprqp$ is not. See the discussion of Figures~\ref{bij1} and~\ref{bij2} for more examples and geometric discussion of Theorem~\ref{pidown fibers}.
\end{example}

We now establish the basic properties of $\pidown^c$, beginning with well-definition.
Calculating $\pidown^c$ recursively involves the choice of a sequence of simple generators.
\emph{A priori} it appears that the result might depend on the choice of this sequence.
In fact, the result is independent of the sequence chosen:
\begin{lemma}\label{PidownWellDef}
For any $w\in W$, the element $\pidown^c(w)$ is independent of which choice is made of an initial element in each step of the recursion.
\end{lemma}
\begin{proof}
If~$c$ has only one initial element then the result follows by induction on length or rank.
Otherwise, let~$s$ and $s'$ be two distinct initial elements of $c$, so that in particular~$s$ and $s'$ commute.
The proof breaks into four cases according to whether~$s$ and $s'$ are below~$w$ in weak order.
In each case we use~$s$ and $s'$ to calculate $\pidown^c(w)$ in two ways.
Each of these involves an expression which, by induction on length or rank, is independent of the subsequent choices of initial letters.
Thus in every case, it is sufficient to show: first, that $\pidown^c(w)$ can be evaluated by choosing~$s$ then $s'$ or by choosing $s'$ then~$s$; and second, that the two expressions for $\pidown^c(w)$ obtained in this manner coincide.

If $w\not\ge s$ and $w\not\ge s'$ then the two possibilities for $\pidown^c(w)$ are $\pidown^{sc}(w_{\br{s}})$ and $\pidown^{s'c}(w_{\br{s'}})$.
Since~$s$ and $s'$ commute, $s'$ is initial in $sc$ and~$s$ is initial in $s'c$.
Since $w_{\br{s}}\le w$ we have $w_{\br{s}}\not\ge s'$ and similarly $w_{\br{s'}}\not\ge s$.
Thus $\pidown^{sc}(w_{\br{s}})=\pidown^{s'sc}(w_{\br{s,s'}})=\pidown^{s'c}(w_{\br{s'}})$.
(Here the notation $\br{s,s'}$ stands for $S\setminus\set{s,s'}$.)

If~$w$ is above exactly one of~$s$ and $s'$ then we may as well take $w\ge s$ and $w\not\ge s'$.
The two possibilities for $\pidown^c(w)$ are $s\pidown^{scs}(sw)$ and $\pidown^{s'c}(w_{\br{s'}})$.
Again, since~$s$ and $s'$ commute, $s'$ is initial in $scs$ and~$s$ is initial in $s'c$.
Since the map $x\mapsto x_{\br{s'}}$ is order preserving and $s_{\br{s'}}=s$, we have $w_{\br{s'}}\ge s$.
Since $\inv(sw)=s(\inv(w))s\setminus\set{s}$ and $ss's=s'$, the fact that $w\not\ge s'$ implies that $sw\not\ge s'$.
Thus $\pidown^{scs}(sw)=\pidown^{s'scs}\left((sw)_{\br{s'}}\right)$ and $\pidown^{s'c}(w_{\br{s'}})=s\pidown^{ss'cs}\left(s(w_{\br{s'}})\right)$.
We have $s'scs=ss'cs$ and we compare inversion sets to see that $(sw)_{\br{s'}}=s(w_{\br{s'}})$.
Specifically, 
\[\inv\left((sw)_{\br{s'}}\right)=\inv(sw)\cap W_{\br{s'}}=[s(\inv(w))s\setminus\set{s}]\cap W_{\br{s'}}\]
which is equal to
\[[s(\inv(w))s\cap W_{\br{s'}}]\setminus\set{s}=[s(\inv(w)\cap W_{\br{s'}})s]\setminus\set{s}=\inv\left(s(w_{\br{s'}})\right)\]

If $w\ge s$ and $w\ge s'$ then the two possibilities for $\pidown^c(w)$ are $s\pidown^{scs}(sw)$ and $s'\pidown^{s'cs'}(s'w)$.
Again, $s'$ is initial in $scs$,~$s$ is initial in $s'cs'$, $s'w\ge s$ and $sw\ge s'$.
Thus the two expressions for $\pidown^c(w)$ are $s's\pidown^{s'scss'}(s'sw)$ and $ss'\pidown^{ss'cs's}(ss'w)$, which are equal because~$s$ and $s'$ commute.
\end{proof}

We note the following consequences of the definition of $\pidown^c$.
\begin{prop} \label{decreasing}
The element $\pidown^c(w)$ is $c$-sortable for any $w \in W$. We have $\pidown^c(w) \leq w$, with equality if and only if~$w$ is $c$-sortable. 
\end{prop}
\begin{proof}
Let~$s$ be initial in~$c$.
We again argue by induction on length and rank. 
If $w\ge s$ then $\pidown^c(w)=s \pidown^{scs}(sw)$. 
By induction on length, $\pidown^{scs}(sw) \le sw$, with equality if and only if $sw$ is $scs$-sortable.
Now Propositions~\ref{AboveBelow} and~\ref{CSortRecursive} imply that $\pidown^c(w) \leq w$, with equality if and only if~$w$ is $c$-sortable. Also, by induction on length, $\pidown^{scs}(sw)$ is $scs$-sortable so, by Proposition~\ref{CSortRecursive}, $s\pidown^{scs}(sw)=\pidown^c(w)$ is $c$-sortable.

If $w\not\ge s$ then $\pidown^c(w)=\pidown^{sc}(w_{\br{s}})$ and by induction on rank, $\pidown^{sc}(w_{\br{s}}) \leq w_{\br{s}}$ with equality if and only if $w_{\br{s}}$ is $sc$-sortable.
Since $w_{\br{s}}\le w$ with equality if and only if $w\in W_{\br{s}}$, Proposition~\ref{CSortRecursive} shows that $\pidown^c(w) \leq w$, with equality as desired. Also, $\pidown^c(w)=\pidown^{sc}(w_{\br{s}})$ which is, by induction, $sc$-sortable and thus $c$-sortable.
\end{proof}

\begin{prop}\label{idempotent}
The map $\pidown^c$ is idempotent, i.e.\ $\pidown^c \circ \pidown^c=\pidown^c$.
\end{prop}
\begin{proof}
Let $w\in W$.
Then $\pidown^c(w)$ is $c$-sortable by Proposition~\ref{decreasing}.
Thus by Proposition~\ref{decreasing} again, $\pidown^c(\pidown^c(w))=\pidown^c(w)$.
\end{proof}

\begin{prop} \label{AboveS}
Let $s$ be initial in $c$. Then $w \geq s$ if and only if $\pidown^c(w) \geq s$.
\end{prop}

This is actually true without the hypothesis that $s$ is initial, as an immediate corollary of Theorem~\ref{order preserving}.
However, the proof of Theorem~\ref{order preserving} is much harder, and explicitly depends on Proposition~\ref{AboveS}.

\begin{proof}
If $w \not \geq s$ then Proposition~\ref{decreasing} implies that $\pidown^c(w)\not\ge s$. 
If $w \geq s$, Proposition~\ref{AboveBelow} states that $sw \not \geq s$. 
Then $\pidown^{scs}(sw) \not \geq s$ by Proposition~\ref{decreasing}.
Applying Proposition~\ref{AboveBelow} again, we obtain $\pidown^c(w)=s \pidown^{scs}(sw) \geq s$.
\end{proof}

The map $\pidown^c$ is also well-behaved with respect to standard parabolic subgroups.
A very strong statement about $\pidown^c$ and standard parabolic subgroups occurs later as Proposition~\ref{pidown para}. Here is a much simpler result.

\begin{prop} \label{DeletionAndPidown}
Let~$c$ be a Coxeter element of~$W,$ let $J\subseteq S$ and let $c'$ be the restriction of $c$ to $W_J$.
If $w\in W_J$ then $\pidown^c(w)=\pidown^{c'}(w)$.
\end{prop}
Note that the left hand side of this equation involves a map $\pidown^c$ defined on~$W$ whereas the right hand side uses a map $\pidown^{c'}$ defined on $W_J$.
\begin{proof}
It is enough to prove the proposition in the case where $J=S\setminus\set{s}$ for some $s\in S$.

Let $s'$ be initial in~$c$.
We argue as usual by induction on the length of~$w$ and on the rank of~$W$.
If $s'=s$ then the assertion of the proposition is that $\pidown^c(w)=\pidown^{sc}(w)$, which is true by definition because $w=w_{\br{s}}\not\ge s$.
If $s'\neq s$ then consider two cases, $w\ge s'$ and $w\not\ge s'$.
In the case $w\ge s'$, the element $s'w$ is in $W_{\br{s}}$ and $s'c's'$ is the Coxeter element of $W_{\br{s}}$ obtained by deleting~$s$ from $s'cs'$.
Thus induction on length shows that $\pidown^{s'cs'}(s'w)=\pidown^{s'c's'}(s'w)$, so that $\pidown^c(w)=\pidown^{c'}(w)$.
In the case $w\not\ge s'$ we have $\pidown^c(w)=\pidown^{s'c}(w_{\br{s'}})$, which by induction on rank equals $\pidown^{s'c'}(w_{\br{s'}})=\pidown^{c'}(w)$.
\end{proof}

We now proceed to prove Theorems~\ref{order preserving} and~\ref{pidown fibers}, beginning with a special case of Theorem~\ref{pidown fibers}.
This special case will be used to prove Theorem~\ref{order preserving}, which will in turn be used in the proof of the full version of Theorem~\ref{pidown fibers}.
\begin{prop} \label{PartwayToBoundary}
Let $v$ be $c$-sortable and let $v \le w$. Then $\pidown^c(w)=v$ if and only if $w D \subseteq \Cone_c(v)$.
\end{prop}
Since $\pidown^c(w)=v$ implies $w \geq v$, Proposition~\ref{PartwayToBoundary} falls short of Theorem~\ref{pidown fibers} only in the following sense:
The proposition guarantees that $wD\subseteq \Cone_c(\pidown^c(w))$, but \emph{a priori} it is possible that there are additional sortable elements $v$ (not below~$w$) such that $w D \subseteq \Cone_c(v)$.

\begin{proof}
Our proof is by induction on rank and on the length of~$v$.
The base for induction on rank is trivial and the base for induction on length is the observation that $\Cone_c(e)=D$ when $e$ is the identity element of~$W$.
Let~$s$ be initial in~$c$. We again break into several cases.

\textbf{Case 1:} 
$w\not\ge s$ and $v \not \geq s$.
Since $v$ is $c$-sortable and $v\not\ge s$ we have $v\in W_{\br{s}}$.
Furthermore $C_c(v)=C_{sc}(v)\cup\set{\alpha_s}$, so that
\begin{eqnarray*}
\Cone_{c}(v) &= &V^*_{\not \geq s}\,\cap\bigcap_{\beta\in C_{sc}(v)}\set{x^*\in V_{\br{s}}:\br{x^*,\beta} \geq 0}\\
&=&V^*_{\not\geq s}\,\cap(P_{\br{s}})^{-1}(\Cone_{sc}(v)),
\end{eqnarray*}
where $V^*_{\not \geq s}=\set{x^*\in V^*:\br{x^*, \alpha_s} \geq 0}$ and $P_{\br{s}}$ is the projection defined in Section~\ref{Parabolic Subsec}.  

Suppose that $\pidown^c(w)=v$, so $\pidown^{sc}(w_{\br{s}})=v$. 
Since $w \geq v$, we have $w_{\br{s}} \geq v_{\br{s}}=v$ so, by induction on rank, we deduce that $w_{\br{s}} D_{\br{s}} \subseteq \Cone_{sc}(v)$. 
Then $w D \subseteq P_{\br{s}}^{-1} ( w_{\br{s}} D_{\br{s}}) \subseteq P_{\br{s}}^{-1} ( \Cone_{sc}(v))$. 
Also $w D \subseteq V^*_{\not \geq s}$ because $w \not \ge s$. 
So $w D \subseteq   P_{\br{s}}^{-1} ( \Cone_{sc}(v)) \cap  V^*_{\not \geq s}=\Cone_{c}(v)$, as desired.

Conversely, suppose that $w D \subseteq \Cone_{c}(v)$, so $P_{\br{s}}(w D) \subseteq \Cone_{sc}(v)$. 
Now, $\Cone_{sc}(v)$ is a union of cones of the form $xD_{\br{s}}$ for $x\in W_{\br{s}}$.
In particular $\Cone_{sc}(v)$ contains the unique cone $xD_{\br{s}}$ in which $P_{\br{s}}(w D) $ lies, namely the cone $w_{\br{s}} D_{\br{s}}$.
Thus $w_{\br{s}} D_{\br{s}} \subseteq  \Cone_{sc}(v)$ and, by induction on rank, $\pidown^{sc}(w_{\br{s}})=v$. 
But $\pidown^c(w)=\pidown^{sc}(w_{\br{s}})$, so $\pidown^c(w)=v$ as desired.

\textbf{Case 2:} $w\ge s$ and $v \ge s$. 
Since $v \geq s$, the equality $\pidown^c(w)=v$ is equivalent to $\pidown^{scs}(sw)=sv$. 
We have $sv\le sw$ by Proposition~\ref{AboveBelow}.
Thus by induction on the length of~$w$, the equality $\pidown^{scs}(sw)=sv$ is equivalent to $swD\subseteq \Cone_{scs}(sv)$ or, in other words, $wD \subset s \cdot \Cone_{scs}(sv)=\Cone_c(v)$.

\textbf{Case 3:} $w \geq s$ and $v\not\ge s$.
By Proposition~\ref{decreasing}, $\pidown^c(w)\neq v$. By definition, $C_c(v)$ contains the element $\alpha_s$, but interior points $p$ of $wD$ have $\br{p,\alpha_s} < 0$, so that $wD\not\subseteq \Cone_c(v)$.

\textbf{Case 4:} $w \not\geq s$ and $v \geq s$.
This is impossible by the hypothesis $v\le w$.
\end{proof}

Note that the only use of the hypothesis $v\le w$ is to rule out Case 4 and to prove the inequalities (e.g.\ $sv\le sw$ in Case 2) needed to invoke induction.

We are almost ready to prove Theorem~\ref{order preserving}.
The inductive structure of the proof is complex, so we separate out the trickiest part as a lemma. For~$W$ a Coxeter group, $c$ a Coxeter element of~$W$ and $x$ and $y \in W$, let $\Sigma(W,c,x,y)$ denote the statement:

\centerline{``If $y \geq x$ then $\pidown^c(y) \geq \pidown^c(x)$.''}

\noindent Theorem~\ref{order preserving} states that $\Sigma(W,c,x,y)$ is true for all $(W,c,x,y)$. 

\begin{lemma} \label{order preserving key step}
Let~$W$ be a Coxeter group, $c$ a Coxeter element of~$W,$ $s\in S$ and $y \in W$. 
Suppose that both of the following conditions hold:
\begin{enumerate}
\item[(i) ]If the rank of $\hat{W}$ is less then the rank of~$W$ then $\Sigma(\hat{W}, \hat{c}, \hat{x}, \hat{y})$ is true for any Coxeter element $\hat{c}$ and any $\hat{x},\hat{y}\in\hat{W}$.
\item[(ii) ]$\Sigma(W, \hat{c}, \hat{x}, \hat{y})$ is true for any Coxeter element $\hat{c}$ and any $\hat{x},\hat{y}$ with $\ell(\hat{y})<\ell(y)$.
\end{enumerate}
Then $\Sigma(W,c,s,y)$ is true.
\end{lemma}
\begin{proof}
Let $s'$ be initial in $c$.
If  $y \not \geq s$ then the claim is trivial, so assume $y \geq s$. The case $s=s'$ is Proposition~\ref{AboveS}, so we may also assume that $s \neq s'$. 

If $y\not\ge s'$ then $\pidown^{c}(y)=\pidown^{s'c}(y_{\br{s'}})$.
Since $y_{\br{s'}}\ge s_{\br{s'}}=s$, condition (i) guarantees that $\pidown^{s'c}(y_{\br{s'}})\ge s$ and the lemma is proven in this case.

If $y\ge s'$ then $\pidown^{c}(y)=s' \pidown^{s'cs'}(s'y)$.
Furthermore, since $y$ is an upper bound for~$s$ and $s'$, the join $s'\join s$ exists, and $(s'\join s)\le y$.
Thus $s'(s'\join s)\le s'y$ by Proposition~\ref{AboveBelow}.
Since $\ell(s'y)<\ell(y)$, condition (ii) guarantees that $\Sigma(W, s'cs', s' (s' \join s), s' y)$ is true, or in other words, that $\pidown^{s' c s'}(s' y) \geq \pidown^{s' c s'}(s' (s'\join s))$. 
Since $s'$ is final in $s' c s'$, the restriction of $s' c s'$ to the parabolic subgroup generated by $s$ and $s'$ is $ss'$.
Thus Proposition~\ref{DeletionAndPidown} implies that $\pidown^{s'cs'}\left(s'(s'\join s)\right)=\pidown^{ss'}\left(s'(s'\join s)\right)$. 
The unique reduced word for $s'(s'\join s)$ is an initial segment of $ss'ss'ss'\cdots$ and therefore $s'(s'\join s)$ is $ss'$-sortable.
In particular, $\pidown^{ss'}\left(s'(s'\join s)\right)=s'(s'\join s)$.

Combining the inequalities of the previous paragraph, we have $\pidown^{s' c s'}(s' y)  \geq s'(s'\join s)$.  
Proposition~\ref{AboveS} implies that $\pidown^{s' c s'}(s' y)\not\ge s'$, so $s' \pidown^{s' c s'}(s' y) \geq s' \join s$ by Proposition~\ref{AboveBelow}.
 Therefore, $\pidown^{c}(y) =s' \pidown^{s' c s'}(s' y) \geq s' \join s \geq s$.
 \end{proof}

We now prove Theorem~\ref{order preserving}.

\begin{proof}[Proof of Theorem~\ref{order preserving}]
Let $x,y\in W$ with $x \leq y$.
We wish to show that $\pidown^c(x)\le\pidown^c(y)$.
The proof is by induction on rank and on the length of $y$, the rank staying fixed whenever we appeal to induction on length. Note 
that we may immediately reduce to the case that $x \covered y$. 
Let~$s$ be initial in~$c$.
Since $x \leq y$, it is impossible to have $x \geq s$ and $y \not \geq s$.
Thus we must consider three cases.

\textbf{Case 1:} $y\not\ge s$ and $x\not\ge s$. Then $\pidown^c(x)=\pidown^{sc}(x_{\br{s}})$ and $\pidown^c(y)=\pidown^{sc}(y_{\br{s}})$.
Because $x_{\br{s}}$ is covered by or equal to $y_{\br{s}}$, induction on rank shows that $\pidown^c(x)\le\pidown^c(y)$.

\textbf{Case 2:} $x\ge s$ and $y\ge s$. Then $\pidown^c(x)=s\cdot\pidown^{scs}(sx)$ and $\pidown^c(y)=s\cdot\pidown^{scs}(sy)$.
Since $sx\covered sy$, induction on length shows that $\pidown^{scs}(sx)\le\pidown^{scs}(sy)$, and thus $\pidown^c(x)\le\pidown^c(y)$.

\textbf{Case 3:} $x\not\ge s$ and $y\ge s$. 
This case is illustrated in Figure~\ref{OrderLemmaFigure}.
The dashed line separates $W_{\geq s}$ from $W_{\not \geq s}$;
the shaded region is $W_{\br{s}}$;
thick lines indicate cover relations and ordinary lines indicate inequalities. Many of the relations shown are not obvious and will be proved in the course of the proof. 

\begin{figure}[ht] 
\psfrag{a}[ct][cc]{\LARGE$x=sy$}
\psfrag{b}[cc][cc]{\LARGE$y$}
\psfrag{c}[cc][cc]{\LARGE$\pidown^{scs}(x)$}
\psfrag{d}[cc][cc]{\LARGE$\pidown^{scs}(y)=s \pidown^{scs}(x)=\pidown^c(y)$}
\psfrag{e}[cc][cc]{\LARGE$x_{\br{s}}$}
\psfrag{f}[cc][cc]{\LARGE$\begin{matrix} \pidown^c(x)=\pidown^{sc}(x_{\br{s}})\\ =\pidown^{scs}(x_{\br s}) \end{matrix}$}
\psfrag{g}[cc][cc]{\LARGE$W_{\br{s}}$}
\psfrag{h}[cc][cc]{\LARGE$W_{\not \geq s}$}
\psfrag{i}[cc][cc]{\LARGE$W_{\geq s}$}
\scalebox{0.7}{\includegraphics{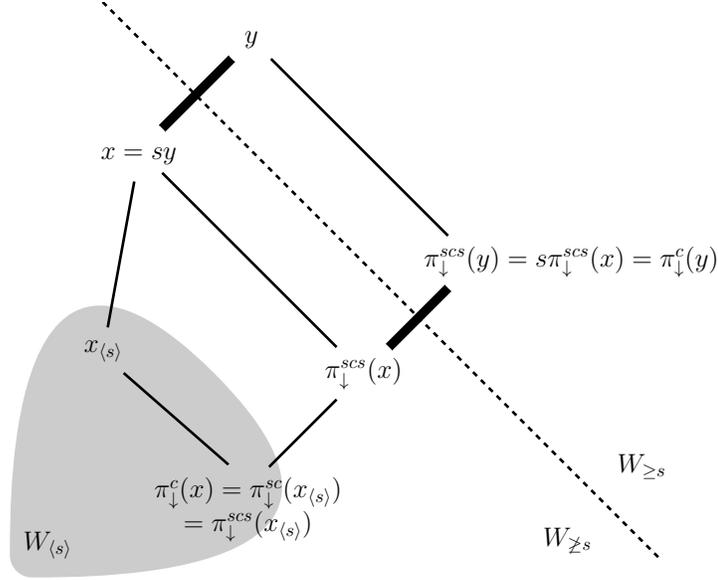}}
\caption{Case 3 of the proof of Theorem~\ref{order preserving}.} \label{OrderLemmaFigure}
\end{figure}

Since $x\covered y$, Proposition~\ref{AboveBelow} states that $x=sy$ and in particular~$s$ is a cover reflection of $y$. 
Condition (i) of Lemma~\ref{order preserving key step} holds by induction on rank, and condition (ii) holds by induction on length, with rank staying fixed.
We conclude that $\Sigma(W, scs, s, y)$ is true, or in other words that $\pidown^{scs}(y)\ge s$.

The fact that $x\not\ge s$ implies, by Proposition~\ref{decreasing}, that $\pidown^{scs}(x)\not\ge s$.
Together with the relation  $\pidown^{scs}(y)\ge s$ from the previous paragraph, this implies that $\pidown^{scs}(x)\neq \pidown^{scs}(y)$.
Since $\pidown^{scs}(x)\le x<y$, we can apply Proposition~\ref{PartwayToBoundary} to conclude that $xD\subseteq \Cone_{scs}(\pidown^{scs}(x))$ but that $yD\not\subseteq \Cone_{scs}(\pidown^{scs}(x))$.
Thus there is some $\beta\in C_{scs}\left(\pidown^{scs}(x)\right)$ such that interior points $p^*$ of $xD$ have $\br{p^*,\beta} > 0$ while interior points $p^*$ of $yD$ have $\br{p^*,\beta} < 0$.
The only hyperplane separating $xD$ from $yD$ is the hyperplane orthogonal to $\alpha_s$.
We therefore conclude that $\alpha_s\in C_{scs}\left(\pidown^{scs}(x)\right)$.
Now $- \alpha_s = s \alpha_s \in C_c\left(s\cdot\pidown^{scs}(x)\right)$ by the definition of $C_c$.
Proposition~\ref{lower walls} now shows that~$s$ is a cover reflection of $s\cdot\pidown^{scs}(x)$, which equals $s\cdot\pidown^{scs}(sy)=\pidown^c(y)$.
We have therefore shown that $\pidown^c(y)\covers s \pidown^c(y) =\pidown^{scs}(sy) = \pidown^{scs}(x)$.
By induction on length, $\pidown^{scs}(x)\ge \pidown^{scs}\left(x_{\br{s}}\right)$ which, according to Proposition~\ref{DeletionAndPidown}, equals $\pidown^{sc}\left(x_{\br{s}}\right)$.
The latter is, by definition, $\pidown^c(x)$.
\end{proof}

Corollary~\ref{max} now follows immediately. 
\begin{proof}[Proof of Corollary~\ref{max}]
By Propositions~\ref{decreasing} and~\ref{idempotent}, the element $\pidown^c(w)$ is $c$-sortable and is below~$w$ in the weak order.
If $v$ is any other $c$-sortable element below~$w$ in the weak order, Theorem~\ref{order preserving} implies that $\pidown^c(v)\le\pidown^c(w)$, and since $v$ is $c$-sortable, $\pidown^c(v)=v$ by the second half of Proposition~\ref{decreasing}.
\end{proof}
Next, we complete the proof of Theorem~\ref{pidown fibers}. 

\begin{proof}[Proof of Theorem~\ref{pidown fibers}]
In the proof of Proposition~\ref{PartwayToBoundary}, the additional hypothesis $v\le w$ was used only to rule out Case 4 and to provide the inequality (analogous to $v\le w$) needed to appeal to induction.
Thus, to prove Theorem~\ref{pidown fibers}, we need only establish Case 4 (assuming inductively that Theorem~\ref{pidown fibers} holds for all examples of lesser length and rank) and copy the proofs of Cases 1, 2 and 3 from the proof of Proposition~\ref{PartwayToBoundary}.

Suppose $v\ge s$ and $w\not\ge s$.
Then $\pidown^c(w)\ne v$ by Proposition~\ref{decreasing}.
Since $sw \geq s$, Theorem~\ref{order preserving} implies that $\pidown^{scs}(sw) \geq \pidown^{scs}(s)=s$. 
In particular, $\pidown^{scs}(sw) \neq sv$. 
Induction on the length of~$v$ establishes that $swD\not\subseteq \Cone_{scs}(sv)=s \Cone_c(v)$. 
Thus $wD \not \subseteq \Cone_c(v)$ as desired.
\end{proof}

We conclude this section with another corollary to Theorem~\ref{order preserving} which considerably strengthens Proposition~\ref{DeletionAndPidown}.

\begin{prop}\label{pidown para}
Let $J\subseteq S$ and let $c'$ be the restriction of $c$ to $W_J$.
Then $\pidown^{c'}(w_J)=\pidown^c(w)_J$ for any $w\in W$.
\end{prop}
\begin{proof}

Since $w\ge w_J$, Theorem~\ref{order preserving} implies that $\pidown^c(w)\ge\pidown^c(w_J)$.
By Proposition~\ref{DeletionAndPidown}, the latter equals $\pidown^{c'}(w_J)$. 
Since the projection $x\mapsto x_J$ is order preserving, $\pidown^c(w)_J \geq \pidown^{c'}(w_J)_J=\pidown^{c'}(w_J)$.

On the other hand, since $w\ge\pidown^c(w)$, we have $w_J\ge (\pidown^{c}(w))_J$. 
Applying the order preserving map $\pidown^{c'}$ to both sides, $\pidown^{c'}(w_J)\ge\pidown^{c'}\left((\pidown^{c}(w))_J \right)$.
But $(\pidown^c(w))_J$ is $c'$-sortable by Propositions~\ref{decreasing} and~\ref{sort para}, so $\pidown^{c'}\left((\pidown^{c}(w))_J\right)=(\pidown^c(w))_J$.
Thus $\pidown^{c'}(w_J)\ge \pidown^c(w)_J$, so that $\pidown^{c'}(w_J) = \pidown^c(w)_J$.
\end{proof}

\begin{samepage}
\addtocontents{toc}{\medskip \textbf{Part II: The Cambrian semilattice and fan} \medskip}
\bigskip
\bigskip
\centerline{\Large{\sc{Part II: The Cambrian semilattice and fan}}}

\section{The Cambrian semilattice} \label{lattice sec}

In this section, we begin to establish the lattice-theoretic properties of $c$-sortable elements. 
We begin by listing and explaining the main results of the section. 
\end{samepage}

\begin{theorem} \label{meets and joins}
Let~$A$ be a collection of $c$-sortable elements of~$W$. 
If~$A$ is nonempty then $\Meet A$ is $c$-sortable. 
If~$A$ has an upper bound then $\Join A$ is $c$-sortable.
\end{theorem}

When~$W$ is finite, $\Join S=w_0$, the longest element of $W$.
Since each $s\in S$ is $c$-sortable, we have the following corollary to Theorem~\ref{meets and joins}.
(Cf. \cite[Corollary~4.4]{sortable}.)
\begin{cor}\label{w0 sort}
If~$W$ is finite then $w_0$ is $c$-sortable for any Coxeter element~$c$.
\end{cor}

The statement about meets in Theorem~\ref{meets and joins} means that the restriction of the weak order to $c$-sortable elements is a sub-(meet-)semilattice of the weak order on all of~$W$. 
This sub-semilattice will be called the \newword{$c$-Cambrian semilattice}.
By Theorem~\ref{order preserving}, for any $w\in W$, the restriction of the weak order to $c$-sortable elements weakly below~$w$ is the interval $[e,\pidown^c(w)]$ in the $c$-Cambrian semilattice.
Theorem~\ref{meets and joins} implies that the interval $[e,\pidown^c(w)]$ in the $c$-Cambrian semilattice is a sublattice of the interval $[e,w]$ in the weak order.
In particular, when~$W$ is finite, the $c$-Cambrian semilattice is called the $c$-Cambrian lattice and is a sublattice of the weak order on~$W$.

Remarkably, the $c$-Cambrian semilattice is not only a subobject of the weak order, but also a quotient object.

\begin{theorem} \label{semilattice hom}
For~$A$ any subset of~$W$, 
if~$A$ is nonempty then $\Meet \pidown^c(A)=\pidown^c(\Meet A)$ and if~$A$ has an upper bound then $\Join \pidown^c(A)=\pidown^c( \Join A)$.
\end{theorem}

Theorem~\ref{semilattice hom} means that $\pidown^c$ is a semilattice homomorphism from the weak order on~$W$ to the $c$-Cambrian semilattice, and that the restriction of $\pidown^c$ to any interval $[1,w]$ is a lattice homomorphism from $[1,w]$ to the interval $[1,\pidown^c(w)]$ in the $c$-Cambrian semilattice.
In particular, when~$W$ is finite, $\pidown^c$ is a lattice homomorphism from the weak order on~$W$ to the $c$-Cambrian lattice. 

As a consequence, the $c$-Cambrian semilattice can also be realized as the meet-semilattice quotient of the weak order on~$W$ modulo the lattice congruence defined by the fibers of $\pidown^c$.
As such, the $c$-Cambrian semilattice is the partial order on the fibers of $\pidown^c$ defined as follows:
A fiber $F_1$ is below a fiber $F_2$ if and only if there exists elements $x_1\in F_1$ and $x_2\in F_2$ with $x_1\le x_2$ in the weak order.
In fact, since each fiber has a unique minimal element, the partial order on fibers can also be defined by a formally weaker requirement:
A fiber $F_1$ is below a fiber $F_2$ if and only if, for every $x_2\in F_2$, there exists $x_1\in F_1$ with $x_1\le x_2$ in the weak order.
Not surprisingly, the dual version of this latter characterization does not hold:
It is possible to find fibers $F_1\le F_2$ and an element $x_1\in F_1$ which is below no element of $F_2$.
For example, take~$W$ with $S=\set{s_1,s_2,s_3}$ and $m(s_1,s_2)=m(s_1,s_3)=m(s_2,s_3)=\infty$ and set $c=s_1s_2s_3$.
Then the fiber $F_1=(\pidown^c)^{-1}(s_2)$ contains the element $s_2s_1$.
The fiber $F_2=(\pidown^c)^{-1}(s_2s_3)$ has $F_2\ge F_1$ but no element of $F_2$ is above $s_2s_1$.

There is also a geometric characterization of the $c$-Cambrian semilattice, inherited from the geometric characterization of the weak order which we now review.
Let $b^*$ be a vector in the interior of $D$. 
All cover relations of the weak order on~$W$ are obtained as follows:
Suppose $w_1D$ and $w_2D$ are adjacent and let $n\in V$ be a normal vector to their common boundary hyperplane, with $n$ pointing in the direction of $w_2D$.
The characterization of the weak order in terms of inversion sets implies that $w_1D\covered w_2D$ if and only if $\br{b^*,n}<0$. In light of Lemma~\ref{AB inv} and Theorem~\ref{pidown fibers}, the $c$-Cambrian semilattice is the partial order on the cones $\Cone_c(v)$ with covers defined analogously in terms of the same vector $b^*$.

Another important lattice-theoretic property of $c$-sortable elements is the following theorem, which was pointed out for finite~$W$ in~\cite[Remark 3.8]{sort_camb}.
Although the proof of Theorem~\ref{c to scs} depends (via Theorem~\ref{meets and joins}) on most of the machinery developed thus far, if Theorem~\ref{c to scs} were known from the beginning, almost all of the preceding results would follow easily.

\begin{theorem}\label{c to scs}
If~$s$ is initial in~$c$ then the map
\[v\mapsto\left\lbrace\begin{array}{cl}
sv&\mbox{if }v\ge s\\
s\join v&\mbox{if }v\not\ge s
\end{array}\right.\]
is a bijection from the set $\set{\mbox{$c$-sortable elements $v$ such that $s\join v$ exists}}$ to the set of all $scs$-sortable elements.
The inverse map~is
\[x\mapsto\left\lbrace\begin{array}{cl}
sx&\mbox{if }x\not\ge s\\
x_{\br{s}}&\mbox{if }x\ge s
\end{array}\right..\]
\end{theorem}
Theorem~\ref{c to scs} is illustrated pictorially in Section~\ref{aff sec}, especially Figures~\ref{bij1} and~\ref{bij2}.
We now proceed to prove Theorems~\ref{meets and joins},~\ref{semilattice hom} and~\ref{c to scs}.

\begin{proof}[Proof of Theorem~\ref{meets and joins}]
First consider the assertion about $\Meet A$. Define $I$ to be $\bigcap_{a \in A} \inv(a)$. Let $W'$ be a noncommutative generalized rank two parabolic subgroup of $W$, with canonical generators $r_1$ and $r_2$ and with $\omega_c(\beta_{r_1}, \beta_{r_2}) > 0$. Then $I \cap W' = \bigcap_{a \in A} \left( \inv(a) \cap W' \right)$ and every set in the latter intersection is either the empty set, the singleton $\{ r_2 \}$ or a set of the form $\{ r_1, r_1 r_2 r_1, r_1 r_2 r_1 r_2 r_1, \ldots, r_1 r_2 \cdots r_2 r_1 \}$. 
Thus, the intersection $I \cap W'$ is also of one of these forms. 
If $\omega_c(\beta_{r_1}, \beta_{r_2}) = 0$ then a slight modification of this argument shows that $I\cap W'$ is either $\emptyset$, $\{ r_1 \}$ or $\{ r_2 \}$. 
Thus, from Lemma~\ref{PilkLemma}, we see that $I$ is the set of inversions of some element $w$ of $W$. 
Necessarily, $w=\Meet A$.  
Moreover, we see that $w$ is $c$-aligned and hence, by Theorem~\ref{sortable is aligned}, $w$ is $c$-sortable. 

We now prove the assertion about $\Join A$. Let $w = \Join A$, which exists since we assumed that~$A$ has an upper bound. For every $a \in A$, we have $\pidown^c(w) \geq \pidown^c(a)=a$, where the inequality is by Theorem~\ref{order preserving} and the equality is by Proposition~\ref{decreasing}. So $\pidown^c(w) \geq \Join A=w$. Then, by Proposition~\ref{decreasing}, $\pidown^c(w)=w$ so $w=\Join A$ is $c$-sortable. 
\end{proof}

\begin{remark}
The proof of Theorem~\ref{meets and joins} shows that, for any nonempty set $A$ of $c$-sortable elements, $\inv \left( \Meet_{a \in A} a \right) = \bigcap_{a \in A} \inv(a)$. This is not true for arbitrary nonempty subsets $A$ of $W$. 
\end{remark}

\begin{proof}[Proof of Theorem~\ref{semilattice hom}]
Suppose~$A$ is nonempty so that $\Meet A$ exists.
For each $a \in A$, we have $a \geq \Meet A$ so that $\pidown^c(a) \geq \pidown^c(\Meet A)$ by Theorem~\ref{order preserving}.
Thus $\Meet \pidown^c(A) \geq \pidown^c( \Meet A)$. 

However, $\Meet \pidown^c(A)\le\Meet A$ because $\pidown^c(a)\le a$ for each $a$ in~$A$.
Thus by Theorem~\ref{order preserving}, $\pidown^c\left(\Meet \pidown^c(A)\right)\le\pidown^c\left(\Meet A\right)$.
But Theorem~\ref{meets and joins} states that $\Meet \pidown^c(A)$ is $c$-sortable, so that $\pidown^c\left(\Meet \pidown^c(A)\right)=\Meet \pidown^c(A)$ and $\Meet \pidown^c(A) \leq \pidown^c \left( \Meet A \right)$.
Thus $\Meet \pidown^c(A)=\pidown^c( \Meet A)$.

Now, consider the assertion about joins.
If~$A$ has an upper bound $b$ then since the set of elements below $b$ is finite, $A$ is finite. 
Thus it is enough to prove, for any $x$ and $y$ in~$W$ such that $x\join y$ exists, that $\pidown^c(x)\join\pidown^c(y)=\pidown^c(x\join y)$. Let $s$ be initial in $c$. We first prove a special case of the result:

\textbf{Special Case:} $x=s$ and $y \not \geq s$.
In this case,~$s$ is a cover reflection of $s\join y$ by Lemma~\ref{cov w br s}.  
By Theorem~\ref{order preserving}, $\pidown^c(s\join y)$ is above~$s$ and by Proposition~\ref{decreasing}, $\pidown^c(s(s\join y))$ is not above~$s$.
Thus by Theorem~\ref{pidown fibers}, the hyperplane $H_s$ defines a wall of $\Cone_c(\pidown^c(s\join y))$ so $-\alpha_s\in A_c(\pidown^c(s\join y))$.
Theorem~\ref{lower walls} implies that $s$ is a cover reflection of $\pidown^c(s\join y)$.
Now Proposition~\ref{Cc initial} implies that every other cover reflection of $\pidown^c(s\join y)$ is in $W_{\br{s}}$, so that $\pidown^c(s\join y)=s\join\left(\left[\pidown^c(s\join y)\right]_{\br{s}}\right)$ by Lemma~\ref{s join w br s}.
By Propositions~\ref{pidown para} and~\ref{para hom} we have
\[\left[\pidown^c(s\join y)\right]_{\br{s}}=\pidown^{sc}\left((s\join y)_{\br{s}}\right)=\pidown^{sc}\left(s_{\br{s}}\join y_{\br{s}}\right)=\pidown^{sc}\left(1\join y_{\br{s}}\right)=\pidown^{sc}\left(y_{\br{s}}\right),\]
which is $\pidown^c(y)$ by Proposition~\ref{CSortRecursive}, so $\pidown^c(s\join y)=s\join\left(\left[\pidown^c(s\join y)\right]_{\br{s}}\right)=\pidown^c(s) \join \pidown^c(y)$. 

We now prove the general result by induction on rank and on the length of $x\join y$.
We break into several cases.

\textbf{Case 1:} $x\ge s$ and $y\ge s$. In this case, $x\join y\ge s$.
Thus $\pidown^c(x)=s\pidown^{scs}(sx)$, $\pidown^c(y)=s\pidown^{scs}(sy)$ and $\pidown^c(x\join y)=s\pidown^{scs}(s(x\join y))$.
Now $sx\join sy = s(x\join y)$ has length $\ell(x\join y)-1$, so by induction on $\ell(x\join y)$ we obtain $\pidown^{scs}(sx)\join\pidown^{scs}(sy)=\pidown^{scs}(sx\join sy)$.
Thus $\pidown^c(x)\join\pidown^c(y)$ is
\[(s\pidown^{scs}(sx))\join(s\pidown^{scs}(sy))=s\left(\pidown^{scs}(sx)\join\pidown^{scs}(sy)\right)=s\pidown^{scs}(sx\join sy)\]
which equals $s\pidown^{scs}(s(x\join y))=\pidown^c(x\join y)$.

\textbf{Case 2:} $x\not\ge s$ and $y\not\ge s$. Then $(x\join y)\not\ge s$ by Corollary~\ref{join below}.
We have $\pidown^c(x)=\pidown^{sc}(x_{\br{s}})$, $\pidown^c(y)=\pidown^{sc}(y_{\br{s}})$ and $\pidown^c(x\join y)=\pidown^{sc}((x\join y)_{\br{s}})$.
The latter equals $\pidown^{sc}(x_{\br{s}}\join y_{\br{s}})$ by Proposition~\ref{para hom}.
By induction on rank, $\pidown^{sc}(x_{\br{s}})\join\pidown^{sc}(y_{\br{s}})=\pidown^{sc}(x_{\br{s}}\join y_{\br{s}})$, so $\pidown^c(x)\join\pidown^c(y)=\pidown^c(x\join y)$.

\textbf{Case 3:} Exactly one of $x$ and $y$ is above $s$.
Without loss of generality, we take $x\ge s$ and $y\not\ge s$. Let $z=s\join y$, so that $x\join y=x\join z$ and $\pidown^c(x\join y)=\pidown^c(x\join z)$.
Since $sx\join sz = s(x\join z)$ has length $\ell(x\join z)-1=\ell(x\join y)-1$ we argue by induction on $\ell(x\join y)$ as in Case 1 to conclude that $\pidown^c(x\join z)=\pidown^c(x)\join\pidown^c(z)$.
Now $\pidown^c(z)=\pidown^c(s\join y)=\pidown^c(s)\join\pidown^c(y)$ by the Special Case.
Thus $\pidown^c(x\join y)=\pidown^c(x)\join\pidown^c(s)\join\pidown^c(y)$.
Since $\pidown^c(s)\le\pidown^c(x)$ by Theorem~\ref{order preserving}, we obtain $\pidown^c(x\join y)=\pidown^c(x)\join\pidown^c(y)$.
\end{proof}

\begin{proof}[Proof of Theorem~\ref{c to scs}]
Proposition~\ref{CSortRecursive} implies that the involution $v\mapsto sv$ is a bijection between \mbox{$c$-sortable} elements above~$s$ and $scs$-sortable elements not above~$s$.
Thus it remains to show that $v\mapsto s\join v$ is a bijection from \mbox{$c$-sortable} elements~$v$ with $v\not\ge s$ such that $s\join v$ exists to $scs$-sortable elements~$x$ above~$s$ and that the inverse map is $x\mapsto x_{\br{s}}$.

If $v$ is $c$-sortable with $v\not\ge s$ then by Proposition~\ref{CSortRecursive}, $v$ is an $sc$-sortable element of $W_{\br{s}}$.
In particular, $v$ is an $scs$-sortable element of~$W$ by Proposition~\ref{sort para easy}.
Thus Theorem~\ref{meets and joins} now implies that $s\join v$ is an $scs$-sortable element (necessarily above~$s$).  
By Proposition~\ref{para hom}, $(s\join v)_{\br{s}}=s_{\br{s}}\join v_{\br{s}}=1\join v = v$.

If $x$ is $scs$-sortable with $x\ge s$ then since~$s$ is final in $scs$, $x_{\br{s}}$ is $sc$-sortable by Proposition~\ref{sort para}.
In particular, $x_{\br{s}}$ is a \mbox{$c$-sortable} element not above~$s$. 
Furthermore, Proposition~\ref{Cc final} states that $x=s\join x_{\br{s}}$.
\end{proof}

\section{Canonical join representations}\label{join rep sec}

In this section we consider a lattice theoretic construction which has important consequences for sortable elements, particularly as they relate to noncrossing partitions.
Specifically, we consider a canonical way of representing an element~$w$ of the weak order as the join of an antichain~$A$ of join-irreducible elements.
If $A$ is the unique minimal (in a sense made precise below) antichain whose join is $w$, then $A$ is called the canonical join representation of $w$.
The notion of canonical join representations is a key component in the theory of free lattices (see~\cite{FreeLattices}).
We will prove the following theorem.

\begin{theorem}\label{W can join rep}
Every $w\in W$ has a canonical join representation~$A$.
Furthermore $\cov(w)$ is the disjoint union, over all $j\in A$, of $\cov(j)$.
\end{theorem}
Note that since each $j\in A$ is join-irreducible, each set $\cov(j)$ is a singleton.
This theorem implies, by a well-known fact \cite[Theorem~2.24]{FreeLattices} from the theory of free lattices and by Lemma~\ref{anti}, that every interval in the weak order is a semi-distributive lattice in the sense of \cite[Section~I.3]{FreeLattices}.
A stronger result was proven, for~$W$ finite, in~\cite{boundedref}.  Cf.~\cite{Caspard,hyperplane}.

More importantly for our purposes, canonical join representations of sortable elements are well-behaved in the following sense.
\begin{prop}\label{csort canon}
If $v\in W$ is $c$-sortable then every element of its canonical join representation is $c$-sortable.
\end{prop}

Call a reflection \newword{$c$-accessible} if it occurs in the inversion set of some $c$-sortable element of~$W$.
Proposition~\ref{csort canon} leads to the following theorem, which establishes a bijection between $c$-accessible reflections and $c$-sortable join-irreducible elements.
\begin{theorem}\label{exactly one ji}
Let $t$ be any $c$-accessible reflection in~$W$.
Then there exists a unique $c$-sortable join-irreducible~$j$ whose unique cover reflection is $t$.
Furthermore,~$j$ is the unique minimal element among $c$-sortable elements having $t$ as an inversion.
\end{theorem}
For examples illustrating Theorems~\ref{W can join rep} and~\ref{exactly one ji} and Proposition~\ref{csort canon}, see Section~\ref{aff sec}.
Theorems~\ref{W can join rep} and~\ref{exactly one ji} and Proposition~\ref{csort canon} also give greater insight into a map $\nc_c$ defined on $c$-sortable elements.
When~$W$ is finite, $\nc_c$ is a bijection from $c$-sortable elements to $c$-noncrossing partitions \cite[Theorem~6.1]{sortable}.
At the end of this section, we recall the definition of noncrossing partitions and of $\nc_c$ and explain how Theorem~\ref{W can join rep} and Proposition~\ref{csort canon} lead to the insight that, for finite~$W,$ noncrossing partitions encode the canonical join representations of sortable elements.
For infinite $W$, there is not yet a consensus on the right definition of noncrossing partitions, so we cannot generalize the result that $\nc_c$ is a bijection.
However, we show in Theorem~\ref{nc inj} that $\nc_c$, and several variants of $\nc_c$, are injective.

We now proceed to give details and proofs.
Recall from Section~\ref{CoxeterConventions} that a partially ordered set $P$ is \newword{finitary} if, for every $x\in P$, the set of elements less than or equal to $x$ is finite.
If $L$ is a finitary meet-semilattice then $L$ has a unique minimal element, denoted $\0$, and furthermore, by a standard argument, for every $x$, the interval $[\0,x]$ is a lattice.
Define a pre-order\footnote{A preorder is a relation that is reflexive and transitive but may not be antisymmetric.} ``$\lleq$'' on subsets of $L$ by setting $A\lleq B$ if and only if for every $a\in A$ there exists $b\in B$ with $a\le b$.
Notice that if $A\lleq B$ and $\Join B$ exists then $\Join A$ exists and $\Join A\le \Join B$ in $L$.

For $A\subseteq L$, we write $I(A)$ for the (lower) order ideal generated by~$A$. 
Another description of the relation $\lleq$ is that $A \lleq B$ if and only if $I(A) \subseteq I(B)$. 
Any finite order ideal is of the form $I(A)$ for precisely one anti-chain~$A$.
Thus $\lleq$ induces a partial order (rather than just a preorder) on the finite anti-chains of $L$. 

An antichain $A\subseteq L$ is called the \newword{canonical join representation} of $x$ if:
\begin{enumerate}
\item[(i) ] $x=\Join A$; and
\item[(ii)] If $B$ is any other antichain in $L$ with $\Join B=x$ then $A \lleq B$.
\end{enumerate}
The canonical join representation of $x$, if it exists, is the unique minimal antichain, with respect to $\lleq$, among antichains joining to $x$.
Since $L$ is finitary, the canonical join representation of $x$ is a finite set.

Recall also that an element $j\in L$ is called join-irreducible if any finite set $X$ with $j=\Join X$ has $j\in X$ or equivalently if $j$ covers exactly one element of $L$.
\begin{lemma}\label{can rep ji}
If~$A$ is the canonical join representation of some $x\in L$ then every element of~$A$ is join-irreducible.
\end{lemma}
\begin{proof}
Suppose~$A$ is an antichain with $\Join A=x$.
If some $j\in A$ is not join-irreducible, then $j=\Join X$ for some $X\subseteq L$ with $j\not\in X$.
Since $L$ is finitary, $X$ is finite.
The set $A'$ of maximal elements of $X\cup (A\setminus \set{j})$ is a finite antichain joining to $x$.
But no element of $A'$ is above~$j$, so $A\notlleq A'$, and thus~$A$ is not the canonical join representation of $x$.
\end{proof}

\begin{remark}\label{free lattices remark}
Our development of canonical join representations is loosely modeled on the treatment in \cite[Section~II.1]{FreeLattices}.
One difference is that we use the symbol ``$\lleq$'' in place of ``$\ll$'' to emphasize the fact that $A\lleq A$ for any $A\subseteq L$. \end{remark}

The weak order is finitary, so the general notions outlined above apply.  
The proof of Theorem~\ref{W can join rep} begins with the following observation.

\begin{lemma} \label{BelowWAboveT}
Let $w \in W$ and let $t$ be a cover reflection of~$w$. 
Then there is some $j(w,t)$ in $W$ such that the set $\set{v\in W: v\le w \mbox{ and } v\not\le tw}$ is equal to the interval $[j(w,t), w]$. \end{lemma}
\begin{proof}
Applying Lemma~\ref{anti}, the lemma is equivalent to the statement that the (lower) order ideal $\set{v\in W:v\le w^{-1}\mbox{ and }v\not\ge w^{-1}tw}$ has a unique maximal element.
But $tw\covered w$ so $tw=ws$ and $w^{-1}tw=s$ for some simple generator $s$. 
Now Corollary~\ref{join below} implies that $\set{v\in W:v\le w^{-1}\mbox{ and }v\not\ge w^{-1}tw}$ is closed under joins.
Since this set is finite, it has a unique maximal element.
\end{proof}

To prove Theorem~\ref{W can join rep}, we show that the set $\set{j(w,t):t\in\cov(w)}$ is the canonical join representation of~$w$.

\begin{proof}[Proof of Theorem~\ref{W can join rep}]
Let~$A$ stand for the set $\set{j(w,t):t\in\cov(w)}$.
First, we check that~$A$ is an antichain. 
Suppose $t,t'\in\cov(w)$ with $t\neq t'$.
Then $\inv(tw)=\inv(w)\setminus\set{t}$ and $\inv(t'w)=\inv(w)\setminus \{ t' \}$.
In particular, $t'w$ is an element of the set $\set{v\in W: v\le w \mbox{ and } v\not\le tw}$, so $t'w\ge j(w,t)$.
This implies that $j(w,t)$ is not in $\set{v\in W: v\le w \mbox{ and } v\not\le t'w}=[j(w,t'),w]$, and since $j(w,t)\le w$, we have $j(w,t)\not\ge j(w,t')$.

Next, we check that $w = \Join A$. 
Note first that~$w$ is an upper bound for~$A$, so $\Join A$ exists and is weakly below~$w$. 
Let $v$ be any element covered by~$w$, so that in particular $v=tw$ for some $t\in\cov(w)$.
Since $t$ is an inversion of $j(w,t)$ but $\inv(v)=\inv(w)\setminus\set{t}$, the element $v$ is not above $j(w,t)$.
Thus~$w$ is the minimal upper bound for~$A$.

Now suppose that $w = \Join B$ for some antichain $B$ in $L$.
We need to show that $A\lleq B$.
That is, we must show that for every $t\in\cov(w)$, there exists $b\in B$ with $j(w,t)\le b$. 
Fix one such $t$.
Since $tw < w$ and $\Join B=w$, we know that $tw$ is not an upper bound for $B$. 
Let $b \in B$ be such that $b \not \leq tw$. Then, since we know $b \leq w$, we have $j(w,t) \leq b$.
This proves that $A\lleq B$, thus completing the proof that~$A$ is the canonical join representation of~$w$.

In particular, Lemma~\ref{can rep ji} implies that each $j(w,t)$ is join-irreducible.
Let $t'$ be the unique cover reflection of $j(w,t)$.
Since $j(w,t)$ is the unique minimal element of $\set{v\in W: v\le w \mbox{ and } v\not\le tw}$, we must have $t'j(w,t)\le tw$.
But $\inv(tw)=\inv(w)\setminus\set{t}$ and $\inv(t'j(w,t))=\inv(j(w,t))\setminus\set{t'}$, so $t=t'$.
This establishes the second assertion of the theorem.
\end{proof}

Proposition~\ref{csort canon} now follows easily from Theorem~\ref{semilattice hom}.

\begin{proof}[Proof of Proposition~\ref{csort canon}]
Let $v$ be $c$-sortable and let~$A$ be its canonical join representation.
Consider the set $\pidown^c(A)=\set{\pidown^c(j):j\in A}$.
Theorem~\ref{semilattice hom} implies that 
\[\Join\left(\pidown^c(A)\right)=\pidown^c(\Join A)=\pidown^c(v)=v.\]
Let $B$ be the antichain of maximal elements of $\pidown^c(A)$.
In particular, $\Join B=v$.
Since~$A$ is the canonical join representation of $v$, we have $A\lleq B$.
But since each element of $B$ is $\pidown^c(j)$ for some $j\in A$, we have $B\lleq A$.
Thus $A=B$ and therefore $A=\pidown^c(A)$.
Thus $j=\pidown^c(j)$ for each $j\in A$.
\end{proof}

If $v$ is a $c$-sortable join-irreducible element, then its unique cover reflection is, by definition, a $c$-accessible reflection.
Theorem~\ref{exactly one ji} states in particular that this map from $c$-sortable join-irreducible elements to $c$-accessible reflections is a bijection.
We prove the theorem in two halves.
The first half is proved by a simple inductive argument using no deep results.

\begin{prop}\label{at most one ji}
For any reflection~$t$, there is at most one join-irreducible $c$-sortable element whose unique cover reflection is~$t$.
\end{prop}
\begin{proof}
Suppose $j_1$ and $j_2$ are join-irreducible $c$-sortable elements, each having~$t$ as a cover reflection.
We argue by induction on $\ell(j_1)$ and on the rank of~$W$ that $j_1=j_2$.
Let~$s$ be initial in~$c$.

If either join-irreducible (say $j_1$) is not above~$s$ then it is in $W_{\br{s}}$, so that in particular $t\in W_{\br{s}}$.
Therefore $j_2$ has the property that all of its cover reflections are in $W_{\br{s}}$ and thus $j_2$ is also in $W_{\br{s}}$ by Lemma~\ref{cov para}.
By induction on rank, $j_1=j_2$.

Otherwise both $j_1$ and $j_2$ are above~$s$.
If either (say $j_1$) equals~$s$, then $t=s$, and Lemma~\ref{cov para} implies that $j_2\in W_{\set{s}}$, so that $j_2=s$ as well.
If neither $j_1$ nor $j_2$ is~$s$ then $sj_1$ and $sj_2$ are join-irreducible $scs$-sortable elements having the same cover reflection.
By induction on length, $sj_1=sj_2$, so $j_1=j_2$.
\end{proof}

The second half of the proof of Theorem~\ref{exactly one ji} is implied by the following proposition, which uses, via Proposition~\ref{csort canon}, the machinery developed over the course of the paper.

\begin{prop}\label{at least one ji}
Let~$t$ be a $c$-accessible reflection and let~$v$ be minimal in the weak order among $c$-sortable elements having $t\in\inv(v)$.
Then~$v$ is join-irreducible and has~$t$ as its unique cover reflection.
\end{prop}

\begin{proof}
If $t_1,t_2,\ldots,t_k$ is the sequence of reflections arising from a $c$-sorting word $s_1s_2\cdots s_k$ for $v$ then in particular $t=t_k$.
(If not, then $t=t_i$ for some $i<k$, so that $s_1s_2\cdots s_i$ is a $c$-sortable element having~$t$ as an inversion, contradicting the minimality of $v$.)
Thus in particular,~$t$ is a cover reflection of $v$.
If $v$ has an additional cover reflection then Proposition~\ref{csort canon} and Theorem~\ref{W can join rep} imply that $v$ is the join of a set of two or more join-irreducible $c$-sortable elements strictly less than $v$, one of which has~$t$ as a cover reflection.
This contradicts minimality.
Thus $v$ is a $c$-sortable join-irreducible element whose unique cover reflection is $t_k=t$.
\end{proof}

We now describe noncrossing partitions and explain how they relate to the results of this section.
Let~$W$ be a finite Coxeter group.
An element $w\in W$ can be written (typically in many ways) as a word in the alphabet $T$.
We refer to words in the alphabet $T$ as \newword{$T$-words}, to avoid confusion with words in the alphabet~$S$, which are called simply ``words'' throughout the paper.
A \emph{reduced} $T$-word for~$w$ is a $T$-word for~$w$ which has minimal length among all $T$-words for~$w$.
Say $x\le_T y$ if~$x$ possesses a reduced $T$-word which is a prefix of some reduced $T$-word for~$y$.
The {\em noncrossing partition lattice} in~$W$ (with respect to the Coxeter element~$c$) is the interval $[1,c]_T$ in the partial order ``$\le_T$.''
The elements of this interval are called \newword{$c$-noncrossing partitions}.

Results of \cite{BWorth} show that a noncrossing partition is uniquely determined by its fixed space and that the fixed space $\Fix(x)\subseteq V^*$ of a noncrossing partition $x$ is the intersection $\bigcap_{i=1}^kH_{t_i}$, where $t_1\cdots t_k$ is any reduced $T$-word for $x$ and $H_t$ denotes the hyperplane fixed by $t$.
The subspace $\bigcap_{i=1}^kH_{t_i}$ is uniquely defined by the associated \newword{noncrossing parabolic subgroup}---the subgroup $\Stab\left(\bigcap_{i=1}^kH_{t_i}\right)$ fixing $\bigcap_{i=1}^kH_{t_i}$ pointwise.

In~\cite{sortable}, a map $\nc_c$ was defined on $c$-sortable elements and was shown \cite[Theorem~6.1]{sortable} to be a bijection from $c$-sortable elements to $c$-noncrossing partitions.
The definition of $\nc_c(v)$ is as follows: Let $t_1,t_2,\ldots,t_k$ be the reflection sequence for a $c$-sorting word for $v$. 
Since every cover reflection of $v$ is an inversion of $v$, the cover reflections of $v$ occur in the sequence $t_1,t_2,\ldots,t_k$.
 Let $u_1,u_2,\ldots,u_r$ be the cover reflections of $v$, ordered as a subsequence of $t_1,t_2,\ldots,t_k$. 
Then $\nc_c(v)=u_1 u_2 \cdots u_r$. 
The product $u_1 u_2 \cdots u_r$ is a reduced $T$-word for $\nc_c(v)$. 
By Lemma~\ref{cov canon}, the  reflections $u_1,u_2,\ldots,u_r$ are the canonical generators of the parabolic subgroup associated to $\nc_c(v)$.

Theorem~\ref{exactly one ji} allows us to identify (when~$W$ is finite) the set $T$ of reflections of~$W$ with the set of join-irreducible $c$-sortable elements of~$W$.
Under this identification, Theorem~\ref{W can join rep} and Proposition~\ref{csort canon} imply that, if one specifies noncrossing partitions by the canonical generators of the associated parabolic subgroup, the noncrossing partitions are exactly the canonical join representations of sortable elements.
Put another way, a set of reflections canonically generates a noncrossing parabolic subgroup if and only if the corresponding set of join-irreducible $c$-sortable elements canonically represents its join.

For general $W$, the map $\nc_c$ still makes sense, and our results easily imply that the map $\nc_c$ is injective.
In fact, all of the natural variants of $\nc_c$ are injective, as made precise in the following theorem.
\begin{theorem}\label{nc inj}
The following maps are injective:
\begin{enumerate}
\item[(i) ]The map $\nc_c$ from $c$-sortable elements of $W$ to $W$;
\item[(ii) ]The map $v\mapsto \Fix(\nc_c(v))$ from $c$-sortable elements of $W$ to subspaces of $V^*$;
\item[(iii) ]The map $v\mapsto \Stab(\Fix(\nc_c(v)))$ from $c$-sortable elements of $W$ to parabolic 
subgroups of $W$;
\item[(iv) ] The map $\cov$ from $c$-sortable elements of $W$ to sets of reflections in $W$.
\end{enumerate}
\end{theorem}
\begin{proof}
Suppose $v$ and $v'$ are $c$-sortable elements with $\cov(v)=\cov(v')$.
By Theorem~\ref{W can join rep} and Proposition~\ref{csort canon}, the canonical join representation of $v$ consists of $c$-sortable join-irreducible elements $j$, such that the set of cover reflections occurring as cover reflections of the $j$'s is $\cov(v)$.
The same is true of the canonical join representation of $v'$, so by Theorem~\ref{exactly one ji}, these two join representations coincide, and thus $v=v'$.
This is (iv).

Obviously, $\nc_c(v)$ determines $\Fix(\nc_c(v))$ which determines  $\Stab(\Fix(\nc_c(v)))$. 
We conclude the proof by showing that $\Stab(\Fix(\nc_c(v)))$ determines $\cov(v)$.
Thus, if any of the maps in (i), (ii) or (iii) takes $v$ and $v'$ to the same object, we have $\cov(v)=\cov(v')$ and thus $v=v'$.

Specifically, we will show that for any $w\in W$ and any ordering $t_1,\ldots,t_k$ of $\cov(w)$, the stabilizer of the fixed set of $t_1\cdots t_k$ has canonical generators $\cov(w)$.
Thus, in particular, since every parabolic subgroup has a uniquely defined set of canonical generators, $\Stab(\Fix(\nc_c(v)))$ determines $\cov(v)$.

Let $U=\bigcap_{t\in\cov(w)}H_t$.
Lemma~\ref{cov canon} states that $W' :=\Stab(U)$ is a finite parabolic subgroup with canonical generators $\cov(w)$.
For an ordering $t_1,\ldots,t_k$ of $\cov(w)$, for each $i\in[k]$ we have $t_i=w s_i w^{-1}$ for some $s_i \in S$. 
Let $J=\{ s_1, s_2, \ldots s_k \}$. 
Then $W'=w W_J w^{-1}$, so $W_J$ is finite.
It is known (see \cite[Theorem~V.6.1]{Bourbaki} or \cite[Lemma~3.16]{Humphreys}) that for a finite Coxeter group $W'$ with a reflection representation $V'$ of dimension equal to the rank of $W'$, a Coxeter element fixes no point of $(V')^*$ except the origin.
Since $s_1\cdots s_k$ is a Coxeter element of $W_J$, we have $\Fix(s_1\cdots s_k)=(V_J)^\perp$, so $\Fix(t_1\cdots t_k)=w\Fix(s_1\cdots s_k)=w[(V_J)^\perp]=U$. 
\end{proof}

\begin{remark}
We cannot at present generalize \cite[Theorem~6.1]{sortable} to show that $\nc_c$ is a bijection from $c$-sortable elements of $W$ to $c$-noncrossing partitions, because it is not yet clear what is the right definition of noncrossing partitions in an infinite Coxeter group.
(The ``right'' definition would be the definition that is most useful in studying Artin groups of infinite type.)
The most obvious generalization is to lift the definition verbatim from the finite case, so that the $c$-noncrossing partitions are the elements of $[1,c]_T$.
When $W$ is infinite, the map $\nc_c$ fails to be surjective onto $[1,c]_T$.
\end{remark}

\section{The Cambrian fan}\label{fan sec}
In this section, we study the collection $\F_c$ consisting of all $c$-Cambrian cones together with all of their faces.
We conjecture that this collection of cones is a fan in the usual sense (see Section~\ref{Root subsec}) and prove the following weaker statement, whose terminology will be explained below. 
\begin{theorem}\label{fan in Tits}
The collection $\F_c$ of cones is a fan in the interior of $\Tits(W)$.
The decomposition, induced by $\F_c$, of $\Tits(W)$ into convex regions is a coarsening of the fan $\F(W)=\set{wD:w\in W}$.
\end{theorem}
The first assertion of the theorem is made precise below.
Informally, the assertion is that cones in $\F_c$ intersect nicely inside the Tits cone.

After proving Theorem~\ref{fan in Tits}, we gather some results (Propositions~\ref{fan para}, \ref{fan s} and~\ref{recursive fan}) on the structure of the fan as it relates to standard parabolic subgroups and reflection functors (i.e.\ replacing $c$ by $scs$ for $s$ initial).
These results all follow from results of Sections~\ref{skip sec} and~\ref{pidown sec} concerning $C_c$ and $\pidown^c$.
To conclude the section, we show that locally in the interior of $\Tits(W)$, the fan $\F_c$ looks like a Cambrian fan for a finite standard parabolic subgroup of $W$.
For a very precise statement, see Theorem~\ref{stars}.

%We now make the statement of Theorem~\ref{fan in Tits} precise and prove the theorem.
Let $\F_c$ be the collection of polyhedral cones that occur as the faces of the cones $\Cone_c(v)$ as $v$ ranges over all of the $c$-sortable elements of $W$.
Theorem~\ref{fan in Tits} will imply in particular that when $W$ is finite, $\F_c$ is a fan in $V^*$.
We conjecture that $\F_c$ is a fan even when $W$ is infinite.
In Section~\ref{rank3sec}, we draw several examples of $\F_c$.
In all of these examples (and many others), we have checked that $\F_c$ is in fact a fan. 
For the moment we optimistically refer to $\F_c$ as the $c$-\newword{Cambrian fan}.

Showing that $\F_c$ is a fan in $V^*$ is difficult because we have no control over the intersections of cones outside $\Tits(W)$.
In particular, we do not even know that the interiors of $\Cone_c(v)$ and $\Cone_c(v')$ are disjoint for $v \neq v'$.
However, in light of Theorem~\ref{pidown fibers}, we have great control of the intersections of cones inside $\Tits(W)$.
The following definition is useful in describing the situation.

Let $U$ be an open convex set in $\RR^n$. 
We say that a collection $\F$ of cones in $\RR^n$ is a \newword{fan in $U$} if 
\begin{enumerate}
\item[(i) ] for every $F\in\F$, the faces of $F$ are also in $\F$ and 
\item[(ii$'$) ] for every subset $\X\subset\F$ and every $F\in\X$, there is a face $F'$ of $F$ such that $\bigcap_{G\in\X}G \cap U= F' \cap U$.
\end{enumerate}
When $U=\RR^n$, condition (ii$'$) reduces to condition (ii) in the definition of a fan, given in Section~\ref{Root subsec}.

The proof of Theorem~\ref{fan in Tits} proceeds by first showing that the fan property only needs to be checked locally (Lemma~\ref{local fan}, below) and then applying results of~\cite{con_app}.
For any cone $F\in\RR^n$ and any $x\in F$, the \newword{linearization $\Lin_x(F)$ of $F$ at $x$} is the set of vectors $x'\in\RR^n$ such that $x+\epsilon x'$ is in $F$ for any sufficiently small nonnegative~$\epsilon$.
If $\F$ is a fan in $U$ and $x$ is any point in the union of the cones of $\F$ then the \newword{star} of $x$ in $\F$ is the fan $\Star_x(\F)=\set{\Lin_x(F): x \in F\in\F}$.  
For any cone $F$ of $\F$ such that $U$ intersects the relative interior of $F$, the \newword{star} of $F$ in $\F$ is the fan $\Star_F(\F)=\Star_x(\F)$ for any $x$ in the intersection of $U$ with the relative interior of $F$.
This is well-defined because $\F$ is a fan in $U$:  Condition (ii$'$) implies that any cone in $\F$ that intersects $U\cap\mathrm{relint}(F)$ must contain $U\cap F$.

\begin{lemma}\label{local fan}
Let $\F$ be a set of pointed polyhedral cones in $\RR^n$ and let $U$ be an open convex set in $\RR^n$. 
Suppose that $\F$ satisfies condition (i) of the definition of a fan.
Suppose further that, for every $x \in U$, the set of cones $\Lin_x(\F)=\set{\Lin_x(F):F\in\F}$ is a fan. 
Then $\F$ is a fan in $U$. 
\end{lemma}

\begin{proof}
We check condition (ii$'$) of the definition.
Fix a subset $\X\subseteq\F$ and a cone $F\in\X$ and set $I=\bigcap_{G\in\X} G$. 
We need to find a face $F'$ of $F$ such that $I \cap U= F' \cap U$.
Notice that, since $F \in \X$, we know that $I \cap U \subseteq F$.

If $I \cap U$ is empty, then $U$ must not contain zero. 
In that case, we can take $F'$ to be the face $\{ 0 \}$.
So, from now on, we assume that $I \cap U$ is nonempty.

Suppose that $\sigma$ and $\tau$ are two faces of $F$ whose relative interiors meet $I \cap U$.
Let $s$ be a point in the intersection of $I\cap U$ with the relative interior of $\sigma$ and let $t$ be a point in the intersection of $I\cap U$ with the relative interior of $\tau$.
%sinceTAMS:  Changed "interior" to "relative interior" in the following sentence.  (Not a crucial change.)
Then, since $I \cap U$ is convex, $I \cap U$ contains the relative interior of the line segment $st$ and thus meets the relative interior of the minimal face of $F$ containing $\sigma$ and $\tau$.
Thus, there is a unique maximal face $F'$ of $F$, among those faces of $F$ whose relative interior meets $I \cap U$.
Moreover, $I \cap U \subseteq F'$.
Let $x_0$ be a point in the intersection of $I\cap U$ with the relative interior of $F'$.

Suppose for the sake of contradiction that $F' \cap U \neq I \cap U$. 
Let $x_1$ be a point in $F' \cap U$ but not $I \cap U$. 
Consider the line segment $\overline{x_0 x_1}$. 
By convexity, it lies in $F' \cap U$, but the endpoint~$x_0$ is in $I$ while the other endpoint $x_1$ is not. 
%sinceTAMS:  Changed "interior" to "relative interior" in the following sentence.  (Not a crucial change.)
Since $I$ is closed and convex,  $\overline{x_0 x_1} \cap I=\overline{x_0 x}$ for some point $x$ in the relative interior of the line segment $\overline{x_0 x_1}$. 

Now the cone $\bigcap_{G\in\X}\Lin_x(G)=\Lin_x(I)$ does not contain the vector $x_1-x$. 
However, $x$ is in the relative interior of $F'$, so for any face $F''$ of $F$ containing $x$, we have $F' \subseteq F''$ and, thus,
$x_1-x \in\Lin_x(F')\subseteq\Lin_x(F'')$. 
So, $\bigcap_{G\in\X}\Lin_x(G)$ is not equal to any face of $\Lin_x(F)$ and thus $\set{\Lin_x(F):F\in\F}$ is not a fan.
This contradiction shows that $F'\cap U=I\cap U$.
 \end{proof}
 
Each cone $\Cone_c(v)$ of $\F_c$ is defined by the $n$ linear inequalities specified by the roots $C_c(v)$.
Since the roots $C_c(v)$ are linearly independent, the cones $\Cone_c(v)$ are pointed.
To check the local condition in Lemma~\ref{local fan} in the case where $\F$ is $\F_c$ and $U$ is the interior of $\Tits(W)$, we apply a result of~\cite{con_app}.
Consider for the moment the case where $W$ is finite, and temporarily relax the requirement that the reflection representation of $W$ have dimension equal to the rank of $W$.
The Coxeter fan $\F(W)$ has a finite number of cones and so the weak order on $W$ is a lattice.
Recall that the elements of~$W$ are in bijection with the maximal cones of $\F(W)$. 
For any lattice congruence~$\Theta$ with congruence classes denoted $[w]_\Theta$, consider the collection of cones 
\[\Cone_\Theta(w)=\bigcup_{x\in[w]_\Theta}xD.\]
\emph{A priori} these cones may not be convex, but in fact they are convex as explained in \cite[Section~5]{con_app}.
Let $\F_\Theta$ be the collection of all cones $\Cone_\Theta(w)$ for $w\in W$ together with their faces.
The following theorem is a restatement of a very small part of \cite[Theorem~1.1]{con_app}.

\begin{theorem}\label{finite fan}
Let $W$ be finite and let $\Theta$ be any lattice congruence of the weak order on $W$.
Then $\F_\Theta$ is a complete fan which coarsens $\F(W)$.
\end{theorem}
The term ``complete'' here means that the union of the cones is all of $V^*$.

We now have all the necessary tools to prove the theorem.
\begin{proof}[Proof of Theorem~\ref{fan in Tits}]
Let $x$ be in the interior of $\Tits(W)$.
Let $W'$ be the stabilizer of $x$ in $W$. 
By Lemma~\ref{cov canon}, $W'$ is finite.

%Let $W'$ be the stabilizer of $x$ in $W$.
%Since $x$ is in the interior of $\Tits(W)$, there is a cone $wD$ containing~$x$. 
%The stabilizer of $w^{-1}x$ is a standard parabolic subgroup $W_J$ for some $J$, so $W'=wW_Jw^{-1}$ is a parabolic subgroup.
%Furthermore, we can choose $w$ so that $wD$ is separated from $D$ by every reflecting hyperplane in $W'$.
%Since $\inv(w)$ is finite, $W'$ must also be finite.

The fan $\Star_x(\F(W))$, the star of $x$ in the Coxeter fan of $W$, coincides with the Coxeter fan $\F(W')$ of $W'$, where we again temporarily relax the requirement that the reflection representation of $W'$ have dimension equal to the rank of $W'$.
The cones in $\Star_x(\F(W))$ are naturally in bijection with the cones $wD$ such that $x\in wD$.  
Let $w_1$ be such that $x\in w_1D$ and $w_1D$ is not separated from $D$ by any reflecting hyperplane in $W'$ and let $w_2$ be such that $x\in w_2D$ and $w_2D$ is separated from $D$ by every reflecting hyperplane in $W'$.  
Then a cone $wD$ contains $x$ if and only if $w$ is in the weak-order interval $[w_1,w_2]$.
This interval is isomorphic to the weak order on $W'$.
Thus by Theorem~\ref{semilattice hom}, the fibers of $\pidown^c$ define a lattice congruence on the weak order on $W'$.
Theorem~\ref{finite fan} implies that the fibers of $\pidown^c$ define a fan which coarsens $\F(W')=\Star_x(\F(W))$.
The maximal cones of this coarsening are the unions, over congruence classes, of the maximal cones of $\F(W')$.
This fan is exactly the set $\set{\Lin_x(F):x\in F\in\F_c}$.
\end{proof}

We now establish some important structural facts about the intersection of the $c$-Cambrian fan with $\Tits(W)$.
For any collection $\F$ of cones, define $\F\cap\Tits(W)$ to be the set $\set{F\cap\Tits(W):F\in\F}$.  

\begin{prop}\label{fan para}
Let $J\subseteq S$ and let $c'$ be the restriction of $c$ to $W_J$.
Then $\left[(P_J)^{-1}(\F_{c'})\right]\cap\Tits(W)$ is a coarsening of $\F_c\cap\Tits(W)$.
\end{prop}
Less formally, restrict the representation of $W$ on $V$ to $W_J$ and build the $c'$-Cambrian fan $\F_{c'}$ in this representation of $W_J$. Then, within $\Tits(W)$, the walls of $\F_{c'}$ are respected by the cones of $\F_c$.
That is, each cone $\Cone_c(v)$ has $\Cone_c(v)\cap\Tits(W)\subseteq\Cone_{c'}(v')$ for some $c'$-sortable element $v'$ of $W_J$.
\begin{proof}
Let $v$ be a $c$-sortable element of $W$.
We will show that $\Cone_c(v)\cap\Tits(W)\subseteq\Cone_{c'}(v_J)\cap\Tits(W)$.
Theorem~\ref{pidown fibers} implies that 
\[\Cone_c(v)\cap\Tits(W)=\bigcup_{w\in(\pidown^c)^{-1}(v)}wD\]
for any $c$-sortable element $v$ of $W$.
Thus 
\[P_J\left(\Cone_c(v)\cap\Tits(W)\right)\subseteq\bigcup_{w\in(\pidown^c)^{-1}(v)}w_JD_J.\]
By Proposition~\ref{pidown para}, the latter is contained in 
\[\bigcup_{x\in(\pidown^{c'})^{-1}(v_J)}xD_J=\Cone_{c'}(v_J)\cap\Tits(W_J).\qedhere\]
\end{proof}

In the special case where $J$ is a singleton, the fan $(P_J)^{-1}(\F_{c'})$ has two cones, both halfspaces.  
In this case, Proposition~\ref{fan para} becomes the following.
\begin{prop}\label{fan s}
For any $s\in S$ and any $c$-sortable element $v$, the interior of $\Cone_c(v)\cap\Tits(W)$ is disjoint from $H_s$.
\end{prop}
In other words, no $c$-Cambrian cone crosses the hyperplane $H_s$ within $\Tits(W)$.
We conjecture that the interior of $\Cone_c(v)$ is disjoint from $H_s$, and in fact this conjecture is the only missing step in a proof, by the usual induction on length and rank, of the conjecture that $\F_c$ is a fan in $V$.

The recursive definition of $C_c(v)$ translates into the following geometric statement about the $c$-Cambrian fan.
For any collection $\F$ of cones, define 
\[\F_{\ge s}=\set{F\cap\set{x^*\in V^*:\br{x^*,\alpha_s}\le 0}:F\in\F}.\]
Thus, in $\F_{\geq s} \cap \Tits(W)$, each cone is intersected with the portion of $\Tits(W)$ consisting those of cones $wD$ for which $w\ge s$. 
Define $\F_{>s}$ by making the inequality in the definition strict, and define $\F_{\le s}$ and $\F_{<s}$ by reversing the inequality.

\begin{prop}\label{recursive fan}
Let $s$ be initial in $c$.
Then
\begin{enumerate}
\item[(i) ]$(\F_c)_{\ge s}\cap\Tits(W) = s\left((\F_{scs})_{\le s}\right)\cap\Tits(W)$, and
\item[(ii) ]$(\F_c)_{<s}\cap\Tits(W)=\left[(P_{\br{s}})^{-1}(\F_{sc})\right]_{<s}\cap\Tits(W)$.
\end{enumerate}
\end{prop}
The conjecture that $\Cone_c(v)$ is disjoint from $H_s$ for any $c$-sortable element $v$ would imply that both statements in the proposition remain true when ``$\cap\,\Tits(W)$'' is removed throughout. 

Proposition~\ref{Cc final} is a combinatorial statement about $C_c(v)$.
However, it can be understood geometrically in the following manner:   
\begin{prop} \label{recursive fan final}
 Let $s$ be final in $c$. Then $(\F_c)_{> s} \cap \Tits(W)=\left[(P_{\br{s}})^{-1}(\F_{cs})\right]_{>s}\cap\Tits(W)$.
\end{prop}
See Section~\ref{aff sec} for illustrations of Proposition~\ref{recursive fan} and~\ref{recursive fan final}. 

We conclude the section with a precise statement about the local structure of~$\F_c$ inside $\Tits(W)$. 
Let $F$ be a cone in $\F_c$ whose relative interior intersects the interior of $\Tits(W)$.
Let $x$ be any point in the intersection of the interior of $\Tits(W)$ with the relative interior of $F$.
If $b$ is a vector in the interior of $D$ then as $\epsilon$ approaches zero from above, the point $x-\epsilon b$ remains in the interior of a certain $c$-Cambrian cone $\Cone_c(v)$, and also is in the interior of $\Tits(W)$.
Since~$\F_c$ is a fan in the interior of $\Tits(W)$, this $v$ is independent of the choice of $x$.
For the same reason, the intersection of $F$ with the interior of $\Tits(W)$ equals the intersection of some collection of facets of $\Cone_c(v)$ with the interior of $\Tits(W)$.
All of these facets of $\Cone_c(v)$ are associated to cover reflections of $v$, and indeed,~$F$ is specified uniquely by the choice of the $c$-sortable element $v$ and some subset of the cover reflections of $v$.
Let $J$ be such that this subset of the cover reflections of~$v$ is $\set{vsv^{-1}:s\in J}$.
Call $v$ \newword{the $c$-sortable element above $F$} and call $J$ the \newword{set of simple generators through $F$}. By Lemma~\ref{cov canon}, the group generated by the cover reflections of $v$ is finite. This group contains $v W_J v^{-1}$, so $W_J$ is finite and has a maximal element $w_0(J)$. Let $w=v w_0(J)$.
We call $w$ the \newword{element below $F$}. 
The element $w$ is characterized by the condition that $wD \cap F=vD \cap F$ and that $wD$ and $D$ are on the same side of the hyperplane $v H_s$ for each $s\in J$.

Let $v$, $w$, $F$ and $J$ be as in the previous paragraph. 
Since conjugation by $w_0(J)$ permutes $J$, we know that $w J w^{-1}=v w_0(J) J w_0(J)^{-1} v^{-1}$ is a subset of the cover reflections of $v$. 
Let $\Cox_c(F)$ or $\Cox_c(v,J)$ stand for the element $s_1 s_2 \cdots s_k$, where $J=\{ s_1, s_2, \ldots, s_k \}$ and the elements of $J$ are numbered in the order that the cover reflections $w s_i w^{-1}$ appear in the reflection sequence for the $c$-sorting word of $v$. 
In particular, if $v J v^{-1}=\cov(v)$, then $\Cox_c(F)=w^{-1} \nc_c(v) w$.
We can also describe $\Cox_c$ without any need to compute $c$-sorting words. 
In light of Proposition~\ref{InversionOrdering}, if $\omega_c(\beta_{w s_i w^{-1}},\beta_{w s_j w^{-1}}) > 0$ then $s_i$ must come before $s_j$ and if $\omega_c(\beta_{w s_i w^{-1}},\beta_{w s_j w^{-1}}) = 0$ then $s_i$ and $s_j$ commute.

\begin{theorem}\label{stars}  
%sinceTAMS:  Changed "interior" to "relative interior" in the following sentence. 
If $F$ is a cone in $\F_c$ whose relative interior intersects the interior of $\Tits(W)$ then $\Star_F(\F_c)$ is combinatorially isomorphic to a finite Cambrian fan.
Specifically, $\Star_F(\F_c)=w P_J^{-1}(\F_{\Cox_c(F)})$, where $v$ is the $c$-sortable element above~$F$, $J$ is the set of simple generators through $F$ and $w$ is the element below $F$.
\end{theorem}

This theorem generalizes \cite[Proposition~8.12]{camb_fan} and also strengthens the result by giving an explicit nonrecursive rule for computing $\Cox_c(F)$.  
The following lemma will be used in the proof.
\begin{lemma}\label{pidown join lemma}
Let $s$ be initial in $c$, let $x$ be an element of $W$ such that $x \not\ge s$ but $s\join x$ exists, and let $v$ be a $c$-sortable element of $W$ such that $v\not\ge s$ but $s\join v$ exists.
Then $\pidown^c(x)=v$ if and only if $\pidown^{scs}(s\join x)=s\join v$.
\end{lemma}
\begin{proof}
In light of Proposition~\ref{para hom}, the bijection of Theorem~\ref{c to scs} restricts to a bijection between $c$-sortable elements $v$ with $v\le x$ and $scs$-sortable elements $v'$ with $s\le v'\le s\join x$.
Corollary~\ref{max} now implies the lemma.
\end{proof}

\begin{proof}[Proof of Theorem~\ref{stars}]
Let $v$, $J$ and $w$ be as above.
Write $c=s_1\cdots s_n$ and let $k$ be the smallest nonnegative integer $\le n$ such that $s_k$ is not a cover reflection of $v$.
If every $s_i$ is a cover reflection of $v$ then $W$ is finite, $v$ is the longest element of $W$ and~$w$ is the identity, so that $\Star_F(\F_c)=\F_c$.
Thus whenever $k$ is not well-defined, the theorem holds trivially.
We argue by induction on the rank of $W$, on the length of~$v$ (keeping $W$ fixed), and on $k$ (keeping $W$ and $v$ fixed).

The intersection of all cones $xD$ with $x \in [w,v]$ is a face $F'$ of $\F$ with $F'\subseteq F$ and $\dim(F')=\dim(F)$.
For any point $p$ in the relative interior of $F'$ we have $\Star_F(\F_c)=\Star_p(\F_c)$.
Since $[w,v]=\set{x\in W:p\in xD}$, the star $\Star_F(\F_c)$ depends only on the equivalence relation induced on $[w,v]$ by the fibers of $\pidown^c$. 
The map $z\mapsto wz$ is an isomorphism from the weak order on $W_J$ to the interval $[w,v]$. 
Thus the theorem is equivalent to the assertion that, for all pairs $x,y\in W_J$, 
\[\pidown^{\Cox_c(v,J)}(x)=\pidown^{\Cox_c(v,J)}(y)\mbox{ if and only if }\pidown^c(wx)=\pidown^c(wy).\]

Let $s=s_1$.
If $w\ge s$ then also $v\ge s$.
The element $sv$ is $scs$-sortable and $w^{-1}sw$ is not in $J$, so that the set $\set{svrv^{-1}s:r\in J}$ is a subset of $\cov(sv)$.
Thus $sv$ and $sw=svw_0(J)$ are respectively the sortable element above and the element below some face of $\F_{scs}$.
Furthermore $\Cox_{scs}(sv,J)=\Cox_c(v,J)$, so the theorem follows in this case by induction on $\ell(v)$ and the recursive definition of $\pidown^c$.

If $v\not\ge s$ then $v\in W_{\br{s}}$, so $w\in W_{\br{s}}$ and $J\subseteq\br{s}$.
The theorem follows in this case by induction on the rank of $W$.

If $v\ge s$ but $w\not\ge s$ then $s\in\cov(v)$ so that $k>1$.
Furthermore, $v$ is also $scs$-sortable by Proposition~\ref{Cc initial}(i), Proposition~\ref{sort para easy} and Theorem~\ref{meets and joins}.
Thus $(v,J)$ encodes a face in $\F_{scs}$ and $w$ is the element below that face. 
Let $\cov(v)=\{ s, t_2, t_3, \ldots, t_l \}$. By Lemma~\ref{Cc together}, all of the $t_i$ lie in $W_{\br{s}}$. 
Thus, for $2 \leq i,j \leq l$, we have $\omega_c(\beta_{t_i}, \beta_{t_j}) = \omega_{sc}(\beta_{t_i}, \beta_{t_j})= \omega_{scs}(\beta_{t_i}, \beta_{t_j})$. 
Also, for $2 \leq i \leq l$, we have $\omega_c(\alpha_s, \beta_{t_i}) \geq 0 $ and $\omega_{scs}(\alpha_s, \beta_{t_i}) \leq 0$.
So, setting $r=w^{-1}sw$, we know that $r$ is initial in $\Cox_c(v,J)$ and $\Cox_{scs}(v,J)=r\Cox_c(v,J)r$.

If $x\ge r$ and $y\ge r$, then $wx\ge s$ and $wy\ge s$, so $\pidown^c(wx)=\pidown^c(wy)$ if and only if $\pidown^{scs}(swx)=\pidown^{scs}(swy)$.
Since $swx=wrx$ and $swy=wry$, the latter equation is $\pidown^{scs}(wrx)=\pidown^{scs}(wry)$ which by induction on $k$ holds if and only if $\pidown^{\Cox_{scs}(v,J)}(rx)=\pidown^{\Cox_{scs}(v,J)}(ry)$, or equivalently $\pidown^{\Cox_c(v,J)}(x)=\pidown^{\Cox_c(v,J)}(y)$.

If exactly one of $x$ and $y$ is above $r$ then exactly one of $wx$ and $wy$ is above $s$, and Proposition~\ref{AboveS} implies that 
$\pidown^{\Cox_c(v,J)}(x)\neq\pidown^{\Cox_c(v,J)}(y)$ and $\pidown^c(wx)\neq\pidown^c(wy)$.

If $x\not\ge r$ and $y\not\ge r$ then $wx\not\ge s$ and $wy\not\ge s$ so Lemma~\ref{pidown join lemma} implies that $\pidown^c(wx)=\pidown^c(wy)$ if and only if $\pidown^{scs}(s\join(wx))=\pidown^{scs}(s\join(wy))$.
We have $s\join wx=w(r\join x)$ and $s\join wy=w(r\join y)$.
Thus by induction on $k$, $\pidown^{scs}(s\join(wx))=\pidown^{scs}(s\join(wy))$ if and only if $\pidown^{r\Cox_c(v,J)r}(r\join x)=\pidown^{r\Cox_c(v,J)r}(r\join y)$.
By Lemma~\ref{pidown join lemma} again, $\pidown^{r\Cox_c(v,J)r}(r\join x)=\pidown^{r\Cox_c(v,J)r}(r\join y)$ if and only if $\pidown^{\Cox_c(v,J)}(x)=\pidown^{\Cox_c(v,J)}(y)$.
\end{proof}

\begin{remark} 
It is possible for a face $F$ of the Cambrian fan to lie in the boundary of the Tits cone. 
It seems plausible that Theorem~\ref{stars} could be extended to describe this case.
The primary challenge would be to generalize the definition of $\Cox_c$.
In addition, the parabolic subgroup $W_J$ would have to be replaced by a generalized parabolic subgroup, so developing this theory would require developing a theory of 
generalized parabolic subgroups of rank other than $2$.
This has some subtleties, as explained in Remark~\ref{gen high rank}.
%It is possible for a face $F$ of the Cambrian fan to lie in the boundary of the Tits cone. 
%It seems plausible that Theorem~\ref{stars} could be extended to describe this case.
%The role of the subgroup $W_J$ would have to be replaced by a generalized parabolic subgroup, so developing this theory would require developing a theory of generalized parabolic subgroups of rank other than $2$.
%This has some subtleties, see Remark~\ref{gen high rank}.
%In addition, it is not clear how to define $\Cox_c$, as the description in terms of sorting words is inapplicable and there may be no ordering of the canonical generators such that $\omega_c(\beta_i, \beta_j) \geq 0$ for $i > j$.
%We have done some computations that suggest both these possibilities can be overcome, but we do not have enough data to formulate a conjecture.
\end{remark}

\section{Rank-three examples}\label{rank3sec}
An irreducible Coxeter group $W$ of rank three acts as a reflection group on a two-dimensional manifold of constant curvature.  
Thus the rank-three case is well-suited to illustrations on the page.
In this section, we illustrate many of the results of this paper using examples of rank three.
We give examples of sortable elements and Cambrian fans in each of the three possible geometries: spherical, Euclidean and hyperbolic.
As an aid to the reader scanning this section, the major concepts being illustrated at each point are written in \textbf{boldface}.

\subsection{Type~$\widetilde{G}_2$}\label{aff sec}
The Coxeter groups with Euclidean geometry are the affine Coxeter groups (as well as some degenerate examples, such as the product of a finite group and an affine group).
See e.g.\ \cite[Sections~4.7, 6.5]{Humphreys}.
In rank three, the affine Coxeter groups are $\widetilde{A}_2$, $\widetilde{B}_2$ and~$\widetilde{G}_2$.
The Tits cone of an affine Coxeter group is an open halfspace plus one additional point (the origin) on the boundary.
The action of the group on the Tits cone restricts to an action on a hyperplane parallel to the boundary of the halfspace. 
In this hyperplane, the group acts by Euclidean transformations.
In particular, an affine Coxeter group of rank three is generated by reflections through certain lines in the plane.

The group of type~$\widetilde{G}_2$ has simple generators $\set{r,s,t}$ with $m(r,s)=6$, $m(s,t)=3$ and $m(r,t)=2$.
This is the symmetry group of the regular tiling of the plane by regular hexagons.
Figures~\ref{bij1} and~\ref{bij2} show part of two Cambrian fans for $W$ of type~$\widetilde{G}_2$.
In Figure~\ref{bij1}, the Coxeter element $c$ is $tsr$, while $c=srt$ in Figure~\ref{bij2}.
The fine lines show the underlying Coxeter fan and the bold lines indicate the Cambrian cones. 
The shaded triangles mark cones in the Coxeter fan corresponding to $c$-sortable elements.
To the extent practical, the $c$-sortable elements are labeled with their $c$-sorting words.

\begin{figure}[p]
\psfrag{1}[cc][cc]{\LARGE$e$}
\psfrag{t}[cc][cc]{\LARGE$r$}
\psfrag{s}[cc][cc]{\LARGE$s$}
\psfrag{st}[cc][cc]{\LARGE$sr$}
\psfrag{sts}[cc][cc]{\LARGE$srs$}
\psfrag{stst}[cc][cc]{\LARGE$srsr$}
\psfrag{r}[cc][cc]{\LARGE$t$}
\psfrag{rt}[cc][cc]{\LARGE$tr$}
\psfrag{rs}[cc][cc]{\LARGE$ts$}
\psfrag{rsr}[cc][cc]{\LARGE$tst$}
\psfrag{rst}[cc][cc]{\LARGE$tsr$}
\psfrag{rstr}[cc][cc]{\LARGE$tsrt$}
\psfrag{rsts}[cc][cc]{\LARGE$tsrs$}
\psfrag{rstst}[cc][cc]{\LARGE$tsrsr$}
\centerline{\scalebox{.65}{\includegraphics{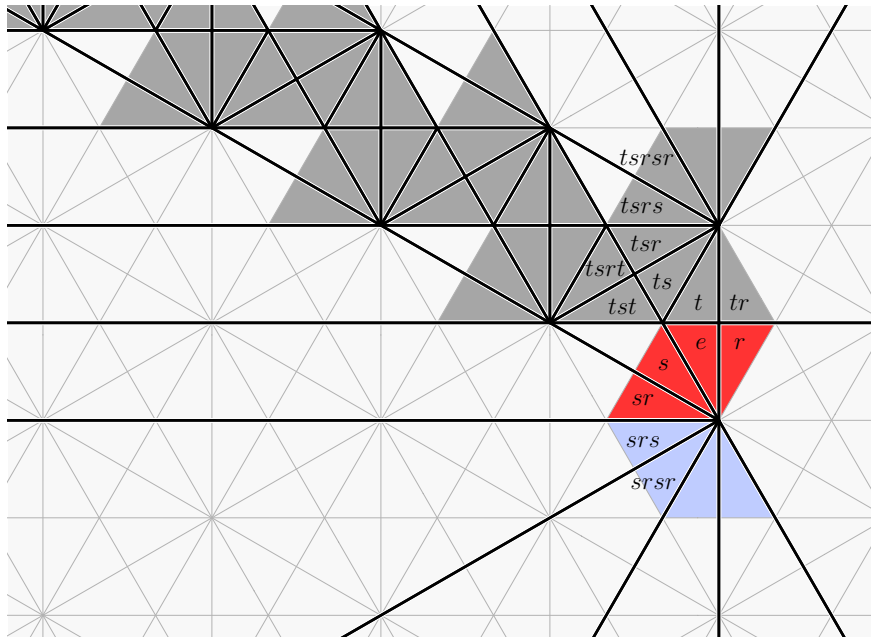}}}
\caption{The Cambrian fan for $W$ of type~$\widetilde{G}_2$ with $c=tsr$.}
\label{bij1}
\vspace{30 pt}
\end{figure}

\begin{figure}[p]
\psfrag{1}[cc][cc]{\LARGE$e$}
\psfrag{t}[cc][cc]{\LARGE$r$}
\psfrag{s}[cc][cc]{\LARGE$s$}
\psfrag{st}[cc][cc]{\LARGE$sr$}
\psfrag{sts}[cc][cc]{\LARGE$srs$}
\psfrag{stst}[cc][cc]{\LARGE$srsr$}
\psfrag{r}[cc][cc]{\LARGE$t$}
\psfrag{tr}[cc][cc]{\LARGE$rt$}
\psfrag{sr}[cc][cc]{\LARGE$st$}
\psfrag{str}[cc][cc]{\LARGE$srt$}
\psfrag{srs}[cc][cc]{\LARGE$sts$}
\psfrag{strststs}[cc][cc]{\LARGE$srtsrsrs$}

\centerline{\scalebox{.65}{\includegraphics{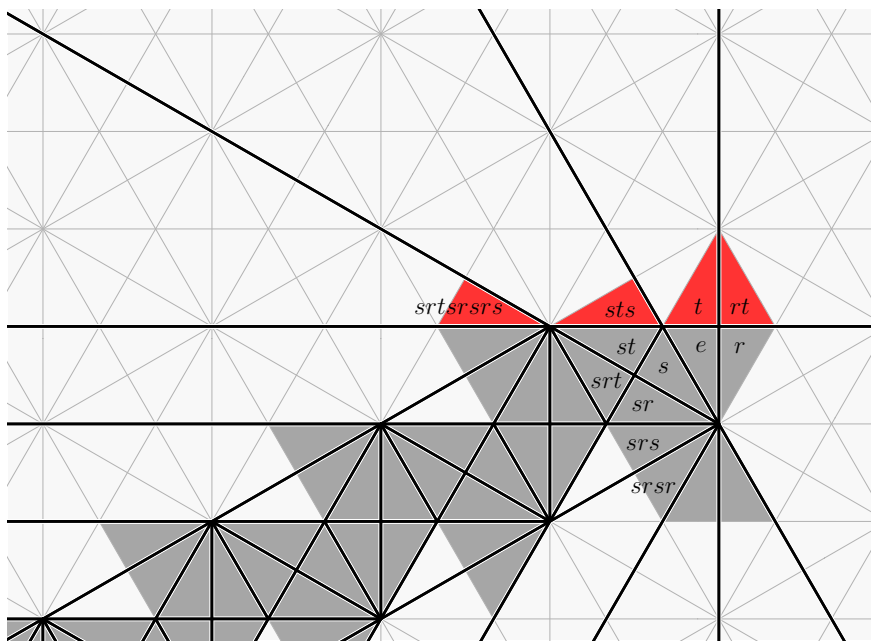}}}
\caption{The Cambrian fan for $W$ of type~$\widetilde{G}_2$ with $c=srt$.}
\label{bij2}
\end{figure}

To begin with, we focus on the shaded triangles in order to visualize results about sortable elements.
The colors of the shaded triangles help to illustrate \textbf{Theorem~\ref{c to scs}}.
Notice that $t$ is initial in $c=tsr$.  
In Figure~\ref{bij1}, the gray shaded triangles are $c$-sortable elements above $t$ in the weak order.
The red (or darker gray) shaded triangles are $c$-sortable elements $v$ not above $t$ such that $t\join v$ exists.
The blue (or lighter gray) shaded triangles are $c$-sortable elements $v$ such that $t\join v$ does not exist.
The bijection of Theorem~\ref{c to scs} maps the gray and red triangles in Figure~\ref{bij1} to the full set of triangles in Figure~\ref{bij2}.
The gray triangles are reflected through the horizontal line representing $H_t$, while each red triangle representing $v$ is mapped to $t\join v$.

One can also see the \textbf{canonical join representations} of $c$-sortable elements in these figures. 
For example, \textbf{Proposition~\ref{csort canon}} is illustrated in Figure~\ref{bij2}, where the $srt$-sortable element $srtsrsrs$ is canonically the join of $srt$-sortable elements $sr$ and~$t$.
This picture also illustrates the proof of \textbf{Theorem~\ref{W can join rep}}:  
The element $sr$ is the unique smallest element below $srtsrsrs$ having the reflection $srs$ as an inversion, and $t$ is the unique smallest element below $srtsrsrs$ having the reflection $t$ as an inversion.
\textbf{Theorem~\ref{exactly one ji}} manifests itself in these pictures as the fact that, for each bold line segment in the picture, there is a unique minimal sortable element $j$ of $W$ that is above the corresponding hyperplane $H$.
This element $j$ is join-irreducible, that is, exactly one of the edges of its triangle $jD$ separates $jD$ from a triangle closer to $D$.
This edge of $jD$ is contained in $H$, and $j$ is the only join-irreducible element such that $jD$ has an edge contained in $H$ that separates $jD$ from $D$.

We now consider the Cambrian cones in the figures.
Theorem~\ref{c to scs} is a statement about sortable elements; \textbf{Proposition~\ref{recursive fan}} is the corresponding statement about \textbf{Cambrian fans}. 
In accordance with Proposition~\ref{recursive fan}(i), since $t$ is initial in $tsr$, the partial fans $(\F_{tsr})_{\ge t}\cap\Tits(W)$ and $(\F_{srt})_{\le t}\cap\Tits(W)$ are related by reflection through $H_t$.  
The partial fan $(\F_{tsr})_{<t}$ agrees with the $sr$-Cambrian fan in $W_{\br{t}}=W_{\set{r,s}}$ in accordance with Proposition~\ref{recursive fan}(ii). 
\textbf{Proposition~\ref{fan para}} specializes to the following statement about these pictures:  For each vertex $x$ of the shaded triangle representing the identity element $e$, the bold line segments emanating from $x$ extend indefinitely, so that no Cambrian cone crosses them.
In particular (\textbf{Proposition~\ref{fan s}}), the line segments bounding the identity element extend indefinitely.

Figures~\ref{bij1} and~\ref{bij2} also illustrate our main lemmas regarding \textbf{forced and unforced skips}: \textbf{Propositions~\ref{walls}}, \textbf{\ref{lower walls}}, \textbf{\ref{Cc final}} and~\textbf{\ref{Cc initial}}. All of these, due to Theorem~\ref{pidown fibers}, can be interpreted as assertions about the Cambrian cone $\boldsymbol{\Cone_c(v)}$. 
Proposition~\ref{lower walls} states that the lower walls (the walls separating the cone from the dominant chamber $D$) of each Cambrian cone $\Cone_c(v)$ correspond to the cover reflections of $v$.
These cover reflections appear as the lower walls of the triangle corresponding to~$v$.
Similarly, with a bit of calculation, we observe that the upper walls are given by unforced skips in the $c$-sorting word. Consider for example the Cambrian cone $\Cone_{srt}(srtsrsrs)$.
Its lower walls are given by its cover reflections $t$ and $srs$.
The letter $t$ is skipped in the word $srtsr\! \cdot \! srs$ (as indicated by the dot in the word).
This skip is unforced, with corresponding reflection $srtsrtrstrs=srsrs$, so the hyperplane $H_{srsrs}$ defines the upper wall of $\Cone_{srt}(srtsrsrs)$. 
In Figure~\ref{bij2}, the hyperplane $H_{srsrs}$ is below, and parallel to, the hyperplane $H_t$, reflecting the fact that these two hyperplanes do not intersect within the Tits cone but rather intersect on the boundary of the Tits cone. Within the Tits cone,   $\Cone_{srt}(srtsrsrs)$ is bounded by $H_t$ and $H_{srs}$ while, in the negation of the Tits cone,   $\Cone_{srt}(srtsrsrs)$ is bounded by $H_{srsrs}$ and $H_{srs}$.
The remainder of the cone $\Cone_{srt}(srtsrsrs)$ is pictured below in Figures~\ref{bij4} and~\ref{bij3}.

The geometric content of Proposition~\ref{Cc final} appears explicitly as \textbf{Proposition~\ref{recursive fan final}}.
This geometry is apparent, for example, in $\Cone_{tsr}(tr)$ in Figure~\ref{bij1}. 
One wall of $\Cone_{tsr}(tr)$ is given by the hyperplane $H_r$ and the other walls are obtained by extending the walls of $\Cone_{ts}((tr)_{\br{r}})=\Cone_{ts}(t)$. 
Similarly, Proposition~\ref{Cc initial} is exemplified by $\Cone_{tsr}(tst)$ in Figure~\ref{bij1}. 
One lower wall of $\Cone_{tsr}(tst)$ is the hyperplane $H_t$. 
The other lower wall is obtained by extending the lower wall of $\Cone_{sr}( (tst)_{\br{t}})=\Cone_{sr}(s)$, while the upper wall is obtained by reflecting the upper wall of $\Cone_{sr}( (tst)_{\br{t}})$ across $H_t$.

We now point out how the content of \textbf{Theorem~\ref{stars}} can be seen in these pictures. 
Consider in Figure~\ref{bij2} the two-dimensional cone $F$ represented in this planar slice by the line segment separating $\Cone_{srt}(sts)$ from $\Cone_{srt}(srtsrsrs)$.
Taking any $x$ in the relative interior of this cone and moving a short distance away from the interior of $D=eD$, we land in $\Cone_{srt}(srtsrsrs)$.
Thus $srtsrsrs$ is the $c$-sortable element above $F$.
The face $F$ of $\Cone_{srt}(srtsrsrs)$ is the intersection of $\Cone_{srt}(srtsrsrs)$ with $H_{srs}$ and $srs$ is a cover reflection $srtsrsrs\!\cdot\!r\!\cdot\!srsrstrs$ of $srtsrsrs$.
Thus $\set{r}$ is the set of simple generators through $F$, so $w_0({\set{r}})=r$ and $srtsrsrs\!\cdot\!r=stsrsrs$ is the element below $F$.
(The cone $stsrsrsD$ is not labeled, but is immediately adjacent to $srtsrsrsD$ through the face $F$.)
The assertion of Theorem~\ref{stars} is that the star of $F$ in $\F_{srt}$ is $stsrsrsP_{\set{r}}^{-1}(\F_r)$.
The fan $\F_r$ consists of the origin and two opposite rays inside the one-dimensional space $V_{\set{r}}^*$, and $P_{\set{r}}^{-1}(\F_r)$ in $V^*$ consists of the hyperplane $H_r$ and the two halfspaces it defines.
Thus $\Star_F(\F_{srt})$ consists of the hyperplane $stsrsrsH_r=H_{srs}$ and the two halfspaces it defines.

As another example, now take $F$ to be the ray defined as the intersection of $\Cone_{srt}(srtsrsrs)$ with $H_{srs}$ and $H_t$.
Note that $t=srtsrsrs\!\cdot\!s\!\cdot\!srsrstrs$ is the other cover reflection of $srtsrsrs$.
Again $srtsrsrs$ is the $c$-sortable element above $F$ and now $\set{r,s}$ is the set of simple generators through $F$, with $w_0({\set{r,s}})=srsrsr$.
The element below $F$ is $srtsrsrs\!\cdot\!srsrsr=st$.
The reflection sequence of $srtsrsrs$ is 
\[s,srs,sts,srtstrs,srtsrstrs,stsrsrsts,tsrst,t\]
and in particular $srs$ occurs before $t$.
Thus $\Cox_{srt}(F)$ is $rs$, rather than $sr$.
The assertion of Theorem~\ref{stars} can be stated as follows in the context of Figure~\ref{bij2}:
Take the lower vertex of the triangle representing $D$ as an origin for the plane and build the $rs$-Cambrian fan around this origin.
(Our Euclidean plane has no natural origin; the plane in which our pictures are drawn does not contain the origin of $V^*$.)
%(Both Figures~\ref{bij1} and~\ref{bij2} show the $sr$-Cambrian fan around this origin, not the $rs$-Cambrian fan.)
Then the rotation in the plane corresponding to $st$ maps the $rs$-Cambrian fan to the star of $F$.

We now offer some additional illustrations of the $srt$-Cambrian fan for $W$ of type~$\widetilde{G}_2$ which illustrate \textbf{how the Cambrian fan extends beyond the Tits cone}.
First, we show the intersection of the $srt$-Cambrian fan with the negative Tits cone $-\Tits(W)$.
To interpret this picture, recall that Figure~\ref{bij2} shows the intersection of the $srt$-Cambrian fan with a plane in $\Tits(W)$.
We are looking at that plane from far above (i.e.\ far into $\Tits(W)$).
Figure~\ref{bij4} shows, from the same vantage point, the intersection of the $srt$-Cambrian fan with a parallel
plane in $-\Tits(W)$.
The fan is shaded gray and its maximal cones are labeled, as are the reflecting hyperplanes bounding these cones.
The antipodal dominant chamber $-D$, located outside the Cambrian fan, is also labeled.

\begin{figure}[p]
\psfrag{1}[cc][cc]{\LARGE$-D$}
\psfrag{b}[cc][cc]{\LARGE$\Cone_{c}(rt)$}
\psfrag{ab}[cc][cc]{\LARGE$\Cone_{c}(srsrsr)$}
\psfrag{aba}[cc][cc]{\LARGE$\Cone_{c}(r)$}
\psfrag{acbcbcb}[cc][cc]{\LARGE$\Cone_{c}(srsrs)$}
\psfrag{cbcbcb}[cc][cc]{\LARGE$\Cone_{c}(t)$}
\psfrag{cbcbc}[cc][cc]{\LARGE$\Cone_{c}(sts)$}
\psfrag{cbcb}[cc][cc]{\LARGE$\Cone_{c}(srtsrsrs)$}

\psfrag{Hr}[cc][cc]{\LARGE$H_t$}
\psfrag{Ht}[cc][cc]{\LARGE$H_r$}
\psfrag{Hs}[cc][cc]{\LARGE$H_s$}
\psfrag{Htst}[cc][cc]{\LARGE$H_{rsr}$}
\psfrag{Hsts}[cc][cc]{\LARGE$H_{srs}$}
\psfrag{Hststs}[cc][cc]{\LARGE$H_{srsrs}$}
\centerline{\scalebox{.65}{\includegraphics{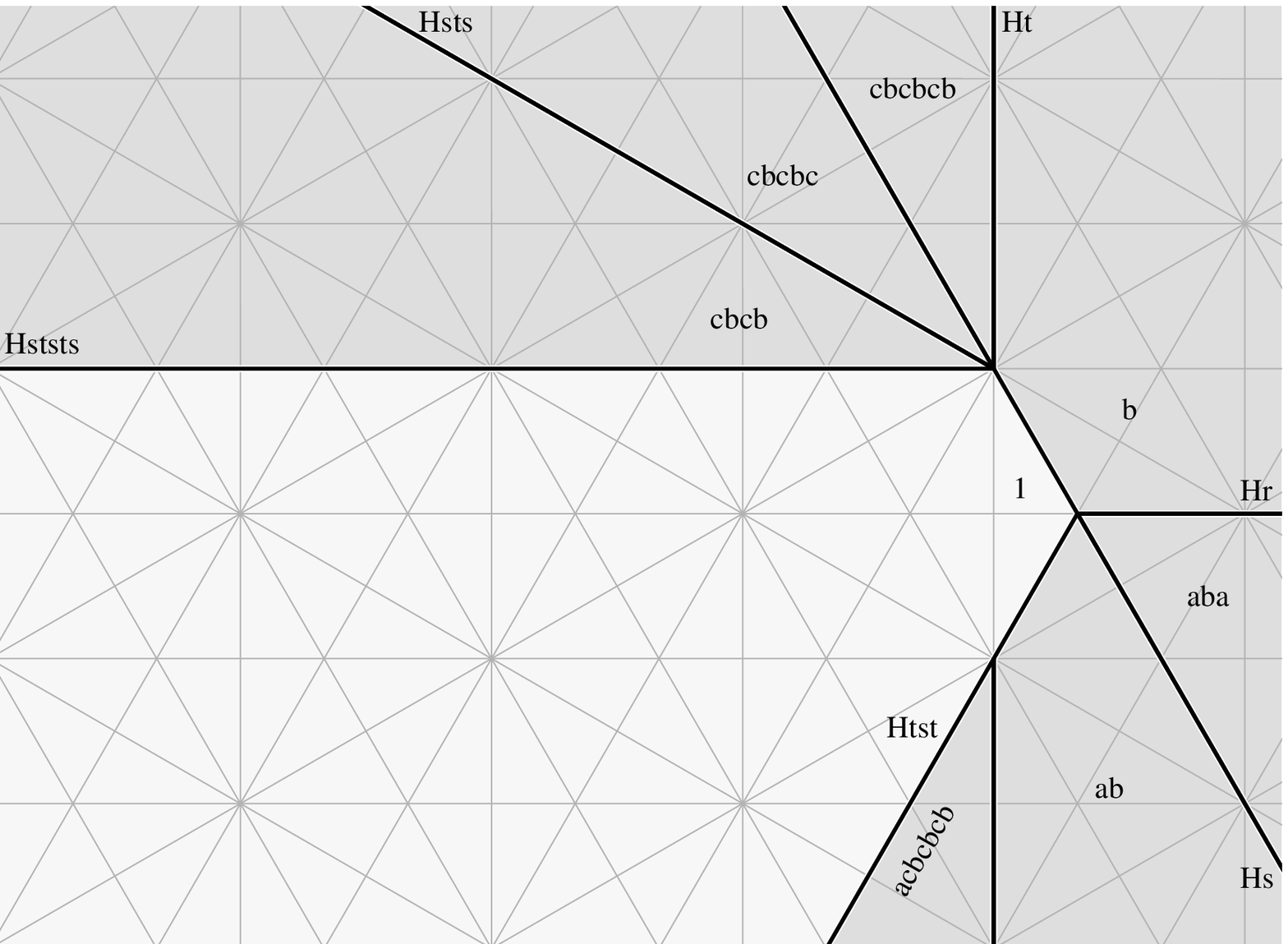}}}
\psfrag{1}[cc][cc]{\LARGE$e$}
\psfrag{c}[cc][cc]{\LARGE$r$}
\psfrag{b}[cc][cc]{\LARGE$s$}
\psfrag{a}[cc][cc]{\LARGE$t$}
\caption{The Cambrian fan as it intersects the negative Tits cone, for $W$ of type~$\widetilde{G}_2$ with $c=srt$.}
\label{bij4}
\vspace{30 pt}
\end{figure}

\begin{figure}[p]
\centerline{\scalebox{.6}{\includegraphics{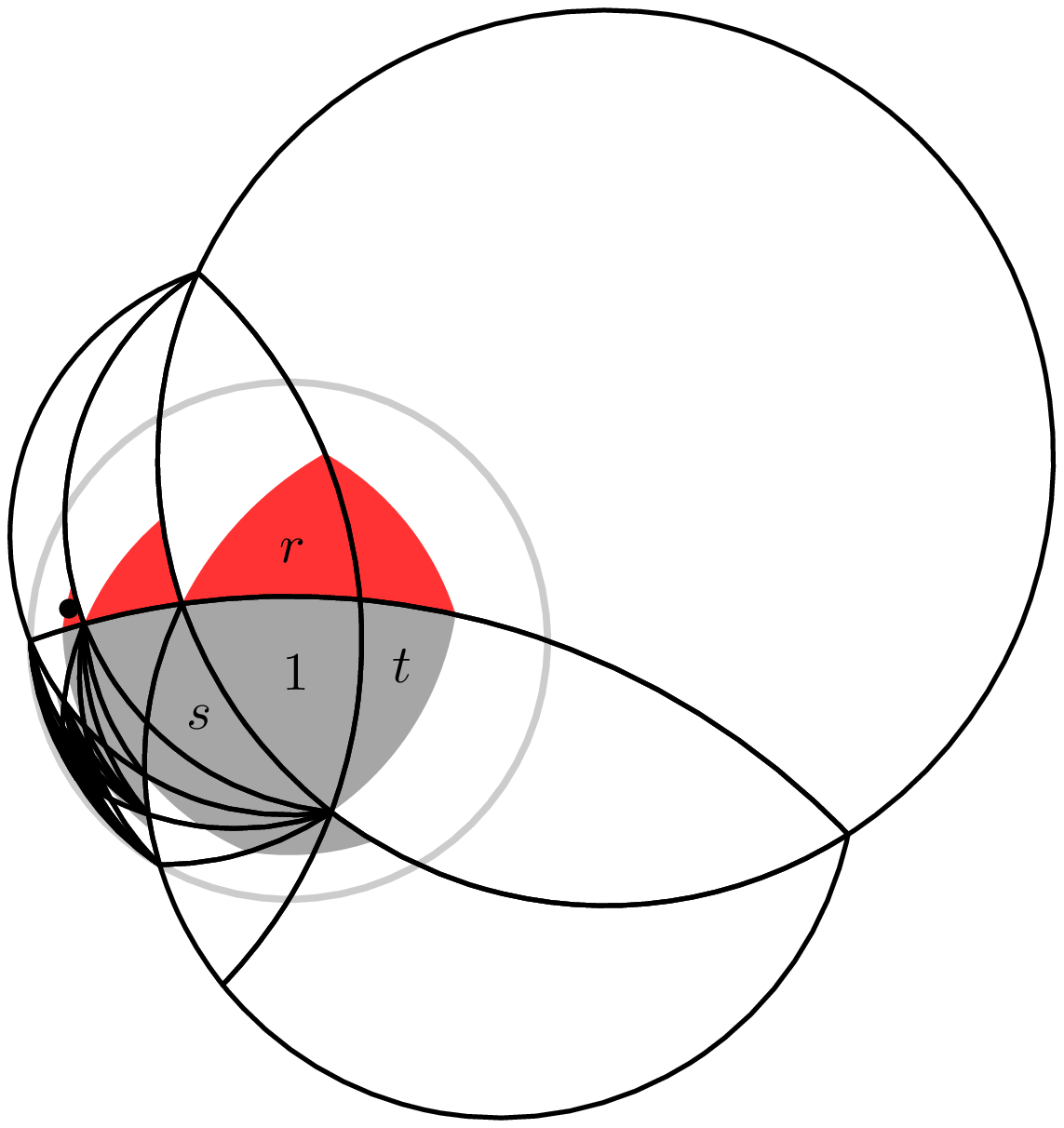}}}
\caption{A stereographic view of the Cambrian fan for $W$ of type~$\widetilde{G}_2$ with $c=srt$.}
\label{bij3}
\end{figure}

To unify these two views, Figure~\ref{bij3} shows the Cambrian fan in stereographic projection.
That is, the intersection of the Cambrian fan with the unit sphere in $V^*$ defines a complex of spherical simplices.
Figure~\ref{bij3} shows this simplicial complex in stereographic projection.
The point at infinity is the image of a point in the antipodal dominant chamber $-D$.
The boundary of the Tits cone projects to the gray circle.
The three walls of $\Cone_{srt}(strsrsrs)$ are now visible. 
(The label $srtsrsrs$ has been replaced by a black dot.)

The form $\omega_c$ is a key ingredient in the proofs presented in this paper.
However, the precise value of $\omega_c$ is not important, while the sign of $\omega_c$ is decisive.
There is an elegant way to visualize \textbf{the sign of} $\boldsymbol{\omega_c}$ for Coxeter groups of rank $3$. 
Every skew symmetric form $\omega$ on $\RR^3$  is of the form $\omega(\alpha,\beta)=\Vol(\alpha \wedge \beta \wedge \zeta)$ for some $\zeta \in \RR^3$, where $\Vol$ is the volume form.  
Then $\omega(\alpha, \beta)$ is positive or negative according to whether the angle from $\alpha$ to $\beta$ circles the line spanned by $\zeta$ in a clockwise or counterclockwise manner.  
For a rank $3$ Coxeter group, let $\zeta_c$ be such that $\omega_c(\alpha,\beta)=\Vol(\alpha\wedge\beta\wedge\zeta_c)$. 
From $\omega_c(\alpha, \beta)=\omega_c(c \alpha, c \beta)$ (and the fact that $c$ acts by a matrix of determinant $-1$ in rank $3$) we deduce that $\zeta_c$ is a $(-1)$-eigenvector of $c$. 
This makes it easy to compute $\zeta_c$. 
Figure~\ref{G2tildezeta} illustrates the example where $W$ is of type~$\widetilde{G}_2$ and $c=tsr$.
Each positive root generates a ray, and the vector~$\zeta_c$ generates a ray.
The plane of the picture is a plane that intersects all of these rays, and each ray is represented by its intersection with the plane.
Since there are infinitely many roots, they are shown as black dots whose size becomes smaller as the simple-root coordinates of the roots increase.
The ray for $\zeta_c$ is shown in red (or gray).
In this planar picture, $\omega_c(\beta, \beta') > 0$ if the angle $\beta \zeta_c \beta'$ (projected to the plane) is clockwise. 
In Section~\ref{hyp sec}, a similar example will be taken a step further to illustrate Proposition~\ref{InversionOrdering}.

\begin{figure}
\psfrag{r}[rc][cc]{\Huge$\alpha_t$}
\psfrag{s}[lc][cc]{\Huge$\alpha_s$}
\psfrag{t}[ct][cc]{\Huge$\alpha_r$}
\centerline{\scalebox{.6}{\includegraphics{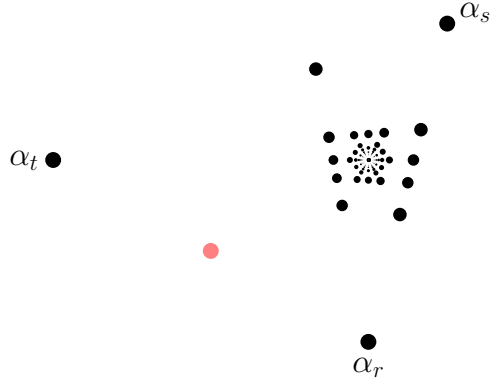}}}
\caption{The~$\widetilde{G}_2$ root system with $\zeta_{tsr}$.}
\label{G2tildezeta}
\end{figure}

\subsection{Type $B_3$}\label{B3 sec}
When $W$ is finite, the Tits cone $\Tits(W)$ is all of $V^*$.
Thus the Cambrian fan also occupies all of $V^*$.
In particular, intersecting with the unit sphere, the Cambrian fan defines a simplicial sphere.
Figure~\ref{B3pic} shows a Cambrian fan (as a simplicial sphere in stereographic projection) in the case where $W$ is of type $B_3$.
Here $W$ is generated by $\set{r,s,t}$ with $m(r,s)=4$, $m(s,t)=3$ and $m(r,t)=2$, and the figure shows the case where $c=rst$.
The stereographic projection, despite being difficult to interpret at first, has the advantage of showing essentially the entire fan in one picture.
The three largest circles represent the simple generators, with~$t$ being the largest, followed by $r$, then $s$.
The dominant chamber $D$, representing the identity element, is the triangle inside all three of these circles.
As in previous figures, the sortable elements are represented by shaded regions.
The unbounded region represents the longest element $w_0$.

\begin{figure}[ht]
\centerline{\scalebox{.7}{\includegraphics{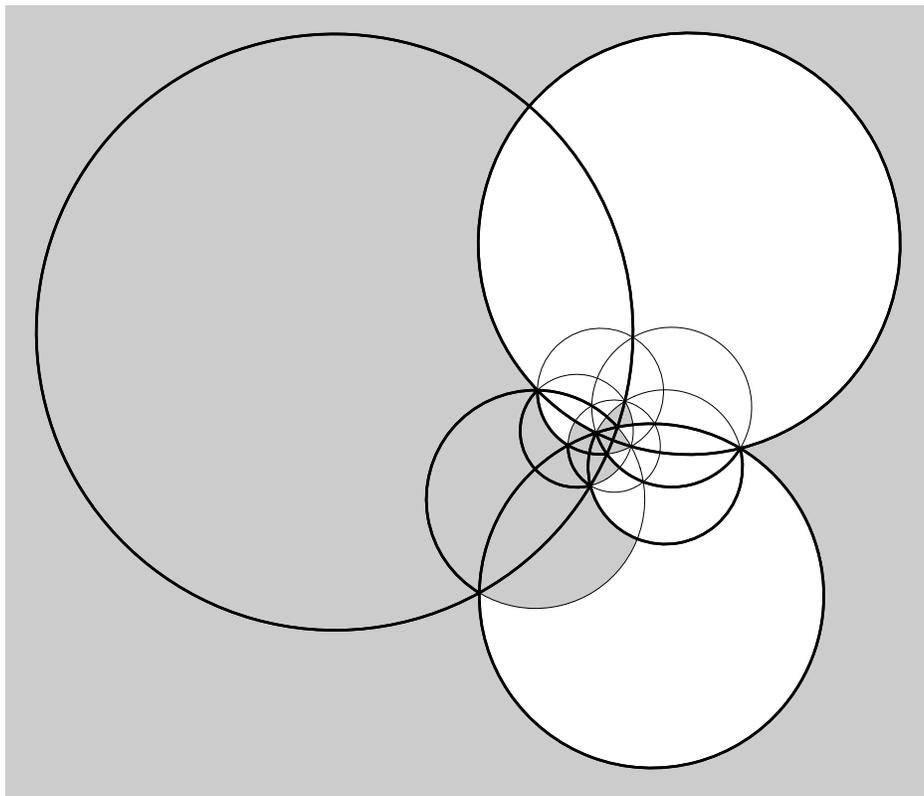}}}
\caption{A stereographic view of the Cambrian fan for $W$ of type $B_2$ with $c=rst$.}
\label{B3pic}
\vspace{20 pt}
\end{figure}

\begin{figure}[ht]
\psfrag{r}[rc][cc]{\Huge$\alpha_r$}
\psfrag{s}[rc][cc]{\Huge$\alpha_s$}
\psfrag{t}[lc][cc]{\Huge$\alpha_t$}
\centerline{\scalebox{.5}{\includegraphics{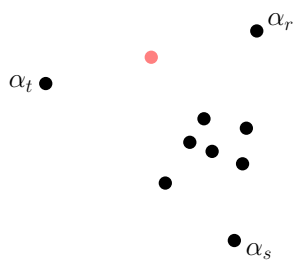}}}
\caption{The $B_3$ root system, with $\zeta_{rst}$.}
\label{B3Zeta}
\end{figure}

The Cambrian fan shown in Figure~\ref{B3pic} is isomorphic, as a simplicial complex, to the dual of the generalized associahedron of type $B_3$.
(See~\cite{ga}.)
More generally, as proved in~\cite{camb_fan}, for any finite Coxeter group and any Coxeter element $c$, the $c$-Cambrian fan is isomorphic to the normal fan of the generalized associahedron of the appropriate type.

Figure~\ref{B3Zeta} illustrates $\omega_{rst}$ for $W$ of type $B_3$ and $c=rst$.
As in Figure~\ref{G2tildezeta}, each represents a root, except for the red dot, which represents $\zeta_c$.
Again $\omega_c(\beta, \beta') > 0$ if the angle $\beta \zeta_c \beta'$ is clockwise. 

\begin{figure}
\psfrag{1}[cc][cc]{\Large$e$}
\psfrag{b}[cc][cc]{\Large$s$}
\psfrag{bc}[cc][cc]{\Large$st$}
\psfrag{bcb}[cc][cc]{\Large$sts$}
\psfrag{bcbc}[cc][cc]{\Large$stst$}
\psfrag{c}[cc][cc]{\Large$t$}
\psfrag{a}[cc][cc]{\Large$r$}
\psfrag{ac}[cc][cc]{\Large$rt$}
\psfrag{ab}[cc][cc]{\Large$rs$}
\psfrag{abc}[cc][cc]{\Large$rst$}
\psfrag{abca}[cc][cc]{\Large$rstr$}
\psfrag{abcab}[cc][cc]{\Large$rstrs$}
\psfrag{aba}[cc][cc]{\Large$rsr$}
\psfrag{abab}[cc][cc]{\Large$rsrs$}
\psfrag{ababa}[cc][cc]{\Large$rsrsr$}
\psfrag{abcb}[cc][cc]{\Large$rsts$}
\psfrag{abcbc}[cc][cc]{\Large$rstst$}
\centerline{\scalebox{.8}{\includegraphics{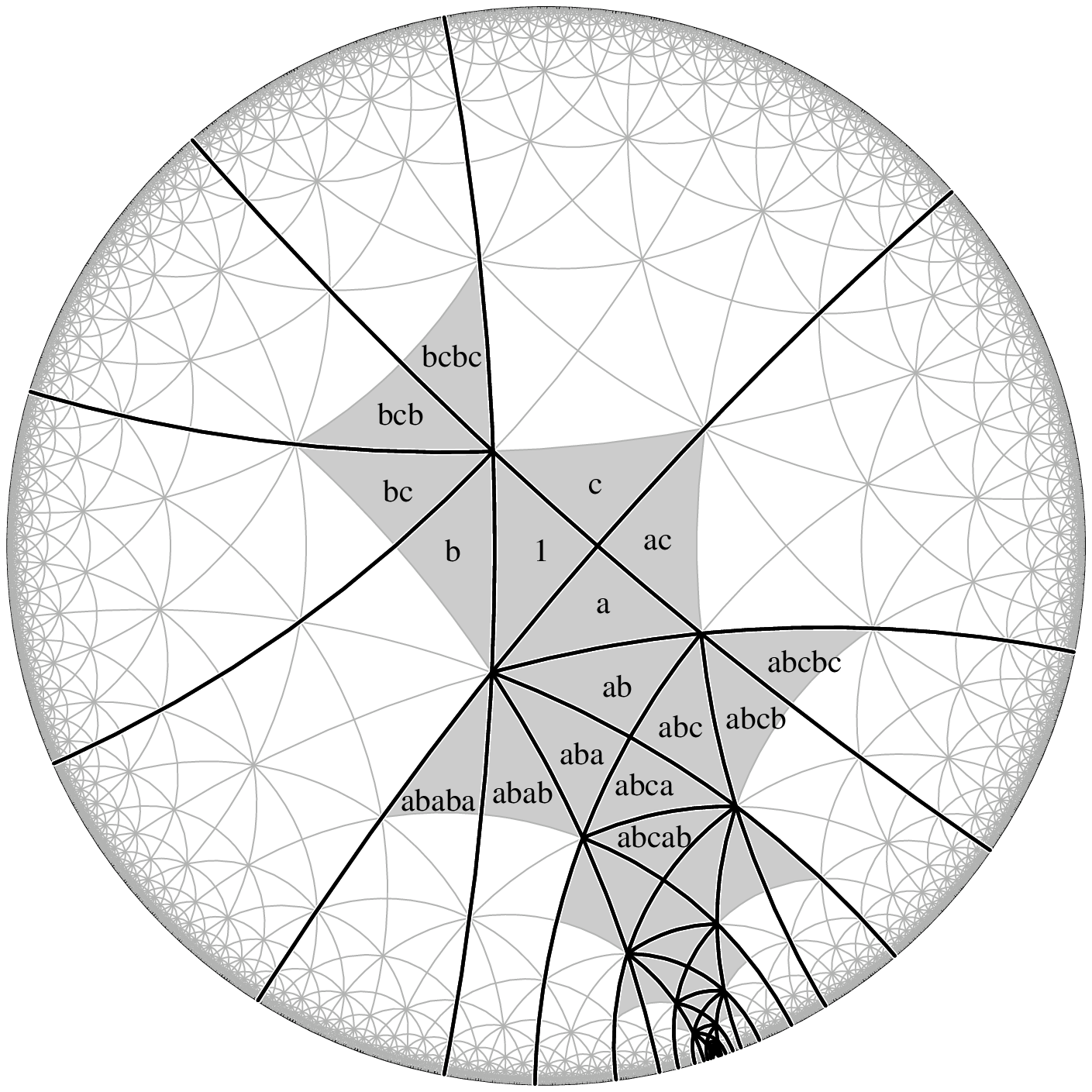}}}
\caption{A Cambrian fan in the Poincar\'e disk.}
\label{542pic}
\end{figure}

\subsection{A hyperbolic example}\label{hyp sec}
Whenever a rank-three Coxeter group is neither finite nor affine, it is always possible to choose a reflection representation of $W$ such that $\Tits(W)$ is a circular cone.
The reflecting hyperplanes, intersected with one sheet of an appropriate hyperboloid, are lines in the hyperboloid model of the hyperbolic plane.
These lines can further be projected to a disk so as to become lines in the Poincar\'{e} model of the hyperbolic plane.
In particular, the intersection of the Cambrian fan with $\Tits(W)$ defines a decomposition of the Poincar\'{e} disk into convex regions.
Figure~\ref{542pic} shows this decomposition in the case where $W$ is generated by $\set{r,s,t}$ with $m(r,s)=5$, $m(s,t)=4$ and $m(r,t)=2$.
The Coxeter element is $c=rst$.

As in the affine example of~$\widetilde{G}_2$, the Cambrian fan in this example extends beyond the Tits cone.  
Figure~\ref{542pic2} shows the stereographic projection of the Cambrian fan.
(Some parts of $-D$ and nearby regions have been cropped to make the image fit on the page.)
A representative group of reflecting hyperplanes (also extending outside $\Tits(W)$) is also shown.
These reflecting hyperplanes also serve to define the boundaries of $\Tits(W)$ and $-\Tits(W)$:  
The area outside of $\Tits(W)$ and $-\Tits(W)$ is dense with reflecting hyperplanes.
Figure~\ref{542pic3} is a magnified view of the neighborhood of $\Tits(W)$.

\begin{figure}[p]
\centerline{\scalebox{.57}{\includegraphics{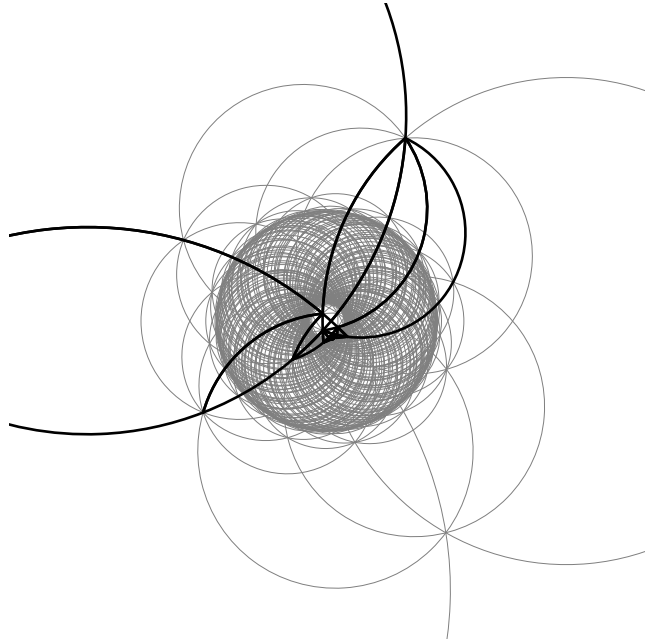}}}
\caption{A stereographic view of a hyperbolic Cambrian fan.}
\label{542pic2}
\vspace{20 pt}
\end{figure}

\begin{figure}[p]
\psfrag{1}[cc][cc]{\huge$e$}
\psfrag{b}[cc][cc]{\huge$s$}
\psfrag{bc}[cc][cc]{\huge$st$}
\psfrag{bcb}[cc][cc]{\huge$sts$}
\psfrag{bcbc}[cc][cc]{\huge$stst$}
\psfrag{c}[cc][cc]{\huge$t$}
\psfrag{a}[cc][cc]{\huge$r$}
\psfrag{ac}[cc][cc]{\huge$rt$}
\psfrag{ab}[cc][cc]{\huge$rs$}
\psfrag{abc}[cc][cc]{\huge$rst$}
\psfrag{abca}[cc][cc]{\huge$rstr$}
\psfrag{aba}[cc][cc]{\huge$rsr$}
\psfrag{abab}[cc][cc]{\huge$rsrs$}
\psfrag{ababa}[cc][cc]{\huge$rsrsr$}
\psfrag{abcb}[cc][cc]{\huge$rsts$}
\psfrag{abcbc}[cc][cc]{\huge$rstst$}
\centerline{\scalebox{.57}{\includegraphics{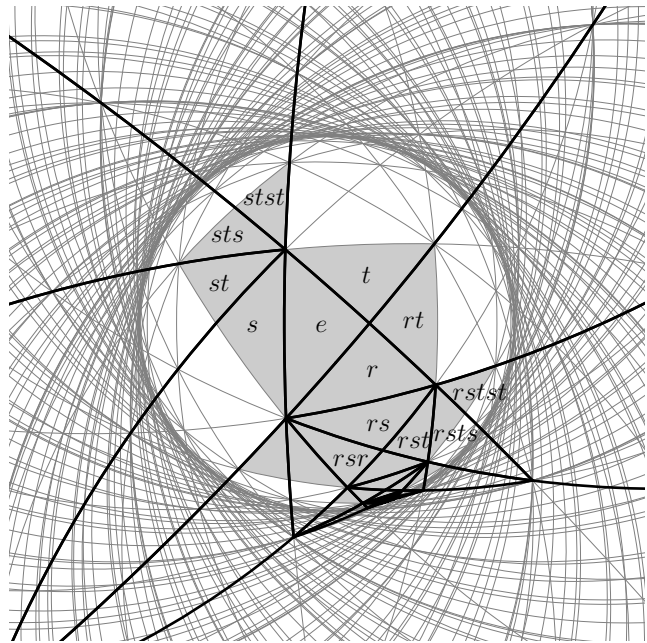}}}
\caption{A magnified view of the fan from Figure~\ref{542pic2}.}
\label{542pic3}
\end{figure}

Figure~\ref{542Zeta} illustrates the root system for this hyperbolic Coxeter group.
The infinitely many roots are projected to the page in a manner similar to the~$\widetilde{G}_2$ example in Section~\ref{aff sec}, except that certain roots are represented by circled numbers, for reasons we will explain below.
The small circle containing no dots is the dual cone in $V$ to $\Tits(W)\subset V^*$.
Also shown is $\zeta_{rst}$.
Once again, $\omega_{rst}(\beta, \beta') > 0$ if the angle $\beta \zeta_c \beta'$ is clockwise. 

\begin{figure}
\psfrag{r}[cc][cc]{\huge $\alpha_r$}
\psfrag{s}[cc][cc]{\huge $\alpha_s$}
\psfrag{t}[cc][cc]{\huge $\alpha_t$}
\psfrag{1}[cc][cc]{\LARGE$1$}
\psfrag{2}[cc][cc]{\LARGE$2$}
\psfrag{3}[cc][cc]{\LARGE$3$}
\psfrag{4}[cc][cc]{\LARGE$4$}
\psfrag{5}[cc][cc]{\LARGE$5$}
\psfrag{6}[cc][cc]{\LARGE$6$}
\psfrag{7}[cc][cc]{\LARGE$7$}
\centerline{\scalebox{.65}{\includegraphics{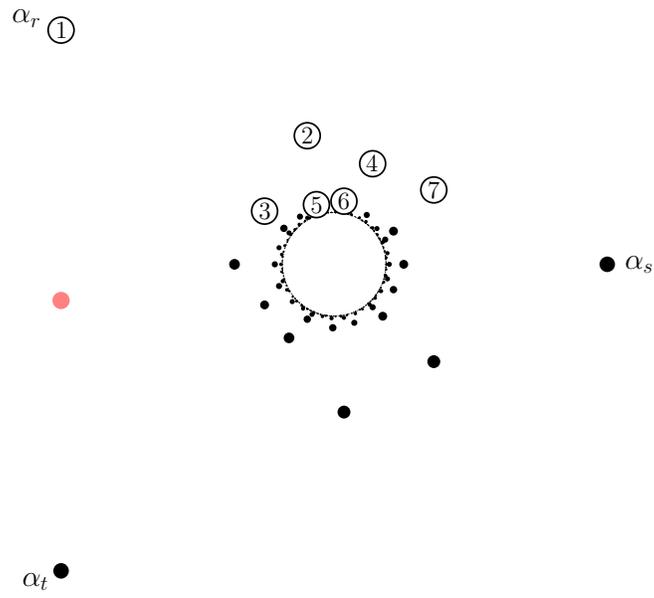}}}
\caption{The reflection sequence of an $rst$-sorting word.}
\label{542Zeta}
\vspace{20 pt}
\end{figure}

\begin{figure}
\psfrag{0}[cc][cc]{\LARGE$0$}
\psfrag{1}[cc][cc]{\LARGE$1$}
\psfrag{2}[cc][cc]{\LARGE$2$}
\psfrag{3}[cc][cc]{\LARGE$3$}
\psfrag{4}[cc][cc]{\LARGE$4$}
\psfrag{5}[cc][cc]{\LARGE$5$}
\psfrag{6}[cc][cc]{\LARGE$6$}
\psfrag{7}[cc][cc]{\LARGE$7$}
\centerline{\scalebox{.65}{\includegraphics{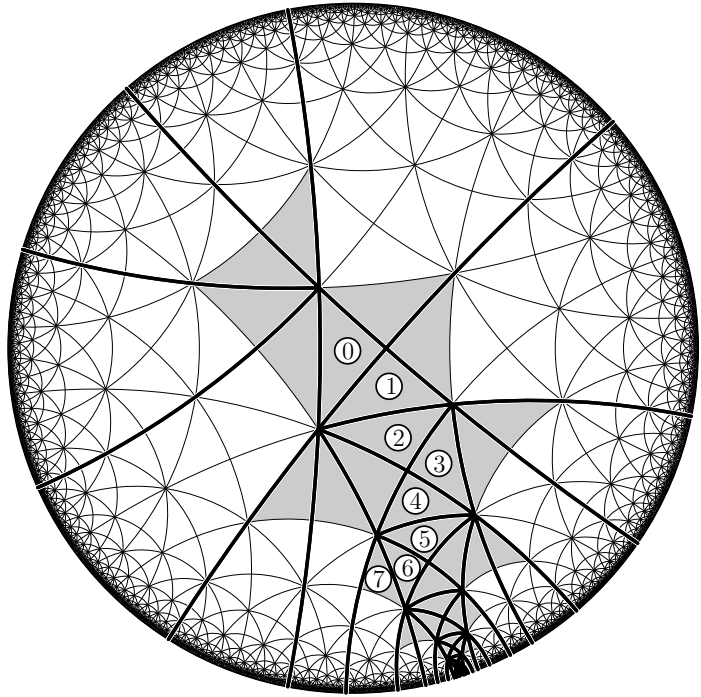}}}
\caption{An $rst$-sorting word.}
\label{word fig}
\end{figure}

The example of Figure~\ref{542Zeta} can be taken further to provide an illustration of \textbf{Proposition~\ref{InversionOrdering}}.
Figure~\ref{word fig} is a representation of the $rst$-sorting word $rstrsts$ which represents an $rst$-sortable element.
A word in the generators $S$ of $W$ describes a walk through the maximal cones in $\Tits(W)$, starting at the cone $D$ representing the identity.
In Figure~\ref{word fig}, the cones visited on this walk are numbered, starting with $0$ on the identity element and ending with $7$ on the shaded triangle representing $rstrsts$.
Each number $i$ labels the element given by the first $i$ letters of the word $rstrsts$.
Each of these elements is $rst$-sortable.

Recall that a reduced word $a_1\cdots a_k$ in the generators $S$ defines a reflection sequence $t_1,\ldots,t_k$ with $t_i=a_1\cdots a_i\cdots a_1$.
The circled numbers in Figure~\ref{542Zeta} represent the reflection sequence associated to $rstrsts$ (as a sequence of positive roots).
For example, the root labeled $1$ is associated to the reflection $r$, the root labeled $2$ is associated to the reflection $rsr$, and so forth.
Proposition~\ref{InversionOrdering} asserts that this reflection sequence must have $\omega_{c}(\beta_{t_i}, \beta_{t_j}) \geq 0$ for all $i\le j$ with strict inequality holding unless $t_i$ and $t_j$ commute.
Indeed, the labeling of the roots in Figure~\ref{542Zeta} is consistent with the order in which these roots are met by a clockwise line through $\zeta_{rst}$, initially vertical.
The roots marked $3$ and $4$ are collinear with $\zeta_{rst}$, reflecting the fact that the reflections $t_3$ and $t_4$ in the reflection sequence commute.  
The word $rsrtsts$, obtained from $rstrsts$ by transposing the adjacent commuting letters $r$ and $t$, also satisfies condition (i) of Proposition~\ref{InversionOrdering}.

The circled numbers are precisely the roots corresponding to inversions of  $rstrsts$. By viewing them as a set, while ignoring the ordering, we can see that  $rstrsts$  is $\boldsymbol{rst}$-\textbf{aligned}.

\addtocontents{toc}{\mbox{ }}

\section*{Acknowledgments}

We are grateful to Matthew Dyer, John Stembridge and Hugh Thomas for helpful conversations.
We are also grateful to anonymous referees for helpful comments.

\end{document}